\documentclass[a4paper]{amsart}
\usepackage{amssymb} 

\usepackage[utf8]{inputenc}
\usepackage[T1]{fontenc} 
\usepackage{lmodern}
\usepackage{microtype}
\usepackage{mathtools}
\mathtoolsset{centercolon} 
\usepackage[normalem]{ulem}	

\usepackage[shortlabels]{enumitem}
  \setlist{nosep} 

\usepackage[hidelinks,pagebackref]{hyperref} 
   \hypersetup{final} 

\usepackage{makecell}
\usepackage{booktabs} 

\usepackage{pdflscape} 
\usepackage{afterpage} 
\usepackage{capt-of}

\newtheorem{theorem}{Theorem}[section]
\newtheorem{lemma}[theorem]{Lemma}
\newtheorem{corollary}[theorem]{Corollary} 
\newtheorem{proposition}[theorem]{Proposition} 
\newtheorem{conjecture}{Conjecture}[] 

\theoremstyle{definition}
\newtheorem{definition}[theorem]{Definition}
\newtheorem{example}[theorem]{Example}

\theoremstyle{remark}
\newtheorem{remark}[theorem]{Remark}

\numberwithin{equation}{section}

    \newcommand*{\RR}{\mathbb{R}}
    \newcommand*{\NN}{\mathbb{N}}

    \newcommand*{\EE}{\mathbb{E}}
    \newcommand*{\PP}{\mathbb{P}}

   \newcommand*{\sgn}{\operatorname{sgn}} 

    \newcommand*{\eqdistr}{\overset{d}{\sim}} 

    \newcommand*{\hadprod}{\mathbin{\circ}} 
    \newcommand*{\conv}{\operatorname{conv}} 
    \newcommand*{\Ext}{\operatorname{Ext}} 

    \newcommand*{\ind}{\mathbf{1}} 

       \newcommand*{\LZeroSlepian}{L_{K}} 
       \newcommand*{\LZeroSlepianValue}{\ln(n)^{1/p^*}\ln(m)^{1/q}} 

     \newcommand*{\prho}{\rho} 
     
     \newcommand*{\rearr}{\downarrow} 

\usepackage{tikz} 
\usetikzlibrary{patterns}  

\colorlet{ColorGHLP}{black}  
\colorlet{ColorRadek}{black} 
\colorlet{ColorEasy}{black} 
\colorlet{ColorGHLPLight}{ColorGHLP} 
\colorlet{ColorRadekLight}{ColorRadek}
\colorlet{ColorEasyLight}{ColorEasy}

\pgfdeclarepatternformonly{south west lines}{\pgfqpoint{-0pt}{-0pt}}{\pgfqpoint{3pt}{3pt}}{\pgfqpoint{3pt}{3pt}}{
	\pgfsetlinewidth{0.4pt}
	\pgfpathmoveto{\pgfqpoint{0pt}{0pt}}
	\pgfpathlineto{\pgfqpoint{3pt}{3pt}}
	\pgfpathmoveto{\pgfqpoint{2.8pt}{-.2pt}}
	\pgfpathlineto{\pgfqpoint{3.2pt}{.2pt}}
	\pgfpathmoveto{\pgfqpoint{-.2pt}{2.8pt}}
	\pgfpathlineto{\pgfqpoint{.2pt}{3.2pt}}
	\pgfusepath{stroke}}

\pgfdeclarepatternformonly{south east lines}{\pgfqpoint{-0pt}{-0pt}}{\pgfqpoint{3pt}{3pt}}{\pgfqpoint{3pt}{3pt}}{
	\pgfsetlinewidth{0.4pt}
	\pgfpathmoveto{\pgfqpoint{0pt}{3pt}}
	\pgfpathlineto{\pgfqpoint{3pt}{0pt}}
	\pgfpathmoveto{\pgfqpoint{.2pt}{-.2pt}}
	\pgfpathlineto{\pgfqpoint{-.2pt}{.2pt}}
	\pgfpathmoveto{\pgfqpoint{3.2pt}{2.8pt}}
	\pgfpathlineto{\pgfqpoint{2.8pt}{3.2pt}}
	\pgfusepath{stroke}}

\pgfdeclarepatternformonly{loose south west lines}{\pgfqpoint{-0pt}{-0pt}}{\pgfqpoint{10pt}{10pt}}{\pgfqpoint{10pt}{10pt}}{
	\pgfsetlinewidth{0.4pt}
	\pgfpathmoveto{\pgfqpoint{0pt}{0pt}}
	\pgfpathlineto{\pgfqpoint{10pt}{10pt}}
	\pgfpathmoveto{\pgfqpoint{9.8pt}{-.2pt}}
	\pgfpathlineto{\pgfqpoint{10.2pt}{.2pt}}
	\pgfpathmoveto{\pgfqpoint{-.2pt}{9.8pt}}
	\pgfpathlineto{\pgfqpoint{.2pt}{10.2pt}}
	\pgfusepath{stroke}}

\pgfdeclarepatternformonly{loose south east lines}{\pgfqpoint{-0pt}{-0pt}}{\pgfqpoint{10pt}{10pt}}{\pgfqpoint{10pt}{10pt}}{
	\pgfsetlinewidth{0.4pt}
	\pgfpathmoveto{\pgfqpoint{0pt}{10pt}}
	\pgfpathlineto{\pgfqpoint{10pt}{0pt}}
	\pgfpathmoveto{\pgfqpoint{.2pt}{-.2pt}}
	\pgfpathlineto{\pgfqpoint{-.2pt}{.2pt}}
	\pgfpathmoveto{\pgfqpoint{10.2pt}{9.8pt}}
	\pgfpathlineto{\pgfqpoint{9.8pt}{10.2pt}}
	\pgfusepath{stroke}}

\tikzset{
	hatch distance/.store in=\hatchdistanceprim,
	hatch distance=7.3pt, 
	hatch thickness/.store in=\hatchthicknessprim,
	hatch thickness=0.8pt
}
\makeatletter
\pgfdeclarepatternformonly[\hatchdistanceprim,\hatchthicknessprim]{MyLooseHor}
{\pgfqpoint{0pt}{0pt}}
{\pgfqpoint{\hatchdistanceprim}{\hatchdistanceprim}}
{\pgfpoint{\hatchdistanceprim-1pt}{\hatchdistanceprim-1pt}}%
{
	\pgfsetcolor{\tikz@pattern@color}
	\pgfsetlinewidth{\hatchthicknessprim}
	\pgfpathmoveto{\pgfqpoint{0pt}{0pt}}
	\pgfpathlineto{\pgfqpoint{\hatchdistanceprim}{0pt}}
	\pgfusepath{stroke}
}
\pgfdeclarepatternformonly[\hatchdistanceprim,\hatchthicknessprim]{MyLooseVert}
{\pgfqpoint{0pt}{0pt}}
{\pgfqpoint{\hatchdistanceprim}{\hatchdistanceprim}}
{\pgfpoint{\hatchdistanceprim-1pt}{\hatchdistanceprim-1pt}}%
{
	\pgfsetcolor{\tikz@pattern@color}
	\pgfsetlinewidth{\hatchthicknessprim}
	\pgfpathmoveto{\pgfqpoint{0pt}{0pt}}
	\pgfpathlineto{\pgfqpoint{0pt}{\hatchdistanceprim}}
	\pgfusepath{stroke}
}
\makeatother

\newcommand*{\MyHalfThickness}{0.035} 
\newcommand*{\SqrtTwo}{1.41421356}
\newcommand*{\MyRadiusBig}{2*\MyHalfThickness} 
\newcommand*{\MyLabelSpace}{0.1} 

\makeatletter
\@namedef{subjclassname@2020}{%
  \textup{2020} Mathematics Subject Classification}
\makeatother

\begin{document}

\title[Norms of structured matrices]{Norms of structured random matrices}

\author[R. Adamczak]{Rados\l{}aw Adamczak}
\address{Rados{\l}aw Adamczak, Institute of Mathematics, University of Warsaw, Banacha 2, 02--097 Warsaw, Poland.}
\email{radamcz@mimuw.edu.pl}

\author[J. Prochno]{Joscha Prochno}
\address{Joscha Prochno, Faculty of Computer Science and Mathematics, University of Passau, Innstra{\ss}e 33, 94032 Passau, Germany.}
\email{joscha.prochno@uni-passau.de}

\author[M. Strzelecka]{Marta Strzelecka}
\address{Marta Strzelecka, Institute of Mathematics and Scientific Computing, University of Graz, Heinrichstra{\ss}e 36, 8010 Graz, Austria; 
	Institute of Mathematics, University of Warsaw, Banacha 2, 02--097 Warsaw, Poland.}
\email{martast@mimuw.edu.pl (corresponding author)}

\author[M. Strzelecki]{Micha\l{} Strzelecki}
\address{Micha{\l} Strzelecki, Institute of Mathematics, University of Warsaw, Banacha 2, 02--097 Warsaw, Poland.}
\email{michalst@mimuw.edu.pl}

\begin{abstract}
For $m,n\in\mathbb{N}$, let $X=(X_{ij})_{i\leq m,j\leq n}$ be a random matrix,
$A=(a_{ij})_{i\leq m,j\leq n}$ a real deterministic matrix,
and $X_A=(a_{ij}X_{ij})_{i\leq m,j\leq n}$ the corresponding structured random matrix.
We study the expected operator norm of $X_A$ considered as a random operator between $\ell_p^n$ and $\ell_q^m$ for $1\leq p,q \leq \infty$.
We prove optimal bounds up to logarithmic terms when the underlying random matrix $X$ has i.i.d.\ Gaussian entries,
independent mean-zero  bounded entries,
or independent mean-zero $\psi_r$ ($r\in(0,2]$) entries.
In certain cases, we determine the precise order of the expected norm up to constants. Our results are expressed through a sum of operator norms of Hadamard products $A\circ A$ and $(A\circ A)^T$.
\end{abstract}

\date{March 31, 2023}



\maketitle

\vspace{-4pt}
\tableofcontents

%
\section{Introduction and main results}
%

With his work on the statistical analysis of large samples \cite{W1928}, Wishart initiated the systematic study of large random matrices. Ever since, random matrices have continuously entered more and more areas of mathematics and applied sciences beyond probability theory and statistics, for instance, in numerical analysis through the work of Goldstine and von Neumann \cite{vN,GvN1951} and in quantum physics through the works of Wigner \cite{W1955,W1957,W1958} on his famous semicircle law, which resulted in significant effort to understand spectral statistics of random matrices from an asymptotic point of view. Today, random matrix theory has grown into a vital area of probability theory and statistics, and within the last two decades, random matrices have come to play a major role in many areas of (algorithmic) computational mathematics, for instance, in questions related to sparsification methods \cite{AM2007,ST2004} and sparse approximation \cite{Tr2008_b,Tr2008}, dimension reduction \cite{AC2009,BDN2015,LLR1995}, or combinatorial optimization \cite{NRV2014,So2011}. We refer the reader to \cite{ABDF2015, AGZ2010, T2015} for more information.

In this paper, we are interested in the non-asymptotic theory of (large) random matrices. This theory plays a fundamental role in geometric functional analysis at least since the '70s, the connection coming in various different flavors. It is of particular importance in the geometry of Banach spaces and the theory of operator algebras \cite{BGN,BG1981,DS2001,Gl1983,G1985,HPV2021} and their applications to high-dimensional problems, for instance, in convex geometry \cite{FY2019,GLSW2002}, compressed sensing \cite{CGLP2012,FR2013,Rau2010,V2012}, information-based complexity \cite{HKNPUsurvey,HKNPU2021}, or statistical learning theory \cite{RV2007,V2018}. On the other hand, geometric functional analysis had and still has enduring influence on random matrix theory as is witnessed, for instance, through applications of measure concentration techniques; we refer to \cite{DS2001,L2007} and the references cited therein.
The quantity we  study and focus on here concerns the expected operator norm of random matrices considered as operators between finite-dimensional $\ell_p$ spaces; recall that $\ell_p^n$ denotes the space $\RR^n$ equipped with the (quasi-)norm $\|\cdot\|_p$, given by $\|(x_j)_{j=1}^n\|_p = (\sum_{j=1}^n |x_j|^p)^{1/p}$ for $0<p<\infty$ and $\|(x_j)_{j=1}^n\|_\infty = \max_{j\le n}|x_j|$ if $p=\infty$. We address the following problem:
for $1 \leq p,q \leq \infty$ and $m,n\in\NN$, determine the right order (up to constants that may depend on the parameters $p$ and $q$) of
 \[
 \EE\| X_A\colon \ell^n_p \to \ell^m_q\|,
 \]
where, given a deterministic real $m\times n$ matrix $A = (a_{ij})_{i\leq m, j\leq n}$
and a random matrix $X = (X_{ij})_{i\leq m, j\leq n}$, we denote by
\[
X_A \coloneqq A\hadprod X = (a_{ij} X_{ij})_{i\leq m, j\leq n}
\]
the structured random matrix; the symbol $\hadprod$ stands for the Hadamard product of matrices (i.e., entrywise multiplication).
The bounds on the expected operator norm should be of optimal order and expressed in terms of the coefficients $a_{ij}$, $i\leq m,j\leq n$. Understanding such expressions and related quantities is important, for instance, when studying the worst-case error of optimal algorithms which are based on random information in function approximation problems \cite{HKNPU2021} (see also \cite{KU2021}) or the quality of random information for the recovery of vectors from an $\ell_p$-ellipsoid, where (the radius of) optimal information is given by Gelfand numbers of a diagonal operator \cite{HPS2021}.

In the case where the random entries of $X$ are i.i.d.~standard Gaussians (then we write $G_A$ instead of $X_A$) and $1\leq p,q \leq \infty$, we will show the following bound, which is sharp up to logarithmic terms:
  \begin{equation}\label{main-result-gaussian-intro}
 D_1 + D_2 \lesssim \EE \|G_A \colon \ell^n_p\to \ell^m_q\|
  \lesssim  (\ln n)^{1/p^*}  (\ln m)^{1/q} \bigl[  \sqrt{\ln(mn)} D_1
  +\sqrt{\ln n} D_2\bigr],
  \end{equation}
  where $D_1  \coloneqq \|A\hadprod A \colon \ell^n_{p/2} \to \ell^m_{q/2}\|^{1/2}$ , $D_2  \coloneqq \|(A\hadprod A)^T \colon \ell^m_{q^*/2} \to \ell^n_{p^*/2}\|^{1/2}$, and $p^*$ denotes the H\"older conjugate of $p$ defined by the relation $1/p+1/p^*=1$.
 As will be explained later, we obtain sharp estimates in certain cases and derive results similar to \eqref{main-result-gaussian-intro} for other models of randomness.

%
\subsection{History of the problem and known results}\label{sec:known-results}
%

In what follows, $A = (a_{ij})_{i,j}$ is a real deterministic matrix and $G=(g_{ij})_{i,j}$ always stands for a random matrix with i.i.d.\ standard Gaussian entries (usually the matrices are of size $m\times n$ unless explicitly stated otherwise).
  We use $C(r)$, $C(r,K)$, etc.\ for positive constants which may depend only on the parameters given in brackets and write $C, C', c,c',\dots$ for positive absolute constants.
 The symbols $\lesssim$, $\lesssim_{r}$, $\lesssim_{r, K}$, etc.\ denote that the inequality holds up to multiplicative constants depending only on the parameters given in the subscripts;
we write $a\asymp b$ if $a\lesssim b$ and $b\lesssim a$, and $\asymp_r$, $\asymp_{r,K}$, etc. if the constants may depend on the parameters given in the subscript.

In 1975, Bennett, Goodman, and Newman \cite{BGN} proved that
 if $X$ is an $m\times n$ random matrix
with independent, mean-zero entries taking values in $[-1,1]$,
and $2\leq q < \infty$,
then
\begin{equation}
	\label{eq:BGN}
\EE \| X\colon \ell^n_2 \to \ell^m_q\| \lesssim_q \max\{n^{1/2}, m^{1/q}\}.
\end{equation}
In fact, up to constants, this estimate is best possible:
for any $m\times n$ matrix $X'$ with $\pm 1$ entries one readily sees that $ \| X'\colon \ell^n_2 \to \ell^m_q\| \geq \max\{n^{1/2}, m^{1/q}\}$; just use standard unit vectors and operator duality.
Moreover,
in this `unstructured' case, where $a_{ij}=1$ for all $i,j$,
it is easy to extend \eqref{eq:BGN} to the whole range of $p, q\in [1,\infty]$
(see~\cite{Bennett,CarlMaureyPuhl} or Remark~\ref{rem:BGN-extended-for-all-p,q} below). Also, if all entries are i.i.d.\ Rademacher random variables, then
the bounds are two-sided, i.e.,
the expected operator norm is, up to constants, the same as the minimal norm for all $p$, $q$
(see \cite[Proposition~3.2]{Bennett} or \cite[Satz~2]{CarlMaureyPuhl}).

The case most studied  in the literature is the one of the spectral norm, i.e., the $\ell_2^n \to \ell_2^m$ operator norm.
 Seginer~\cite{Seginer}
 proved in 2000 that if $X = (X_{ij})_{i\leq m, j\leq n}$ is an $m\times n$ random matrix
 with i.i.d.\ mean-zero entries,
 then its operator norm is of the same order as the sum of expectations
 of the maximum Euclidean norm of rows and columns of $X$, i.e.,
 \begin{align}
 	\label{eq:Seginer}
 	\EE \|X\colon \ell_{2}^n\to\ell_2^m\|
 	& \asymp
    \EE \max_{j \leq n} \|(X_{ij})_{i=1}^m\|_2
    + \EE \max_{i \leq m}\|(X_{ij})_{j=1}^n\|_{2}.
 \end{align}
Riemer and Sch\"utt~\cite{Riemer-Schuett} proved that,
up to a logarithmic factor $\ln(en)^2$, the same holds true for any random matrix with independent but not necessarily identically distributed mean-zero entries. Let us also mention that in the Gaussian setting one can use a non-commutative Khintchine bound (see, e.g., \cite[Equation (4.9)]{Tropp12}) to show that, up to a factor $\sqrt{\ln n}$, the expected spectral norm is of the order of the largest Euclidean norm of its rows and columns.

In the very same setting that was considered by Riemer and Sch\"utt, Lata{\l}a~\cite{Latala-Some-estimates} had obtained a few years earlier the dimension-free estimate
\begin{equation*}
\label{eq:Latala}
\EE  \|X \colon \ell_{2}^n\to\ell_2^m\|
\lesssim
 \max_{j\leq n} \Bigl(\sum_{i=1}^m \EE X_{ij}^2\Bigr)^{1/2}
 + \max_{i\leq m} \Bigl(\sum_{j=1}^n \EE X_{ij}^2\Bigr)^{1/2}
 +  \Bigl(\sum_{i=1}^m \sum_{j=1}^n \EE X_{ij}^4\Bigr)^{1/4}.
\end{equation*}
This bound is superior to the Riemer--Sch\"utt bound in the case of 
matrices with all entries equal to $1$ and is optimal for Wigner matrices. In other cases, like the one of diagonal matrices, the Riemer--Sch\"utt bound is better.

In the case of structured Gaussian matrices, Lata{\l}a, van~Handel, and Youssef~\cite{LvHY}, building upon earlier work of Bandeira and van~Handel \cite{BvH2016} (which combined the moment method with combinatorial considerations) as well as results proved by van~Handel in \cite{vH2017} (which used  Slepian's lemma),
obtained the precise behavior without any logarithmic terms in the dimension, namely
  \begin{align}
 	\label{eq:LvHY}
 	\EE \| G_A \colon \ell_{2}^n\to\ell_2^m\|
 	& \asymp
 \EE \max_{j \leq n} \|(a_{ij} g_{ij})_{i=1}^m\|_2
 + \EE \max_{i \leq m}\|(a_{ij} g_{ij})_{j=1}^n\|_{2}  \\ \nonumber
 & \asymp \max_{j \leq n} \|(a_{ij})_{i=1}^m\|_2
  + \max_{i \leq m}\|(a_{ij})_{j=1}^n\|_{2} + \EE\max_{i\leq m, j\leq n}|a_{ij}g_{ij}| .
\end{align}
Their proof is based on a clever block decomposition of the underlying matrix (see \cite[Figure 3.1]{LvHY}). This result finally answered in the affirmative a conjecture made by Lata{\l}a more than a decade before. We also refer the reader to the survey \cite{vH2017_survey} discussing in quite some detail results prior to \cite{LvHY} and \cite{vH2017} --- the latter work discusses the conjectures of Lata{\l}a and van Handel and shows their equivalence.

Very recently, Lata\l{}a and \'Swi\k{a}tkowski \cite{latala-swiatkowski2021norms} investigated a similar problem when the underlying random matrix has Rademacher entries.
They proved a lower bound which, up to a $\ln \ln n$ factor, can be reversed  for randomized $n\times n$ circulant  matrices.

In \cite{GHLP}, Gu\'edon, Hinrichs, Litvak, and Prochno studied our main and motivating question on the order of the expected operator norm of structured random matrices considered as operators between $\ell_p^n$ and $\ell_q^m$ in the special case where $ p\leq 2\leq q $ and the random entries are Gaussian. In this situation, where we are not dealing with the spectral norm, the moment method cannot be employed. The approach in \cite{GHLP} was therefore different and based on a majorizing measure construction combining the works \cite{GMPT2008} and \cite{GR2007}. In \cite[Theorem 1.1]{GHLP}, the authors proved that if $1< p\leq 2\leq q < \infty$, then
\begin{align}
\nonumber
\EE \|G_A\colon \ell_{p}^n\to\ell_q^m\|
\lesssim
 \gamma_q  \max_{j \leq n}\|(a_{ij})_{i=1}^m\|_q &+ (p^*)^{5/q} (\ln m)^{1/q}
\gamma_{p^*}\max_{i \leq m}\|(a_{ij})_{j=1}^n\|_{p^*} \\
& +  (p^*)^{5/q} (\ln m)^{1/q} \gamma_q \ \EE\max_{i\leq m, j\leq n}|a_{ij}g_{ij}|,
\label{eq:GHLP}
\end{align}
where $\gamma_r \coloneqq (\EE |g|^r)^{1/r}$ for a standard Gaussian random variable $g$. Moreover, for $p = 1$ and $q \ge 2$, it was noted in \cite[Remark~1.4]{GHLP} (see also \cite[Twierdzenie~2]{Matlak}) that
\begin{align}\label{eq:Matlak}
 \EE \|G_A\colon \ell_{1}^n\to\ell_q^m\|
& \lesssim \sqrt{q}  \max_{j\le n} \|(a_{ij})_{i=1}^m\|_q + \EE\max_{i\leq m, j\leq n}|a_{ij}g_{ij}|.
\end{align}
Later, an extension of \eqref{eq:GHLP} to the case of matrices with i.i.d.\ isotropic log-concave rows
was obtained by Strzelecka in~\cite{Strzelecka}.

Trying to extend the upper bound for $\EE\|G_A\colon \ell_p^n \to \ell_q^m\|$ to the whole range $1\leq p, q\leq \infty$ one encounters two difficulties.
First of all, the methods used in order to prove  \eqref{eq:GHLP} fail
if $q\leq 2$ or $p \geq 2$, because the majorizing measure construction used in \cite{GHLP} is restricted to the case $q\geq 2$ and the assumption $1<p\leq 2$ is required in a H\"older bound. Moreover, when $q\leq 2$ or $p \geq 2$ the result cannot hold with the right-hand side of the same form as in~\eqref{eq:GHLP}
(see Remark~\ref{rem:BGN-extended-for-all-p,q} below for counterexamples\footnote{By Jensen's inequality, the expected norm of a matrix with i.i.d.\ Rademacher entries is less than or equal to $\sqrt{2/\pi}$ times the expected norm of the matrix with Gaussian entries, so \eqref{eq:GHLP} for $q\le 2$ or $p\ge 2$ would imply the same (up to a constant) bound for $\pm 1$ matrices, which does not hold in this range of $(p,q)$ as we  explain in Remark~\ref{rem:BGN-extended-for-all-p,q}.}
to \eqref{eq:GHLP} in the cases $q\leq 2$ and $p \geq 2$). 
This explains the different form of expressions $D_1$ and $D_2$
in \eqref{main-result-gaussian-intro}, which in the range $p \le 2 \le q$ reduce to the maxima of norms on the right-hand side of~\eqref{eq:GHLP} --- see~\eqref{eq:display-after-1.8} below.

\subsection{Lower bounds and conjectures}\label{subsect:conjectures}

By arguments similar to the ones used in order to prove the lower bound in~\eqref{eq:LvHY}, one can check that in the range considered in \cite{GHLP,Matlak} (i.e., $1\leq p\leq 2\leq q \leq \infty$) one has
\begin{align}
\label{eq:lower-bound-GHLP}
 \EE \|G_A\colon \ell_{p}^n\to\ell_q^m\|
 \gtrsim_{p,q}
\max_{j \leq n}\|(a_{ij})_{i=1}^m\|_q &+ \max_{i \leq m}\|(a_{ij})_{j=1}^n\|_{p^*} \\
&+ \EE\max_{i\leq m, j\leq n}|a_{ij}g_{ij}|. \nonumber
\end{align}
Note that for $p=1$,
\begin{displaymath}
\max_{i \leq m}\|(a_{ij})_{j=1}^n\|_{p^*} = \max_{i\le m,j\le n} |a_{ij}| \le \sqrt{\pi/2}\,\EE\max_{i\leq m, j\leq n}|a_{ij}g_{ij}|,
\end{displaymath}
which explains the simplified form of \eqref{eq:Matlak}.

We remark that the proof of \eqref{eq:lower-bound-GHLP} is based merely on the observation that the operator norm is greater than the maximum entry of the matrix and the appropriate maximum norms of its rows and columns, combined with comparison of moments for Gaussian random vectors. Another but related way to proceed, valid for all $1 \le p, q \le \infty$, is to exchange expectation and suprema over the $\ell_p^n$ and $\ell_{q^\ast}^m$ balls in the definition of the operator norm.
 We present the details in Subsection~\ref{subsec:lower-bounds}.
In particular, Proposition \ref{prop:cm-lower-bound} and Corollary \ref{cor:cm-lower-bound} imply\footnote{We use here also a trivial observation that $\|G_A\colon \ell_{p}^n\to\ell_q^m\| \ge \max_{i,j} |a_{ij}g_{ij}|$.} that, for $1\le p,q \le \infty$,
\begin{align}
\EE \|G_A\colon \ell_{p}^n\to\ell_q^m\| \gtrsim  \|A\hadprod A \colon \ell^n_{p/2} \to \ell^m_{q/2} \|^{1/2} &+ \|(A\hadprod A)^T \colon \ell^m_{q^*/2} \to \ell^n_{p^*/2}\|^{1/2} \nonumber\\
&+ \EE\max_{i\leq m, j\leq n}|a_{ij}g_{ij}|.
\label{eq:lower-bound}
\end{align}
It is an easy observation (see Lemma \ref{lem:special-norms} below) that for  $p\le 2 \le q$,
\begin{equation}
\label{eq:display-after-1.8}
\begin{split}
\|A\hadprod A \colon \ell^n_{p/2} \to \ell^m_{q/2} \|^{1/2}  &= \max_{j \leq n}\|(a_{ij})_{i=1}^m\|_q, \\
\|(A\hadprod A)^T \colon \ell^m_{q^*/2} \to \ell^n_{p^*/2}\|^{1/2}  &= \max_{i \leq m}\|(a_{ij})_{j=1}^n\|_{p^\ast}.
\end{split}
\end{equation}
Thus, in the range $1 \le  p \le 2 \le q < \infty$ considered in \cite{GHLP,Matlak}, the lower bounds \eqref{eq:lower-bound-GHLP} and \eqref{eq:lower-bound} coincide.

Although it would be natural to conjecture at this point that the bound~\eqref{eq:lower-bound} may be reversed up to a multiplicative constant depending only on $p,q$, such a reverse bound turns out not to be true in the case $p\le q< 2$ (and in the dual one $2< p\le q$) as we shall show in Subsection~\ref{subsection:counterexample}.

In order to conjecture the right asymptotic behavior of $\EE\|G_A\colon \ell_p^n \to \ell_q^m\|$, one may take a look at the boundary values of $p$ and $q$,
i.e., $p\in \{1,\infty \}$ or $q\in \{1, \infty \}$. Note that \eqref{eq:Matlak} provides an asymptotic behavior of $\EE\|G_A\colon \ell_p^n \to \ell_q^m\|$ on a part of this boundary 
(i.e., for $p=1$ and $2\le q\le \infty$ and in the dual case $q=\infty$ and $1\le p\le 2$). We provide sharp results on the remaining parts of the boundary of $[1,\infty]\times[1,\infty]$ (see dense lines on the boundary of Figure~\ref{fig:diagram-big-after} below):
			\begin{alignat*}{2}
			\EE \|G_A \colon \ell^n_p\to \ell^m_1\|	&\asymp_{p}
			D_1 +D_2 &&\qquad \text{for all }1<p\le \infty,
			\\
			\EE \|G_A \colon \ell^n_\infty\to \ell^m_q\|	&\asymp_{q}
			D_1 +D_2 &&\qquad \text{for all }1\le q<\infty,\\
			\EE \|G_A \colon \ell^n_1\to \ell^m_q\|	&\asymp_{\phantom{r}} D_1 + \max_{j\le n} (\sqrt{\ln(j+1)} b_j^{\rearr{}})	&&\qquad \text{for all }1\le q\leq 2,
			\\
			\EE \|G_A\colon \ell_p^n \to \ell_\infty^m\| &\asymp_{\phantom{r}} D_2+ \max_{i\le m} (\sqrt{\ln(i+1)} d_i^{\rearr{}}) &&\qquad \text{for all }2\leq p\le \infty,
			\end{alignat*}
			where 
		\begin{equation*}
\begin{split}
D_1  &\coloneqq \|A\hadprod A \colon \ell^n_{p/2} \to \ell^m_{q/2}\|^{1/2},\\
D_2  &\coloneqq \|(A\hadprod A)^T \colon \ell^m_{q^*/2} \to \ell^n_{p^*/2}\|^{1/2},
\end{split}
\qquad\qquad
\begin{split}
b_j&\coloneqq \|(a_{ij})_{i\le m} \|_{2q/(2-q)}, \\
d_i &\coloneqq \|(a_{ij})_{j\le n} \|_{2p/(p-2)},
\end{split}
\end{equation*}
and with $(x_1^{\rearr{}}, \ldots,x_n^{\rearr{}})$ denoting 
the non-increasing rearrangement of $(|x_1|,\ldots, |x_n|)$ for a given $(x_j)_{j\le n}\in\RR^n$.
(For the precise formulation see Propositions~\ref{prop:gauss-q=1} and \ref{prop:gauss-p=1-q<2}, and Corollary~\ref{cor:gauss-p>2-q=infty} below.)

Moreover, in Subsection~\ref{subsec:lower-bounds} we generalize the lower bounds from the boundary into the whole range $(p,q)\in [1,\infty]\times [1,\infty]$ (see Figure~\ref{fig:diagram-big-after} below), i.e., we prove
\begin{equation} \label{eq:hippo-intro-lower}
\EE \|G_A\colon \ell_{p}^n\to\ell_q^m\|
\gtrsim_{p,q} D_1+D_2 +
\begin{cases}
  \EE \max_{i \le m,j\le n} |a_{ij}g_{ij}| &  \text{if }\  p\leq 2\leq q,\\
  \max_{j\le n}\sqrt{\ln (j+1)} b_j^{\rearr{}}  &    \text{if }\ p\leq q\leq 2,\\
  \max_{i\le m}\sqrt{\ln (i+1)} d_i^{\rearr{}}&   \text{if } \ 2\leq p \leq q,\\
  0 &   \textrm{if } \ q<p.
\end{cases} 
\end{equation}


\begin{figure}[h]
	\centering	
	\begin{tikzpicture}[scale=3.5]
	\draw [pattern=south west lines, pattern color=ColorGHLP] (1,0-\MyHalfThickness) -- (1,0+\MyHalfThickness) -- (2-\MyHalfThickness,0+\MyHalfThickness) -- (2-\MyHalfThickness,1) -- (2+\MyHalfThickness,1) -- (2+\MyHalfThickness,0-\MyHalfThickness) -- cycle;
	\draw [pattern=south west lines, pattern color=ColorGHLP]  (1,1) circle [radius=\MyRadiusBig];	
	\draw [pattern=loose south west lines,, pattern color=ColorGHLPLight]  (1,0+\MyHalfThickness) -- (1,1-\MyRadiusBig) arc(270:360:\MyRadiusBig) -- (2-\MyHalfThickness,1) --  (2-\MyHalfThickness,0+\MyHalfThickness) -- cycle;	
	\draw [pattern=horizontal lines, pattern color=ColorRadek] (2-\MyHalfThickness,1) -- (2-\MyHalfThickness,{2-cos(30)*\MyRadiusBig})  arc (240:-60:\MyRadiusBig) -- (2+\MyHalfThickness,1) --  cycle;
	\draw [pattern=MyLooseHor, pattern color=ColorRadekLight] (1+\MyRadiusBig,1) arc(0:45:\MyRadiusBig)  -- (2-\MyRadiusBig/\SqrtTwo,2-\MyRadiusBig/\SqrtTwo) arc(225:240:\MyRadiusBig)-- (2-\MyHalfThickness,1) --  cycle;
	\draw [pattern=vertical lines, pattern color=ColorRadek] ({0+\MyRadiusBig*cos(30)},0-\MyHalfThickness) arc(330:30:\MyRadiusBig) -- (1,0+\MyHalfThickness) -- (1,0-\MyHalfThickness) --  cycle;
	\draw [pattern=MyLooseVert, pattern color=ColorRadekLight] ({0+\MyRadiusBig*cos(30)},0+\MyHalfThickness) arc(30:45:\MyRadiusBig) -- (1-\MyRadiusBig/\SqrtTwo,1-\MyRadiusBig/\SqrtTwo) arc(225:270:\MyRadiusBig) -- (1,0+\MyHalfThickness) --  cycle;
	\draw [pattern=south east lines, pattern color=ColorEasy] ({0+\MyHalfThickness},{0+\MyRadiusBig*cos(30)}) arc(60:120:\MyRadiusBig) -- (0-\MyHalfThickness,2+\MyHalfThickness) -- ({2-\MyRadiusBig*cos(30)},2+\MyHalfThickness) arc(150:210:\MyRadiusBig) -- (0+\MyHalfThickness,2-\MyHalfThickness)  -- cycle;
	\draw [pattern=loose south east lines, pattern color=ColorEasyLight]
	({0+\MyRadiusBig/\SqrtTwo},{0+\MyRadiusBig/\SqrtTwo}) arc(45:60:\MyRadiusBig)-- (0+\MyHalfThickness,{0+\MyRadiusBig*cos(30)}) -- (0+\MyHalfThickness,2-\MyHalfThickness) -- ({2-\MyRadiusBig*cos(30)},2-\MyHalfThickness) arc(210:225:\MyRadiusBig) -- (1+\MyRadiusBig/\SqrtTwo,1+\MyRadiusBig/\SqrtTwo) arc(45:225:\MyRadiusBig) -- cycle;
	\node [below] at (2,0-\MyLabelSpace) {$p=1$};
	\node [below] at (1,0-\MyLabelSpace) {$p=2$};
	\node [below] at (0,0-\MyLabelSpace) {$p=\infty$}; 
	\node [left] at (0-\MyLabelSpace,0) {$q=\infty$}; 
	\node [left] at (0-\MyLabelSpace,1) {$q=2$};
	\node [left] at (0-\MyLabelSpace,2) {$q=1$};
	\end{tikzpicture}
	\caption[foo]{
The third summand in \eqref{eq:hippo-intro-lower} and  in Conjecture~\ref{conj:conjecture_1}:
		\begin{alignat*}{2}
		&\text{northeast lines: \ } && \qquad\EE \max_{i \le m,j\le n} |a_{ij}g_{ij}|,\\
		&\text{horizontal lines: \ }  && \qquad \max_{j\le n}\sqrt{\ln (j+1)} b_j^{\rearr{}},\\
		&\text{vertical lines: \ } &&\qquad \max_{i\le m}\sqrt{\ln (i+1)} d_i^{\rearr{}},\\
		&\text{northwest lines: \ } &&\qquad 0.
		\end{alignat*}
		Note that the horizontal axis represents $1/p$ and the vertical one $1/q$.
		Dense lines correspond to exact asymptotics
		and loosely spaced lines to upper and lower bounds matching up to logarithms. 
	}	
	\label{fig:diagram-big-after} 
\end{figure}	


Let us now discuss the relation between the terms appearing above. We postpone the proofs of all the following claims to Section~\ref{sec:lower-bounds}.

In the case $p\le 2 \le q$, we have
\begin{align}		\label{eq:bound_Emax_term}
	D_1+D_2+ \EE \max_{i \le m,j\le n} |a_{ij}g_{ij}| 
	&\asymp_{p, q} D_1+D_2+\max_{i\le m, j\le n}\sqrt{\ln (j+1)} a_{ij}' 
	\\
	&\asymp_{p, q} D_1+D_2+\max_{i\le m, j\le n}\sqrt{\ln (i+1)} a_{ij}'',\nonumber
\end{align}
where the matrices $(a_{ij}')_{i,j}$ and $(a_{ij}'')_{i,j}$ are obtained by permuting the columns and rows, respectively, of the matrix $(|a_{ij}|)_{i,j}$ in such a way that $\max_i a_{i1}'\ge \dots \ge \max_i a_{in}'$ and $\max_j a_{1j}'' \ge \dots \ge \max_j a_{mj}''$. 
Therefore, in the range $1\le p\le q \le \infty$ the right-hand side of \eqref{eq:hippo-intro-lower} changes continuously with $p$ and $q$ (for a fixed matrix~$A$).

Obviously,  $ \max_{j\le n}\sqrt{\ln (j+1)} b_j^{\rearr{}} \ge \max_{i\le m, j\le n}\sqrt{\ln (j+1)} a_{ij}'  $ and, in general, the former quantity may be of larger order than the latter one.
In Subsection~\ref{subsection:counterexample} we shall present a more subtle relation: for every $1\le p\le q< 2$ we shall give an example showing that the right-hand side of \eqref{eq:hippo-intro-lower} may be of larger order than $D_1+D_2+\EE \max_{i \le m,j\le n} |a_{ij}g_{ij}| $. Note that by duality, i.e., 
the fact that
	\begin{equation}	\label{eq:duality}
	 \|X_A\colon \ell_{p}^n\to\ell_q^m\| = \|(X_A)^T \colon \ell_{q^*}^m\to\ell_{p^*}^n \| =\|(X^T)_{A^T} \colon \ell_{q^*}^m\to\ell_{p^*}^n \| ,
	 \end{equation}
 the same holds in the case $2< p\le q$. This suggests that the behavior of $\EE\|G_A\colon \ell_p^n\to \ell_q^m\|$ is different in the regions with horizontal or vertical lines than in the region with northeast lines.

Moreover, we have
\begin{equation}
	\label{eq:bound_q<p_third_term_unnec}
D_1+D_2\gtrsim_{p,q}
\begin{cases}
\max_{j\le n}\sqrt{\ln (j+1)} b_j^{\rearr{}}  &  \text{if }  q< p \text{ and } q<2, \\
  \max_{i\le m}\sqrt{\ln (i+1)} d_i^{\rearr{}}&   \text{if } q <p \text{ and } p^\ast<2
\end{cases}
\end{equation}
(see Subsection~\ref{sec:The_proof_of_Inequalities_with_D1+D2}).
 Note that this is not the case for $p \le q$, as one can easily see by considering, e.g., $A$ equal to the identity matrix. 
 This suggests a different (than in other regions), simplified, behavior of $\EE\|G_A\colon \ell_p^n\to \ell_q^m\|$  in the region with northwest lines.

Given the discussion above, the lower bounds presented in \eqref{eq:hippo-intro-lower}, and the fact that they can be reversed for all $p\in[1,\infty]$, $q\in \{1,\infty\}$ (and for all $q\in[1,\infty]$,  $p\in \{1,\infty\}$), it is natural to conjecture the following.

\begin{conjecture}\label{conj:conjecture_1}
For all $1\le p, q \le \infty$,
we conjecture that
\begin{equation} \label{eq:hippo}
\EE \|G_A\colon \ell_{p}^n\to\ell_q^m\|
\asymp_{p,q} D_1+D_2 +
\begin{cases}
  \EE \max_{i \le m,j\le n} |a_{ij}g_{ij}| &  \text{if }\  p\leq 2\leq q,\\
  \max_{j\le n}\sqrt{\ln (j+1)} b_j^{\rearr{}}  &    \text{if }\ p\leq q\leq 2,\\
  \max_{i\le m}\sqrt{\ln (i+1)} d_i^{\rearr{}}&   \text{if } \ 2\leq p \leq q,\\
  0 &   \textrm{if } \ q<p.
\end{cases} 
\end{equation}
\end{conjecture}

\begin{remark}\label{rmk:conj-wrong}
	One could pose another natural conjecture, based on the potential generalization of the first line of the bound \eqref{eq:LvHY},
	 namely that  the inequality
	\begin{equation}\label{eq:conj-wrong}
		\EE \|G_A\colon \ell_{p}^n\to\ell_q^m\|
		\asymp_{p,q} 
		\EE\max_{i\le m} \| (a_{ij}g_{ij})_j\|_{p^\ast} +\EE\max_{j\le n}  \| (a_{ij}g_{ij})_i\|_q 
	\end{equation}
	holds for all $1\le p,q \le \infty$.
	Indeed, the lower bound is true with constant $\frac 12$, since  for every deterministic matrix $X$ one has
	\[
		 \|X\colon \ell_{p}^n\to\ell_q^m\| \ge \max\Bigl\{\max_{i\le m} \| (X_{ij})_j\|_{p^\ast} ,\max_{j\le n}  \| (X_{ij})_i\|_q \Bigr\}.
	\]
	However, as we prove in Subsection~\ref{subsect:conj-wrong}, this conjecture is wrong: although the right-hand sides of \eqref{eq:hippo} and \eqref{eq:conj-wrong} are comparable  in the range $1\le p \le 2\le q\le \infty$, for every pair of $p,q$ outside this range the right-hand side of \eqref{eq:conj-wrong} may be of smaller order than the left-hand side.
\end{remark}

Let us now present a conjecture concerning the boundedness of linear operators given by infinite dimensional matrices. 
In what follows, we say that a matrix $B= (b_{ij})_{i,j\in \NN}$ defines a bounded operator from $\ell_p(\NN)$ to $\ell_q(\NN)$ if for all $x \in \ell_p(\NN)$ the product $B x$ is well defined, belongs to $\ell_q(\NN)$ and the corresponding linear operator is bounded.

\begin{conjecture}\label{conj:conjecture_infty}
Let $A = (a_{ij})_{i,j\in \NN}$ be an infinite matrix with real coefficients and let $1\le p, q \le \infty$.
We conjecture that the matrix $G_A = (a_{ij}g_{ij})_{i,j\in \NN}$  defines a bounded linear operator between $\ell_p(\NN)$ and $\ell_q(\NN)$ almost surely if and only if the matrix $A\circ A$ defines a bounded linear operator between $\ell_{p/2}(\NN)$ and $\ell_{q/2}(\NN)$, the matrix $(A\circ A)^T$ defines a bounded linear operator between $\ell_{q^\ast/2}(\NN)$ and $\ell_{p^\ast/2}(\NN)$, and 
\begin{itemize}
	\item in the case $p\leq 2\leq q$,  $\EE \sup_{i,j\in\NN} |a_{ij}g_{ij}|<\infty$,
	\item in the case $p\leq q\leq 2$, $\lim_{j\to \infty} b_j = 0$,
and $\sup_{j\in \NN}\sqrt{\ln(j+1)} b_j^{\rearr{}} <\infty$, where 
$b_j =  \|(a_{ij})_{i\in \NN}\|_{2q/(2-q)}$, $j\in\NN$,
	\item in the case $2\leq p \leq q$, $\lim_{i\to \infty} d_i = 0$, and $\sup_{i\in \NN}\sqrt{\ln (i+1)} d_i^{\rearr{}} <\infty$, where $d_i \coloneqq \|(a_{ij})_{j\in \NN} \|_{2p/(p-2)}$, $i\in\NN$,
	\item (in the case $q<p$ we do not need to assume any additional conditions).
\end{itemize}
\end{conjecture}

We remark that it suffices to prove Conjecture~\ref{conj:conjecture_1} in order to confirm Conjecture~\ref{conj:conjecture_infty}.
\begin{proposition}	\label{prop:finite-implies-infinite}
	Assume $1\le p, q\le \infty$.
		Then \eqref{eq:hippo} for this choice of $p,q$ implies the assertion of 	 Conjecture~\ref{conj:conjecture_infty}  for the same choice of $p,q$.
\end{proposition}
We postpone the proof of this proposition to Subsection~\ref{subsect:infty-dim}.

In this article,	
in addition to the cases $p=q=2$ obtained in \cite{LvHY} and $p=1, q\ge 2$ proved in \cite{GHLP,Matlak},
we confirm Conjecture~\ref{conj:conjecture_1} when $p\in \{1,\infty\}$, $q\in[1,\infty] $ and when $q\in \{1,\infty\}$, $p\in[1,\infty] $. In all the other cases, we are able to prove the upper bounds only up to logarithmic (in the dimensions $m,n$) multiplicative factors (see Corollary~\ref{cor:main-gauss} below).
In particular, Proposition~\ref{prop:finite-implies-infinite} implies that  Conjecture~\ref{conj:conjecture_infty} holds for all $p\in \{1,\infty\}$, $q\in[1,\infty] $ and for all $q\in \{1,\infty\}$, $p\in[1,\infty] $. 

Note that in the structured case we work with, interpolating the results obtained for the boundary cases $p\in \{1, \infty\}$ or $q\in  \{1, \infty\}$ gives a bound with polynomial (in the dimensions) multiplicative constants which are much worse than logarithmic constants from Corollary~\ref{cor:main-gauss} below. However, as we shall see in Remark~\ref{rem:BGN-extended-for-all-p,q} below, interpolation techniques work well in the non-structured case.

%
\subsection{Main results valid for \texorpdfstring{$1\leq p, q\leq \infty$}{1 leq  p, q leq infty}}
\label{subsect:main-results}
%

We start with general theorems valid for the whole range of $p$, $q$.
Results which are based on methods working only for specific values of $p$, $q$,
but yielding better logarithmic terms, are presented in the next subsection.
A brief summary and comparison of all results can be found in Table~\ref{table:summary}.

Before stating our main results, we need to introduce additional  notation.
	For a non-empty set $J\subset \{1,\ldots,n\}$, and $1\leq p\leq \infty$, we define
	\[
		B_p^J\coloneqq \Bigl\{(x_j)_{j\in J}: \sum_{j\in J}|x_j|^p \le 1, \quad x_j\in \RR \Bigr\}.
	\]
	By	$\ell_p^J$ we denote the space $\RR^J\coloneqq \bigl\{(x_j)_{j\in J}: x_j\in \RR \bigr\}$ equipped with the norm
	\[
		\|x\|_{\ell_p^J} = \Bigl(\sum_{j\in J}|x_j|^p \Bigr)^{1/p},
	\]
	whose  unit ball is $B_p^J$.
Obviously, the space $\ell_p^J$ can be identified with a subspace of $\ell_p^n$.
 If  $A\colon \ell_p^n\to \ell_q^m$ is a linear operator, the notation $A\colon \ell_p^J\to \ell_q^I$ means that $A$ is restricted to the space $\ell_p^J$ and composed with a projection onto $\ell_q^I$.
Moreover, for $x=(x_1,\ldots, x_n)\in \RR^n$, $\sup_{J} \|x\|_{\ell_p^J} = \bigl(\sum_{j\le k} |x_j^{\rearr{}}|^p\bigr)^{1/p}$,
 where the supremum is taken over all $J\subset \{1,\ldots,n\}$ with $|J|=k$, and $(x_1^{\rearr{}}, \ldots,x_n^{\rearr{}})$
  is the non-increasing rearrangement of $(|x_1|,\ldots, |x_n|)$.

\begin{theorem}[Main theorem in a general version with sets $ I_0$, $J_0$]
\label{thm:main-gauss-sets}
	Assume that $m\le M$, $n\le N$, $1\leq p,q \leq \infty$, and $G=(g_{ij})_{i\leq M, j\leq N}$ has i.i.d.\ standard Gaussian entries. Then
 \begin{align*}
 \MoveEqLeft[10]
 \EE \sup_{I_0, J_0} \|G_A\colon \ell_p^{J_0} \to \ell_q^{I_0} \|
 =\EE \sup_{I_0, J_0} \sup_{x\in B_p^ {J_0}} \sup_{y\in B_{q^*}^ {I_0}} \sum_{i\in I_0} \sum_{j\in J_0} y_i a_{ij}g_{ij}x_j \\
 \leq \ln(en)^{1/p^*}  \ln(em)^{1/q}
 \Bigl[
 &\bigl( 2.4 \sqrt{\ln(mn)}+ 8 \sqrt{\ln M} + \sqrt{2/\pi}\bigr)
 \sup_{I_0,J_0} \|A\hadprod A \colon \ell^ {J_0}_{p/2} \to \ell^{ I_0}_{q/2}\|^{1/2}\\
& +\bigl(8\sqrt{\ln N } + 2\sqrt{2/\pi} \bigr)
 \sup_{I_0,J_0} \|(A\hadprod A)^T \colon \ell^{ I_0}_{q^*/2} \to \ell^ {J_0}_{p^*/2}\|^{1/2}
 \Bigr],
 \end{align*}
 where the suprema are taken over all sets $I_0\subset \{1,\ldots,M\}$, $J_0\subset \{1,\ldots,N\}$ such that $|I_0|=m$, $|J_0|=n$.
 \end{theorem}

The above theorem gives an estimate on the largest operator norm among all submatrices of $G_A$ of fixed size. Let us remark that apart from being of intrinsic interest, quantities of this type (for $p=q=2$) have appeared in connection with the study of the restricted isometry property of random matrices with independent rows \cite{MR2949869} or in the analysis of entropic uncertainty principles for random quantum measurements \cite{MR3478525, MR3081910}.

Let us now give an outline of the proof of  Theorem~\ref{thm:main-gauss-sets}. Note that
\begin{equation} \label{eq:outline1}
	 \|G_A\colon \ell_p^{J_0} \to \ell_q^{I_0} \| =  \sup_{x\in B_p^ {J_0}} \sup_{y\in B_{q^\ast}^ {I_0}} \sum_{i\in I_0} \sum_{j\in J_0} y_i a_{ij}g_{ij}x_j.
\end{equation}
In the first step of our proof, we find  polytopes $L$ and $K$ approximating (with accuracy depending logarithmically on the dimension) the unit balls in $\ell_p^{J_0}$ and $\ell_{q^\ast}^{I_0}$, respectively.
The extreme points of the sets $K$ and $L$ have a special and simple structure: absolute values of their non-zero coordinates are all equal to a constant depending only on the size of the support of a given point.
Since $K$ is close to
$B_{q^*}^{I_0}$
and $L$ is close
to $B_p^{J_0}$, we may consider only $x\in \Ext (L), y\in \Ext (K)$ in \eqref{eq:outline1}.
 Since non-zero coordinates of $x\in \Ext (L)$ and $y\in \Ext (K)$, respectively, are all equal up to a sign we may use a symmetrization argument and the contraction principle to
 remove
 $x$ and $y$ in the sum on the right-hand side of  \eqref{eq:outline1}.
  Thus, in the next step of the proof we only need to estimate the expected value of
 \[
 	 \sup_{I_0, J_0} \sup_{\emptyset\neq I\subset I_0}\sup_{\emptyset\neq J\subset J_0} |I|^{-1/q^*}|J|^{-1/p} \sum_{i\in I, j\in J} a_{ij}g_{ij},
 \]
where $I$ and $J$ represent the potential supports of points in $\Ext (K)$ and $\Ext (L)$.
To deal with this quantity, we first consider the suprema over the subsets of fixed sizes and use Slepian's lemma to compare the supremum of the Gaussian process above with the supremum of another Gaussian process, which may be  bounded easily.
Then we make use of the term $|I|^{-1/q^*}|J|^{-1/p}<1$, which allows us to go back to suprema over the sets $B_p^ {J_0}$ and $B_{q^*}^ {I_0}$.
At the end, we use the Gaussian concentration inequality to unfix the sizes of sets $I$ and $J$ and complete the proof.

Applying  Theorem~\ref{thm:main-gauss-sets} with $N=n$, $M=m$ immediately yields the following result, which confirms Conjecture \ref{conj:conjecture_1}
up to some logarithmic terms.

\begin{corollary}[Main theorem -- $\ell_p$ to $\ell_q$ version]
\label{cor:main-gauss}
	Assume that $1\leq p,q \leq \infty$ and $G=(g_{ij})_{i\leq m, j\leq n}$ has i.i.d.\ standard Gaussian entries. Then,
 \begin{align*}
 \EE \|G_A \colon \ell^n_p\to \ell^m_q\|
 \lesssim  (\ln n)^{1/p^*}  (\ln m)^{1/q} \Bigl[ & \sqrt{\ln(mn)} \|A\hadprod A \colon \ell^n_{p/2} \to \ell^m_{q/2}\|^{1/2}\\
& +\sqrt{\ln n}  \|(A\hadprod A)^T \colon \ell^m_{q^*/2} \to \ell^n_{p^*/2}\|^{1/2} \Bigr].
 \end{align*}
\end{corollary}

Moreover, we easily recover the same bound in the case of independent bounded entries.
We state and prove a general version with sets $I_0$ and $J_0$ akin to Theorem~\ref{thm:main-gauss-sets} in Subsection~\ref{subsect:coupling}.

\begin{corollary}	\label{thm:main-bounded-introduction}
Assume that $1\leq p,q \leq \infty$ and $X=(X_{ij})_{i\leq m, j\leq n}$ has independent mean-zero entries taking values in $[-1,1]$.
Then
	\begin{align*}
		\EE \|X_A \colon \ell^n_p\to \ell^m_q\|
		 \lesssim  (\ln n)^{1/p^*}  (\ln m)^{1/q} \Bigl[ & \sqrt{\ln(mn)} \|A\hadprod A \colon \ell^n_{p/2} \to \ell^m_{q/2}\|^{1/2}\\
 		 &+\sqrt{\ln n}  \|(A\hadprod A)^T \colon \ell^m_{q^*/2} \to \ell^n_{p^*/2}\|^{1/2} \Bigr].
 \end{align*}
 \end{corollary}

We use the two results above to obtain their analogue in the case of $\psi_r$   entries for $r\le 2$;  these random variables are defined by \eqref{eq:psi r random variables}.
This class contains, among others,
\begin{itemize}
	\item  log-concave random variables (which are $\psi_1$),
	\item heavy tailed Weibull random variables (of shape parameter $r\in(0,1)$,
	 i.e., $\PP(|X_{ij}|\ge t)=e^{-t^r/L}$ for $t\geq 0$),
		\item	random variables satisfying the condition
	\[
		\|X_{ij}\|_{2\rho} \le \alpha \|X_{ij}\|_{\rho} \qquad \text{for all } \rho\ge 1.
	\]
	These random variables are $\psi_r$ with $r=1/\log_2\alpha$.
	They were considered recently in~\cite{LatalaStrzeleckaMat}.
	\end{itemize}
A general version of the following Corollary~\ref{thm:main-psi-r-introduction} with sets $I_0$ and $J_0$ is stated and proved in Subsection~\ref{subsect:coupling}.

\begin{corollary}
\label{thm:main-psi-r-introduction}
	Assume that $K,L >0$, $r\in(0,2]$, $1\leq p,q \leq \infty$, and $X=(X_{ij})_{i\leq m, j\leq n}$ has independent mean-zero entries satisfying
	\begin{align}\label{eq:psi r random variables}
	\PP(|X_{ij}|\ge t) \le Ke^{-t^r/L} \qquad \text{for all } t\ge 0, i\leq m, j\leq n.
	\end{align}
	Then
	\begin{align*}
	\MoveEqLeft[17]
	 \EE \|X_A \colon \ell^n_p\to \ell^m_q\| \\
	 \lesssim_{r,K,L}  (\ln n)^{1/p^*}  (\ln m)^{1/q}\ln(mn)^{\frac 1r - \frac12} \Bigl[&  \sqrt{\ln(mn)} \|A\hadprod A \colon \ell^n_{p/2} \to \ell^m_{q/2}\|^{1/2}\\
	&+\sqrt{\ln  n}  \|(A\hadprod A)^T \colon \ell^m_{q^*/2} \to \ell^n_{p^*/2}\|^{1/2} \Bigr].
	\end{align*}
  	 \end{corollary}

%
%


\afterpage{%
	\clearpage
	\thispagestyle{empty} 
	\begin{landscape} 
	\begin{table}
			\centering 
			\textnormal{%
			\caption{In several cases we prove a result of the type:
				\[
				\EE \|X_A \colon \ell^n_p\to \ell^m_q\|
				\lesssim_{p,q}  \boldsymbol{L_0}\Bigl( \boldsymbol{L_1} \|A\hadprod A \colon \ell^n_{p/2} \to \ell^m_{q/2}\|^{1/2}
				+ \boldsymbol{L_2} \|(A\hadprod A)^T \colon \ell^m_{q^*/2} \to \ell^n_{p^*/2}\|^{1/2} \Bigr),
				\]
				where $L_0, L_1, L_2$ are some logarithmic terms depending only on $ n, m$ and $p,q \in [1,\infty]$.
				The table presents a summary of the results.
				The expression `set $K$' refers to the use of Lemma~\ref{lem:set-K}.
				We denote
				\[
				\LZeroSlepian \coloneqq \LZeroSlepianValue.
				\]
				\vspace{-1em} 
				}        
			\begin{tabular}{@{}lllrllll@{}} 
				\toprule
				 & & & \multicolumn{3}{c}{logarithmic terms} &  &  \\
				\cmidrule(lr){4-6}
				$X_{ij}$  &  & $p,q$  & \multicolumn{1}{c}{$L_0$} & \multicolumn{1}{c}{$L_1$} & \multicolumn{1}{c}{$L_2$}   &   statement
				& tools used in proof          \\
				\midrule
				Gaussian  & & any  & $ \LZeroSlepian$ & $\ln(mn)^{1/2}$ & $\ln(n)^{1/2}$ & Cor.~\ref{cor:main-gauss}  & \makecell[tl]{set $K$ + Gaussian concentration,\\ contraction principle, Slepian's lemma}\\
				bounded  &  &  any & $ \LZeroSlepian$ & $\ln(mn)^{1/2}$ & $\ln(n)^{1/2}$ &  Cor.~\ref{thm:main-bounded-introduction}   & uses the Gaussian case \\
				$\psi_r$, $r\leq 2$ &  & any       &   $ \LZeroSlepian$  &   $\ln(mn)^{1/r}$ & $\ln(n)^{1/2}\ln(mn)^{1/r-1/2}$     &    Cor.~\ref{thm:main-psi-r-introduction}
&coupling, uses Gaussian and bounded case   \\
				\midrule
				Gaussian & (a)  &  $p \leq 2$ & $1$ & $\ln(n)^{1/p^*}$ & $\ln(n)^{1/2+1/p^*}$ &  Prop.~\ref{prop:gauss-p-smaller-2}& set $K$ + Gaussian concentration \\
				& (a$'$)   &   $q \geq 2$  & $1$ & $\ln(m)^{1/2+1/q}$ & $\ln(m)^{1/q}$ &   & case dual to (a)\\
				& (b)   &  $p> q =1$& $1$ & $1$ & $1$ & Prop.~\ref{prop:gauss-q=1}
&  contraction principle \\
				& (b$'$)   &  $p=\infty >q $  & $1$ & $1$ & $1$ & 
&  case dual to (b) \\
			& (c)   &  $ p=1$       &    \multicolumn{3}{c}{different behavior,  confirming Conjecture~\ref{conj:conjecture_1}}   &Prop.~\ref{prop:gauss-p=1-q<2},   \eqref{eq:Matlak} 
&Gaussian concentration, Lem.~\ref{lem:vH2017-3} and \ref{lem:vH2017-4} \\	& (c')   &   $q =\infty$     &    \multicolumn{3}{c}{different behavior,  confirming Conjecture~\ref{conj:conjecture_1}}   &Cor.~\ref{cor:gauss-p>2-q=infty}    
& case dual to (c) \\
				bounded  &(a)  & $ p\leq 2$ & $1$ & $\ln(n)^{1/p^*}$ & $\ln(n)^{1/2+1/p^*}$ &  Cor.~\ref{prop:bounded-p-smaller-2-nopsi}  &set $K$ +  Talagrand's convex concentration \\		
				 &(a$'$)  & $ q\geq 2$ & $ 1$ & $\ln(m)^{1/2+1/q}$ & $\ln(m)^{1/q}$  &     & case dual to (a)\\		
			 &(b)  & $ p\leq 2 \leq q$ & $1$ & $\ln(n)^{1/p^*}$ & $\ln(n)^{1/q+1/p^*}$ &  Prop.~\ref{prop:bdd-p-2-q-OLD}  &set $K$ +  method {\`a} la \cite{BGN} \\	
				$\psi_r$, $r\leq 2$ & (a)    &  $p \leq 2$           &  $1$  &   $\ln(n)^{1/p^*}$      &  $  \ln(n)^{1/2+ 1/p^*} \ln(mn)^{1/r} $     &    Thm.~\ref{thm:psi-p<2}
& \makecell[tl]{set $K$ +  Talagrand's convex concentration,\\  cut-off at lvl $ \ln(mn)^{1/r}$}     \\
				&  (b)   & $p \leq 2$   &       $\ln(mn)^{1/r-1/2}$ &     $\ln(n)^{1/p^*}$      &  $  \ln(n)^{1/2+ 1/p^*}  $     &   Cor.~\ref{prop:bounded-p-smaller-2-nopsi}
& coupling, uses Gaussian and bounded cases  \\
				& (a$'$) &  $q \geq 2$  &        $1$ &     $ \ln(m)^{1/2+1/q}\ln(mn)^{1/r} $    & $\ln(m)^{1/q}$ &        &  case dual to (a) \\
				& (b$'$) &   $q \geq 2$   &        $\ln(mn)^{1/r-1/2}$ &     $ \ln(m)^{1/2+1/q} $    & $\ln(m)^{1/q}$     &
& case dual to (b)  \\
				\bottomrule
			\end{tabular}
			\label{table:summary} 
		}
		\end{table}
	\end{landscape}
	\clearpage
}
%
%

%
\subsection{Results for particular ranges  of \texorpdfstring{$p$, $q$}{p, q}}
\label{subsec:particular-ranges}
%

We continue with results for some specific ranges of $p$, $q$,
where we are able to prove estimates with better logarithmic
dependence
(results which follow from them by duality \eqref{eq:duality} are stated in Table~\ref{table:summary} to keep the presentation short).
We postpone their proofs to Section \ref{sec:particular-ranges}.
We start with the case of Gaussian random variables.
Recall that $\gamma_q = (\EE |g|^q)^{1/q}$, where $g$ is  a standard Gaussian random variable.

\begin{proposition}
	\label{prop:gauss-p-smaller-2}
For all $1\leq p\leq 2$ and $1\leq q< \infty$, we have
\begin{align}
\label{eq:p-smaller-2-assertion}
\EE \|G_A \colon \ell^n_p\to \ell^m_q\|
& \leq \gamma_q \ln(en)^{1/p^*} \|A\hadprod A \colon \ell^n_{p/2} \to \ell^m_{q/2}\|^{1/2}\\
 &\qquad + 2.2\ln(en)^{1/2+1/p^*}  \|(A\hadprod A)^T \colon \ell^m_{q^*/2} \to \ell^n_{p^*/2}\|^{1/2}. \nonumber
 \end{align}
\end{proposition}

 If $q=1$ or $p=\infty$, then we are able to get a result without logarithmic terms.
Recall that for a sequence $(x_j)_{j\le n}$  we denote by $(x_j^{\rearr{}})_{j\le n}$ the non-increasing rearrangement of $(|x_j|)_{j\le n}$.	

\begin{proposition}
	\phantomsection\label{prop:gauss-q=1} 
	\hfill
	\begin{enumerate}[(i)]
		\item\label{item:prop_q=1_i}
	For $1< p \leq \infty$, we have
		\begin{multline*}
		\|A\hadprod A \colon \ell^n_{p/2} \to \ell^m_{1/2}\|^{1/2} +\|(A\hadprod A)^T \colon \ell^m_{\infty} \to \ell^n_{p^*/2}\|^{1/2}
		\lesssim \EE \|G_A \colon \ell^n_p\to \ell^m_1\|	
		\\
		\leq \gamma_1\|A\hadprod A \colon \ell^n_{p/2} \to \ell^m_{1/2}\|^{1/2} +  2 \gamma_{p^*}  \|(A\hadprod A)^T \colon \ell^m_{\infty} \to \ell^n_{p^*/2}\|^{1/2}.
		\end{multline*}
		\item\label{item:prop_q=1_ii}
		 Moreover,
\begin{equation*}
		\EE \|G_A \colon \ell^n_1\to \ell^m_1\| \asymp \|A\hadprod A \colon \ell^n_{1/2} \to \ell^m_{1/2}\|^{1/2} +  \max_{j\le n}\sqrt{\ln(j+1)} b_j^{\rearr{}},
	\end{equation*}
where $b_j \coloneqq \|(a_{ij})_{i\le m}\|_2$, $j\leq n$.
	 \end{enumerate}
\end{proposition}

Note that \ref{item:prop_q=1_ii} shows in particular that a blow up of the constant $\gamma_{p^\ast}$ in the upper estimate (i) for $p\to 1$  is necessary, since the right most summands   in   \ref{item:prop_q=1_i} and \ref{item:prop_q=1_ii} are non-comparable.

\begin{remark}	\label{rem:assumpts_q=1}
	It shall be clear from the proof that the upper bound in part \ref{item:prop_q=1_i} of Proposition~\ref{prop:gauss-q=1} remains valid for any random matrix $X$ (instead of $G$) with independent  isotropic rows (i.e., rows with mean zero and the covariance matrix equal to the identity) such that
	\begin{equation}	\label{eq:rem:assumpts_q=1}
	\Bigr(  \EE \Bigl|\sum_{i=1}^m \alpha_{i}X_{ij}\Bigr|^{p^*}\Bigr)^{1/p^*}
	\lesssim_p \Bigl(\sum_{i=1}^m \alpha_{i}^2\Bigr)^{1/2} \qquad \text{for all } \alpha \in \RR^m, j\le n.
	\end{equation}
	 Note that the independence and the isotropicity of rows 
	 imply that also the columns of $X$ are isotropic (since the coordinates of every column are independent and have mean zero and  variance $1$). Therefore,  whenever $p\geq 2$, condition \eqref{eq:rem:assumpts_q=1} is always satisfied (because the $p^\ast$-integral norm is bounded above by the $2$-integral norm, which is then equal to the right-hand side of \eqref{eq:rem:assumpts_q=1}, since the covariance matrix of each column is equal to the $m\times m$ identity matrix).
	\end{remark}

The following proposition generalizes part (ii) of Proposition \ref{prop:gauss-q=1} to an arbitrary $q \le   2$. We list it separately since we present a proof using different arguments. Recall that the case $p=1$, $q\ge 2$ was established before, see \eqref{eq:Matlak}.

\begin{proposition}\label{prop:gauss-p=1-q<2}
If $1\leq q \leq 2$, then
\begin{align*}
  \|G_A\colon \ell_1^n \to \ell_q^m\|
  &\asymp \|A\hadprod A\colon \ell_{1/2}^n \to \ell_{q/2}^m\|^{1/2} + \max_{j\le n} (\sqrt{\ln(j+1)} b_j^{\rearr{}})\\
  & = \max_{j\le n} \|(a_{ij})_{i\le m}\|_q + \max_{j\le n} (\sqrt{\ln(j+1)} b_j^{\rearr{}}),
\end{align*}
where $b_j = \|(a_{ij})_{i\le m}\|_{2q/(2-q)}$ for $j\leq n$. 
\end{proposition}

Proposition~\ref{prop:gauss-p=1-q<2} immediately  implies its dual version.

\begin{corollary}	\label{cor:gauss-p>2-q=infty}
If $2\leq p \leq \infty$, then
\begin{align*}
  \|G_A\colon \ell_p^n \to \ell_\infty^m\|
  &\asymp \|(A\hadprod A)^T\colon \ell_{1/2}^m \to \ell_{p^\ast /2}^n\|^{1/2} + \max_{i\le m} (\sqrt{\ln(i+1)} d_i^{\rearr{}})\\
  & = \max_{i\le m} \|(a_{ij})_{j\le n}\|_{p^\ast} + \max_{i\le m} (\sqrt{\ln(i+1)} d_i^{\rearr{}}),
\end{align*}
where $d_i= \|(a_{ij})_{j\le n}\|_{2p^\ast/(2-p^\ast)} = \|(a_{ij})_{j\le n}\|_{2p/(p-2)} $ for $i\leq m$.
\end{corollary}

\begin{remark}
	\label{rem:q=infty}
	Corollary~\ref{cor:gauss-p>2-q=infty} and the dual version of~\eqref{eq:Matlak} provide the exact behavior of expected norm of Gaussian operator from $\ell_p^n$ to $\ell_q^m$ not only when $q=\infty$, but also for $q\ge c_0 \ln m$, as we explain now. For all $q\ge q_0\coloneqq c_0 \ln m$ we have the following inequalities for norms on $\RR^m$,
	\[
	  \|\cdot\|_q\geq \|\cdot\|_\infty \geq m^{-1/q_0} \| \cdot \|_{q_0} = e^{-1/c_0}  \|\cdot\|_{q_0}  \geq e^{-1/c_0}  \|\cdot\|_q ,
	\]
	therefore,
\[	
	\frac{1}{e^{1/c_0}}\EE \|X_A\colon \ell_{p}^n\to\ell_{q}^{m}\|
		 \leq \EE \|X_A\colon \ell_{p}^n\to\ell_\infty^{m}\|
		 \leq  \EE \|X_A\colon \ell_{p}^n\to\ell_{q}^{m}\|.
\]
Similarly,
\begin{align*}
	  \|(A\hadprod A)^T \colon \ell^m_{q^*/2} \to \ell^n_{p^*/2}\| & \asymp_{c_0}   \|(A\hadprod A)^T \colon \ell^m_{1/2} \to \ell^n_{p^*/2}\| .
\end{align*}
\end{remark}

Proposition~\ref{prop:gauss-p-smaller-2} implies the following estimate for matrices with independent $\psi_r$ entries,
in the same way as Corollary~\ref{cor:main-gauss} implies  Corollary \ref{thm:main-psi-r-introduction} (see~Subsection~\ref{subsect:coupling}).

\begin{corollary}
	\label{prop:bounded-p-smaller-2-nopsi}
	Assume that $K,L >0$, $r\in(0,2]$, and $X=(X_{ij})_{i\leq m, j\leq n}$ has independent mean-zero entries satisfying
	\begin{align}\label{eq:psi r random variables2}
	\PP(|X_{ij}|\ge t) \le Ke^{-t^r/L} \qquad \text{for all } t\ge 0.
	\end{align}
	Then, for $1\leq p\leq 2$, $1\leq q \leq \infty$,
\begin{align*}
\EE \|X_A \colon \ell^n_p\to \ell^m_q\|
& \lesssim_{r,K,L}  (\ln n)^{1/p^*}\ln(nm)^{1/r - 1/2} \|A\hadprod A \colon \ell^n_{p/2} \to \ell^m_{q/2}\|^{1/2}\\
&\qquad +  (\ln n)^{1/2+1/p^*}\ln(nm)^{1/r - 1/2}   \|(A\hadprod A)^T \colon \ell^m_{q^*/2} \to \ell^n_{p^*/2}\|^{1/2}. \nonumber
 \end{align*}
\end{corollary}

By Hoeffding's inequality (i.e., Lemma~\ref{lem:BGN-1-Hoeffding}) we know that matrices with independent valued in $[-1,1]$  entries having  mean zero satisfy \eqref{eq:psi r random variables2} with $r=2$ and $K=2=L$.
In this special case of independent bounded random variables
 one can also adapt the methods of~\cite{BGN}
to prove in the smaller range $1\leq p\leq 2 \leq q < \infty$ the following result  with explicit numerical constants
and improved dependence on $n$
(note that the second logarithmic term is better than in Corollary~\ref{prop:bounded-p-smaller-2-nopsi}, where the exponent equals $1/2+1/p^*$).

\begin{proposition}
	\label{prop:bdd-p-2-q-OLD}
	Assume that $X=(X_{ij})_{i\leq m, j\leq n}$ has independent mean-zero entries taking values in $[-1,1]$.
	Then, for $1\leq p\leq 2\leq q< \infty$,
	\begin{align*}
	\EE \|X_A \colon \ell^n_p\to \ell^m_q\|
	& \leq C(q) \ln(en)^{1/p^*} \|A\hadprod A \colon \ell^n_{p/2} \to \ell^m_{q/2}\|^{1/2}  \\
	&\quad + 10^{1/q} \ln(en)^{1/q+1/p^*}  \|(A\hadprod A)^T \colon \ell^m_{q^*/2} \to \ell^n_{p^*/2}\|^{1/2} ,
	\end{align*}
	where $C(q) \coloneqq  2 (q\Gamma(q/2))^{1/q} \asymp \sqrt q$.
\end{proposition}

Finally, we have the following general result for matrices with independent $\psi_r$ entries
(cf. Corollary~\ref{thm:main-psi-r-introduction}).

\begin{theorem}
	\label{thm:psi-p<2}
	Let $K,L>0$, $r\in(0,2]$,	 and assume that	$X=(X_{ij})_{i\leq m, j\leq n}$ has independent mean-zero entries satisfying
		\[
		\PP(|X_{ij}|\ge t) \le Ke^{-t^r/L} \quad \text{for all } t\ge 0.
		\]
			Then, for all $1\leq p\leq 2$ and $1\leq q< \infty$,
	\begin{align*}
	\EE \|X_A \colon \ell^n_p\to \ell^m_q\|
	&\lesssim_{r,K,L} \   q^{1/r} (\ln n)^{1/p^*}  \|A\hadprod A \colon \ell^n_{p/2} \to \ell^m_{q/2}\|^{1/2}\\
	 &\qquad\ \ \ + (\ln n)^{1/2+1/p^*} \ln (mn)^{1/r}  \|(A\hadprod A)^T \colon \ell^m_{q^*/2} \to \ell^n_{p^*/2}\|^{1/2}.
	\end{align*}
\end{theorem}

Having in mind the strategy of proof described after Theorem \ref{thm:main-gauss-sets}, let us elaborate on the idea of proof of Theorem \ref{thm:psi-p<2}. We shall split the matrix $X$ into two parts $X^{(1)}$ and $X^{(2)}$ which we treat separately. In our decomposition, all entries of $X^{(1)}$ are bounded by $C\ln(mn)^{1/r}$ and the probability that $X^{(2)} \neq 0$ is very small. Then we shall deal with $X^{(2)}$ using a crude bound (Lemma \ref{lem:cm-prelim-bound-2}) and the fact that the probability that $X^{(2)} \neq 0$ is small enough to compensate it. In order to bound the expectation of the norm of $X^{(1)}$, we require a cut-off version of Theorem~\ref{thm:psi-p<2} (Lemma~\ref{prop:bounded-p-smaller-2}). To obtain it, we shall replace $B_p^n$ in the expression for the operator norm with a suitable polytope $K$ (and leave $\sup_{y\in B_{q^*}^m}$ as it is) and then apply a Gaussian-type concentration inequality to the function
$Z\mapsto F(Z) \coloneqq \|Z_A x\|_q$ for $x\in \Ext (K)$.

%
\subsection{Tail bounds}
%

All the bounds for $\EE\|X_A\colon \ell_p^n \to \ell_q^m\|$ provided in this work for random matrices $X$
also yield a tail bound for $\|X_A\colon \ell_p^n \to \ell_q^m\|$.
(It is clear from the proof of Proposition~\ref{prop:tail-bounds} --- see Subsection~\ref{subsect:coupling} --- that the same applies to the estimates for $\sup_{I_0, J_0} \|G_A: \ell_p^{J_0} \to \ell_q^{I_0} \| $, but we omit the details to keep the presentation clear.)

\begin{proposition}[Tail bound]
\label{prop:tail-bounds}
	Assume that $K,L \ge 1$, $r\in(0,2]$, $1\leq p,q \leq \infty$, and $\gamma\ge 1$.
	Fix a deterministic $m\times n$ matrix $A$ and assume that
	\[
	D \ge \|A\hadprod A \colon \ell^n_{p/2} \to \ell^m_{q/2}\|^{1/2} .
	\]
	 If for all random matrices  $X=(X_{ij})_{i\leq m, j\leq n}$ with independent mean-zero entries satisfying
	\begin{equation}	\label{eq:tail-bound-assumption}
	\PP(|X_{ij}|\ge t) \le Ke^{-t^r/L} \qquad \text{for all } t\ge 0,\, i\leq m,\, j\leq n,
	\end{equation}
	we have
	\begin{equation}		\label{eq:tail-bound-assumption-bound}
		\EE \|X_A \colon \ell^n_p\to \ell^m_q\|
		 \leq \gamma D,
	 \end{equation}
	 then, for all random matrices with independent  mean-zero entries satisfying \eqref{eq:tail-bound-assumption}, we also have
	 \begin{equation}	\label{eq:tail-bound-moments}
		\bigl( \EE \|X_A \colon \ell^n_p\to \ell^m_q\|^\rho  \bigr)^{1/\rho}
		 \lesssim_{r,K,L} \ \rho^{1/r}\gamma D
		  \qquad \text{for all } \rho \ge 1,
	 \end{equation}
	 and, for all $t>0$,
	 \begin{equation}	\label{eq:tail-bound}
		\PP \bigl(  \|X_A \colon \ell^n_p\to \ell^m_q\| \ge t\gamma D  \bigr)
		  \leq C(r,K,L)	\exp\bigl( -t^r/C(r,K,L)\bigr).
	 \end{equation}
\end{proposition}

Note that
random variables taking values in $[-1,1]$ satisfy condition \eqref{eq:tail-bound-assumption} with $r=2$, $K=e$, and $L=1$.
Thus, Proposition~\ref{prop:tail-bounds} applies also in the setting of bounded or Gaussian entries.

%
\subsection{Organization of the paper}
%
In Section~\ref{sec:preliminaries} we gather various preliminary results we shall use in the sequel. Section~\ref{sec:Slepian-bound} contains the proofs of the main results valid for all $p$, $q$
(i.e., Theorem~\ref{thm:main-gauss-sets} and its corollaries) and the tail bound from Proposition~\ref{prop:tail-bounds}.
In Section~\ref{sec:particular-ranges} we prove the results for specific choices/ranges of $p$, $q$.
In Section~\ref{sec:lower-bounds} we prove lower bounds on expected operator norms, showing in particular that  our estimates are optimal up to logarithmic factors. We also prove other results justifying the proposed form of Conjecture~\ref{conj:conjecture_1}. The last subsection of Section~\ref{sec:lower-bounds} is devoted to infinite dimensional Gaussian operators.

%
\section{Preliminaries}
\label{sec:preliminaries}
%

%
\subsection{General facts}
%

We start with some easy lemmas which will be used repeatedly throughout the paper.

\begin{lemma}
	\label{lem:special-norms}
	For any real $m\times n$ matrix $B = (b_{ij})_{i\leq m, j\leq n}$ and $0<r\leq 1\leq s\leq \infty$,
	we have
	\[
	\|B \colon \ell^n_{r} \to \ell^m_{s} \|
	=\|B\colon \ell^n_{1} \to \ell^m_{s} \|
	= \max_{j\leq n} \| (b_{ij})_{i=1}^m\|_s.
	\]
	Furthermore, for a real $m\times n$ matrix $A = (a_{ij})_{i\leq m, j\leq n}$ and $1\leq p\leq 2$,
	$p \leq q\leq \infty$,
		\[
	\|A\hadprod A \colon \ell^n_{p/2} \to \ell^m_{q/2} \|^{1/2}
	= \max_{j\leq n} \| (a_{ij})_{i=1}^m\|_q.
	\]
\end{lemma}

\begin{proof}
	Since $0<r\leq 1$, we have $\conv B_r^n = B_1^n$, where $\conv S$ denotes the convex hull of the set $S$.
	Moreover,  the extreme points of $B_1^n$ are the signed standard unit vectors, i.e., $\pm e_1,\dots,\pm e_n$, and  $z\mapsto \|z\|_s$ is a convex function (since $s\ge 1$). Thus,
	\[
	\sup_{x\in B_r^n} \|Bx\|_s  = \sup_{x\in \conv B_r^n} \|Bx\|_s  = \sup_{x\in B_1^n}\|Bx\|_s = \max_{1\leq j \leq n} \|Be_j\|_s = \max_{1\leq j \leq n} \|(b_{ij})_{i=1}^m\|_s.
	\]
	This immediately implies the result for the Hadamard product $A\circ A=:B$ if $1\leq p\leq 2\leq q\leq \infty$. 
	
	    If, on the other hand, $1\leq p\leq q\leq 2$, 
    then by the subadditivity of the function $t\mapsto |t|^{q/2}$,
    \begin{align*}
    \|A\hadprod A\colon \ell_{p/2}^n \to \ell_{q/2}^m\|^{q/2}
        &= \sup_{x \in B_{p/2}^n} \sum_{i=1}^m \Bigl|\sum_{j=1}^n a_{ij}^2 x_j\Bigr|^{q/2}
    \le \sup_{x \in B_{p/2}^n} \sum_{i=1}^m \sum_{j=1}^n |a_{ij}|^{q} |x_j|^{q/2} \\
    &= \|(|a_{ij}|^q)_{i\le m,j\le n}\colon \ell_{p/q}^n\to \ell_{1}^m\|
    = \max_{j\le n} \|(a_{ij})_{i\le m}\|_q^q,
    \end{align*}
    where in the last equality we used the  first part of the Lemma. 
    Since we clearly have 
    \begin{displaymath}
    \|A\hadprod A\colon \ell_{p/2}^n \to \ell_{q/2}^m\| \ge \max_{j\le n}\|(a_{ij}^2)_{i\le m}\|_{q/2} =  \max_{j\le n}\|(a_{ij})_{i\le m}\|_{q}^2,
    \end{displaymath}
    we thus obtain
    \begin{equation*}
    \|A\hadprod A\colon \ell_{p/2}^n \to \ell_{q/2}^m\|^{1/2} = \max_{j\le n}\|(a_{ij})_{i\le m}\|_{q}.\qedhere
    \end{equation*} 
\end{proof}

\begin{definition}	\label{def:uncond}
	A set $K\subset \RR^n$  is called unconditional, if for every $(x_j)_{j\leq n}\in K$ and every $(\varepsilon_j)_{j\le n} \in \{-1,1\}^n$ we have $(\varepsilon_ j x_j)_{j\leq n} \in K$.
\end{definition}

We shall use the following version of \cite[Lemma~2.1]{Riemer-Schuett}.

\begin{lemma}
	\label{lem:set-K}
	Assume that $1\leq p\leq \infty$, $n\in \NN$,
	and define the convex set
	\[
	K\coloneqq \conv \Bigl\{ \frac{1}{|J|^{1/p}}\bigl( \varepsilon_j \ind_{\{j\in J\}} \bigr)_{j=1}^n : J\subset\{1,\dots,n\}, J\neq \emptyset, (\varepsilon_j )_{j=1}^n\in \{-1,1\}^n\Bigr\}.
	\]
	Then $B_p^n \subset \ln(en)^{1/p^*} K$.
\end{lemma}

\begin{proof}
	Fix a vector $x=(x_1,\dots,x_n)\in\RR^n$.
	We want to prove that $\|x\|_K \leq \ln(en)^{1/p^*} \|x\|_p$,
	where 
	\[
	\|x\|_K = \inf\{\lambda>0 \colon  x\in \lambda K\}
	\] 
	denotes the norm generated by $K$, i.e., its Minkowski gauge.	
	Since both $K$ and $B_p^n$ are permutationally invariant and unconditional (see Definition~\ref{def:uncond}), we may and will assume that $x_1\geq \dots\geq x_n\geq 0$.
	If we put $x_{n+1}\coloneqq 0$, then
	\begin{align*}
		x = \sum_{j=1}^n x_j e_j = \sum_{j=1}^n (x_j-x_{j+1}) (e_1+\dots + e_j).
	\end{align*}
	Since $\| e_1+\dots+e_j\|_K = j^{1/p}$ for $1\leq j \leq n$,\footnote{Indeed, $j^{-1/p}(e_1+\dots+e_j)\in K$, so $\| e_1+\dots+e_j\|_K \leq j^{1/p}$; on the other hand, $K\subset B_p^n$, so $\| e_1+\dots+e_j\|_K \geq \| e_1+\dots+e_j\|_p = j^{1/p}$.}
	the triangle and H\"older inequalities yield
	\begin{align*}
		\|x\|_K &\leq  \sum_{j=1}^n (x_j-x_{j+1}) j^{1/p} = \sum_{j=1}^n x_j (j^{1/p}-(j-1)^{1/p}) \\
		&\leq  \sum_{j=1}^n x_j j^{1/p - 1}
		\leq \|x\|_p \Bigl(\sum_{j=1}^n \frac{1}{j}\Bigr)^{1/p^*} \leq  \|x\|_p \ln(en)^{1/p^*},
	\end{align*}
	where we also used the elementary estimates $j^{1/p}-(j-1)^{1/p}\leq j^{\frac1p-1}$
	and $\sum_{j=1}^n \frac{1}{j}\leq 1+\int_1^n \frac{1}{t} dt = \ln(en)$.
	This completes the proof.
\end{proof}

\begin{remark}
	The term $\ln(en)^{1/p^*}$ can be replaced by $1+\frac 1p \ln(en)^{1/p^*}$  by writing in the above proof
	\begin{align*}
		\sum_{j=1}^n x_j (j^{1/p}-(j-1)^{1/p})
		&\leq  x_1+\frac{1}{p}\sum_{j=2}^n x_j (j-1)^{\frac1p - 1} \leq
		\|x\|_p \Bigl(1+\frac 1p \Bigl(\sum_{j=1}^{n-1} \frac{1}{j}\Bigr)^{1/p^*} \Bigr).
	\end{align*}
	Here we used the estimates $j^{1/p}-(j-1)^{1/p}\leq\frac{1}{p} (j-1)^{\frac1p-1}$ for $j>1$
	(which follows from the concavity of the function $t\mapsto t^{1/p}$) and the trivial one $x_1\leq\|x\|_p$.	
\end{remark}

\begin{remark}
	The constant $(\ln n)^{1/p^\ast}$ in Lemma~\ref{lem:set-K} is sharp up to a constant depending on $p$ for every $1\le p<\infty$ (when $p=\infty$, $K=B_p^n$ and the constant depending on $p$ degenerates as $p \to \infty$). 
	More precisely, we shall prove that if $B_p^n \subset C(p,n) K$, then $C(p,n) \gtrsim_p (\ln n)^{1/p^\ast}$. 
	Note that $B_p^n \subset C(p,n) K$ if and only if 
	\begin{equation}
	\label{eq:log-optimal-set-K}
	\|\cdot\|_{p^\ast}\le C(p,n) \|\cdot\|_{K}^\ast,
	\end{equation} 
	where $\|\cdot\|_K^\ast$ is norm dual to $\|\cdot\|_K$.

	Let $\operatorname{Ext}K $ be the set of extreme points of $K$, and let $(y_j^{\rearr{}})_{j\leq n}$ be the non-increasing rearrangement of $(|y_j|)_{j\le n}$.
	For every $y\in \RR^n$,
	\begin{align*}
		\|y\|_K^\ast 
		= \sup_{x\in K} \sum_{j=1}^n x_jy_j 
		= \sup_{x\in \operatorname{Ext}K} \sum_{j=1}^n x_jy_j 
		&=\sup_{J\subset [n], J\neq \emptyset} \sum_{j\in J} |y_j| \frac{1}{|J|^{1/p}}
		\\
		&=\sup_{k \le n} \sum_{j=1}^k y_j^{\rearr{}}  \frac{1}{k^{1/p}}.
	\end{align*}
	Assume that $p^\ast \neq 1$ and put $y_j \coloneqq j^{-1/p^\ast}$. We get 
	\[
		\|y\|_K^\ast =\sup_{k\le n} \sum_{j=1}^k j^{-1/p^\ast}  \frac{1}{k^{1/p}} 
		\asymp_p
		\sup_{k\le n} k^{1-\frac 1{p^\ast}} \frac{1}{k^{1/p}} =1,
	\]
	whereas
	\[
		\|y\|_{p^\ast} = \Bigl( \sum_{j=1}^n j^{-1} \Bigr)^{1/p^\ast} \asymp (\ln n)^{1/p^\ast},
	\]
	so  inequality~\eqref{eq:log-optimal-set-K} yields that $C(p,n) \gtrsim_p (\ln n)^{1/p^\ast}$.
\end{remark}

We shall also need the following standard lemma (see, e.g.,
\cite[Section~1.3]{Ledoux}).
We will use the versions with $r=1$ and $r=2$.

\begin{lemma}
	\label{lem:mean-median}
	Let $Z$ be a nonnegative random variable.
	If there exist $a \geq 0$, $b, \alpha, \beta, s_0> 0$, and $r\geq 1$ such that
		\[
		\PP(Z\geq a + b s) \leq \alpha e^{-\beta s^r} \quad \text{for } s\geq s_0,
		\]
		then
		\[
		\EE Z \leq a + b \Bigl(s_0 + \alpha \frac{e^{-\beta s_0^r}}{r\beta s_0^{r-1}}\Bigr).
		\]
\end{lemma}

\begin{proof}
	Integration by parts yields
	\begin{align*}
		\EE Z &\leq a + bs_0 + \int_{a+b s_0}^\infty \PP(Z\geq u) du
		= a + bs_0 + b\int_{s_0}^\infty \PP(Z\geq a+b s) ds \\
		&\leq a + bs_0 + b\alpha\int_{s_0}^\infty e^{-\beta s^r} ds\\
		&\leq a + bs_0 + \frac{b\alpha}{r\beta s_0^{r-1}}\int_{s_0}^\infty r\beta s^{r-1} e^{-\beta s^r} ds
		= a + b \Bigl(s_0 + \alpha \frac{e^{-\beta s_0^r}}{r\beta s_0^{r-1}}\Bigr).
		\qedhere
	\end{align*}
\end{proof}

%
\subsection{Contraction principles}
%

Below we recall the well-known contraction
principle due to Kahane and its extension by Talagrand (see, e.g., \cite[Exercise 6.7.7]{V2018}  and \cite[Theorem~4.4 and the proof of Theorem~4.12]{Ledoux-Talagrand}).
\begin{lemma}[Contraction principle]\label{lem:contraction-principle}
Let $(X,\|\cdot\|)$ be a normed space, $n\in\NN$, and $\rho \ge 1$. Assume that $x_1,\dots,x_n\in X$ and $\alpha\coloneqq(\alpha_1,\dots,\alpha_n)\in\RR^n$. Then, if $\varepsilon_1,\dots,\varepsilon_n$ are independent Rademacher random variables, we have
\[
\EE\bigl\|\sum_{i=1}^n \alpha_i\varepsilon_ix_i \bigr\|^{\rho}  \leq \|\alpha\|_\infty^{ \rho}\, \EE \bigl\|\sum_{i=1}^n \varepsilon_ix_i \bigr\|^{ \rho}.
\]
\end{lemma}

\begin{lemma}[Contraction principle]\label{lem:contraction-principle2}
Let $T$ be a bounded subset of $\RR^n$. Assume that $\varphi_i:\RR\to\RR$ are $1$-Lipschitz and $\varphi_i(0)=0$ for $i=1,\ldots,n$.  Then, if $\varepsilon_1,\dots,\varepsilon_n$ are independent Rademacher random variables, we have
\[
\EE \sup_{t\in T}\sum_{i=1}^n \varepsilon_i\varphi_i(t_i) \leq  \EE \sup_{t\in T}\sum_{i=1}^n \varepsilon_it_i .
\]
\end{lemma}

%
\subsection{Gaussian random variables}
%

The following result is fundamental to the theory of Gaussian processes and referred to as Slepian's inequality or Slepian's lemma \cite{S1962}. We use the following (slightly adapted) version taken from \cite[Theorem 13.3]{BLM2013}.

\begin{lemma}[Slepian's lemma]
	\label{lem:Slepian}
	Let $(X_t)_{t\in T}$ and $(Y_t)_{t\in T}$  be two Gaussian random vectors satisfying $\EE[X_t]=\EE[Y_t]$ for all $t\in T$. Assume that, for all $s,t \in T$, we have $\EE[(X_s-X_t)^2] \leq \EE[(Y_s-Y_t)^2]$.
	Then
	\[
	\EE\sup_{t\in T} X_t   \leq \EE\sup_{t\in T} Y_t .
	\]
\end{lemma}

The next lemma is folklore.
We include a short proof of an estimate with specific constants
for the sake of completeness.

\begin{lemma}
	\label{lem:max-Gaussians}
Assume that $k\ge 2$ and let $g_i$,  $i\le k$, be standard Gaussian random variables (not necessarily independent).
Then
\begin{align*}
\EE \max_{1\leq i\leq k} g_i &\leq \sqrt{2 \ln k},\\
\EE \max_{1\leq i\leq k} |g_i| &\leq 2\sqrt{ \ln k}.
\end{align*}
\end{lemma}

\begin{proof}
Since the moment generating function of a Gaussian random variable is given by $\EE e^{tg_1} = e^{t^2/2}$, it follows from Jensen's inequality that
\begin{align*}
	\EE \max_{i\le k} g_i &\le  \frac 1t \ln\bigl( \EE \exp(t\max_{i\le k } g_i) \bigr)\\
	& \le \frac 1t \ln\bigl(\EE \sum_{i=1}^k \exp(t g_i)\bigr)
	=  \frac 1t \ln\bigl(k e^{t^2/2}\bigr)
	= \frac{\ln k}t +\frac t2.
\end{align*}
By taking $t=\sqrt{2\ln k}$, we get the first assertion.
We apply this inequality with random variables $g_1, -g_1, \ldots, g_k, -g_k$ to get the second assertion, namely
\[
	\EE\max_{i\le k}|g_i| = \EE \max_{i\le k} \max\{g_i, -g_i\}\le\sqrt{2\ln(2k)} \le\sqrt{2\ln(k^2)}=2\sqrt{\ln k}.\qedhere
\]
\end{proof}

The next two lemmas are taken from \cite{vH2017}. Recall that $b_1^{\rearr{}} \geq \ldots \geq b_n^{\rearr{}}$ is the non-increasing rearrangement of $(|b_j|)_{j\le n}$.

\begin{lemma}[{\cite[Lemma~2.3]{vH2017}}]	\label{lem:vH2017-3}
	Assume that $(b_j)_{j\le n}\in\RR^n$ and let $(X_j)_{j\le n}$ be random variables (not necessarily independent) satisfying
	\[
		\PP(|X_j|>t)\le Ke^{-t^2/b_j^2} \qquad \text{for all } t\ge 0,\ j\le n.
	\]
	Then
	\[
		\EE \max_{j\le n} |X_j| \lesssim_{K} \max_{j\le n} b_j^{\rearr{}} \sqrt{\ln(j+1)}.
	\]
	\end{lemma}

\begin{lemma}[{\cite[Lemma~2.4]{vH2017}}]	\label{lem:vH2017-4}
	Assume that $(b_j)_{j\le n}\in\RR^n$ and let $(X_j)_{j\le n}$ be independent random variables with $X_j \sim \mathcal{N}(0,b_j^{2})$ for $j\leq n$. Then
	\[
		\EE \max_{j\le n} |X_j| \gtrsim \max_{j\le n} b_j^{\rearr{}} \sqrt{\ln(j+1)}.
	\]
\end{lemma}

\begin{lemma}[{Hoeffding's inequality, \cite[Theorem~2]{Hoeffding63}}]
	\label{lem:BGN-1-Hoeffding}
	Assume that $(b_j)_{j\le n}\in\RR^n$
	and let $X_j$, $j\leq n$, be independent mean-zero random variables such that $|X_j|\leq 1$ a.s.
	Then, for all $t\geq 0$,
	\[
	\PP\bigl(\bigl|\sum_{j=1}^n b_j X_j \bigr|\geq t \bigr)
	\leq 2 \exp\Bigl( - \frac{t^2}{2  \sum_{j=1}^n b_j^2 } \Bigr).
	\]
\end{lemma}

%
\subsection{Random variables with heavy tails}
\label{subsec:heavy-tails}
%
The following lemma is a special case of \cite[Theorem~1]{Kwapien87}.
\begin{lemma}[Contraction principle]
	\label{lem:contraction-Kwapien}
	Let $K, L>0$ and assume that $(\eta_i)_{i\le n}$ and $(\xi_i)_{i\le n}$ are two sequences of independent symmetric random variables satisfying for every $i\le n$ and  $t\ge 0$,
	\[
		\PP(|\eta_i| \ge t) \le K \PP(L|\xi_i|\ge t).
	\]
	Then, for every convex function $\varphi$ and every $a_1,\ldots, a_n \in \RR$,
	\[
		\EE\varphi\Big(\sum_{i=1}^n a_i \eta_i\Big)\le \EE\varphi\Big(KL\sum_{i=1}^n a_i\xi_i\Big).
	\]
\end{lemma}

\begin{lemma}[{\cite[Theorem~6.2]{HM-SO}}]
	\label{lem:Weibull-moments}
	Assume that $Z_1, \dots, Z_n$ are independent symmetric Weibull random variables with shape parameter $r\in(0,1]$ and scale parameter~$1$,
	i.e., $\PP(|Z_i|\ge t)=e^{-t^r}$ for $t\geq 0$.
	Then, for every $\prho{}\ge 2$ and $a\in \RR^n$,
	\[
		\Bigl\| \sum_{i=1}^n a_iZ_i \Bigr\|_{\prho{}} \asymp \max\bigl\{ \sqrt{\prho{}} \|a\|_2\|Z_1\|_2 , \|a\|_{\prho{}}\|Z_1\|_{\prho{}} \bigr\}.
	\]
\end{lemma}

\begin{remark}[Moments of Weibull random variables]
	\label{rmk:Weibull-moments}
	Note that if $Z$ is a symmetric random variable such that $\PP(|Z|\ge t)=e^{-t^r}$, $r\in(0,2]$, then $Y=|Z|^{r}\sgn(Z)$ has (symmetric) exponential distribution with parameter $1$, so by Stirling's formula we obtain, for all $\rho \ge 1$,
	\[
		\|Z\|_{\prho{}} = \|Y\|_{\prho{}/r}^{1/r} = \Gamma\Bigl(\frac{\prho{}}r +1\Bigr)^{1/{\prho{}}} \leq  \Bigl(\frac Cr\Bigr)^{\frac 1r + \frac 1{2\prho{}}} \prho{}^{1/r} \le  \Bigl(\frac Cr\Bigr)^{\frac 1r + \frac 1{2}} \prho{}^{1/r} ,
	\]
	with $C\ge 1$.	
	\end{remark}

The  three previous results easily imply the following estimate for integral norms of linear combinations of independent $\psi_r$ random variables.

\begin{proposition}	\label{prop:psi_prelim4}
	Let $K, L>0$, $r\in(0,1]$ and assume that $Z_1, \dots, Z_n$  are independent  symmetric random variables satisfying $\PP(|Z_i|\ge t) \le Ke^{-t^r/L}$ for all $t\ge 0$ and $i\leq n$. Then, for every $\prho{}\ge 2$ and $a\in \RR^n$,
	\begin{align*}
		\Bigl\| \sum_{i=1}^n a_iZ_i \Bigr\|_{\prho{}}
		& \lesssim (C/r)^{\frac 1r + \frac 12} KL^{1/r}\max\bigl\{ \sqrt{\prho{}} \|a\|_2 , \prho{}^{1/r} \|a\|_{\prho{}} \bigr\}
		\\
		 &\lesssim (C'/r)^{\frac 1r + \frac 12}  KL^{1/r}\max\bigl\{ \sqrt{\prho{}} \|a\|_2 , \prho{}^{1/r} \|a\|_\infty \bigr\}.
	\end{align*}
\end{proposition}

\begin{proof}
	The first inequality is an immediate consequence of Lemma~\ref{lem:contraction-Kwapien}
	(applied with $\eta_i=Z_i$, independent Weibull variables $\xi_i$ with shape parameter $r$ and scale parameter $1$, and with the convex function $\varphi:t\mapsto |t|^{\rho}$),
	Lemma~\ref{lem:Weibull-moments},
	and Remark~\ref{rmk:Weibull-moments}.
	The second inequality follows from
	\begin{align*}
		\|a\|_{\prho{}}\le \|a\|_2^{2/\prho{}}\|a\|_\infty^{1-2/{\prho{}}}
		&= \prho{}^{\frac{2}{\prho{}r}}\|\prho{}^{-1/r}a\|_2^{2/\prho{}}\|a\|_\infty^{1-2/{\prho{}}}
		\\
		&\le {\prho{}}^{\frac 2{\prho{}r}}\Bigl(\frac 2{\prho{}^{1+1/r}} \|a\|_2 + \Bigl(1-2/{\prho{}}\Bigr) \|a\|_\infty \Bigr),
	\end{align*}
	where in the last step we used the inequality between weighted arithmetic and geometric means.
\end{proof}

The next lemma is standard and provides us with several equivalent formulations of the $\psi_r$ property
expressed through tail bounds,
growth of moments,
and the exponential moments, respectively.
We provide a brief proof, since in the literature one usually finds versions  for $r\ge 1$ only.

\begin{lemma}
	\label{lem:psi-r-equivalence}
	Assume that  $r\in(0,2]$.
	Let $Z$ be a non-negative random variable. The following conditions are equivalent:
	\begin{enumerate}[(i)]
		\item\label{item:psi-r-i}
		There exist $K_1,L_1>0$  such that
		\[
		\PP(Z\geq t) \leq K_1 e^{- t^r/ L_1} \quad \text{for all } t\geq 0.
		\]
		\item\label{item:psi-r-ii}
		There exists $K_2$ such that
		\[
		\|Z\|_{\prho{}} \leq K_2 \prho{}^{1/r} \quad \text{for all } \prho{} \geq 1.
		\]
		\item\label{item:psi-r-iii}
		There exist $K_3, u>0$ such that
		\[
		\EE\exp(u Z^r) \leq K_3.
		\]
	\end{enumerate}	
Here,
\ref{item:psi-r-i} implies \ref{item:psi-r-ii} with $K_2 = C(r)K_1L_1^{1/r}$,
\ref{item:psi-r-ii} implies \ref{item:psi-r-iii} with $K_3 = 1+e^{(2e r)^{-1}}$, $u=(2erK_2^r)^{-1}$,
 and
\ref{item:psi-r-iii} implies \ref{item:psi-r-i} with $K_1 = K_3$, $L_1=u^{-1}$.
	\end{lemma}

\begin{proof}
	Property~\ref{item:psi-r-i} implies \ref{item:psi-r-ii} by Lemma \ref{lem:contraction-Kwapien} (applied with $n=1$, $\eta_1=Z$ and an independent Weibull variable $\xi_1$ with parameter $r$) and Remark \ref{rmk:Weibull-moments}. Property~\ref{item:psi-r-iii} implies \ref{item:psi-r-i} by Chebyshev's inequality:
	\[
		\PP(Z\geq t)=\PP\bigl(\exp(uZ^r)\geq \exp(ut^r)\bigr) \le K_3	 \exp(-ut^r).
	\]
	Assume now that \ref{item:psi-r-ii} holds and denote  $k_0 = \lfloor \frac 1r \rfloor$.
	Then, for every $k\in[1,k_0]$, we have $kr\leq 1$ and
    \[
    \EE Z^{kr}  \leq (\EE Z\bigr)^{kr} \leq K_2^{kr},
	\]
	while for  $k\geq k_0+1 $, we have $kr \geq 1$ and, hence, property~\ref{item:psi-r-ii} yields
	\[
	\EE Z^{kr} \leq K_2^{kr}(kr)^k.
	\]
    Hence, by  Stirling's formula we have for $u=(2erK_2^r)^{-1}$,
		\begin{align*}
		\EE \exp(uZ^r)
		&= 1 + \sum_{k=1}^{k_0} \frac{u^k \EE Z^{kr}}{k!}  + \sum_{k=k_0+1}^\infty \frac{u^k \EE Z^{kr}}{k!} \\
		&\leq 1 + \sum_{k=1}^{k_0} \frac{u^k K_2^{kr}}{k!}  + \sum_{k=k_0+1}^\infty \frac{u^k K_2^{kr}(kr)^k }{\bigl(k/e\bigr)^k} \\
		&= 1 + \sum_{k=1}^{k_0} \frac{u^kK_2^{kr}}{k!}  + \sum_{k=k_0+1}^\infty 2^{-k}
		\leq e^{uK_2^r} +1. \qedhere
		\end{align*}	
\end{proof}

The next lemma states that a linear combination of independent $\psi_r$ random variables is a $\psi_r$ random variable.

\begin{lemma}	\label{lem:psi_prelim5}
	Assume that $u>0$, $r\in (0,2]$,
	and let $(Z_i)_{i\le k}$ be independent symmetric random variables satisfying  $\PP(|Z_i|\ge t) \le Ke^{-t^r/L}$ for all $t\ge 0$.
	Then for every $a\in \RR^k$
	the  random variable $Y\coloneqq \|a\|_2^{-1}\sum_{i=1}^k a_iZ_i$ satisfies, for all $t\ge 0$,
	$$\PP(|Y|\ge t) \le K'e^{-t^r/L'},$$
	where $K'$, $L'$ depend only on $K$, $L$, and $r$.
\end{lemma}

\begin{proof}
	The case $r\ge 1$ is standard (see, e.g., \cite[Theorem~1.2.5]{CGLP2012}), therefore we skip a proof in this case
	(however, in order to prove the lemma in the case $r\ge 1$ it suffices to use the result of Gluskin and Kwapie{\'n}~\cite{GluskinKwapien95} (together with Lemma~\ref{lem:contraction-Kwapien}) instead of Lemma~\ref{lem:Weibull-moments} in the proof below).
	
	Assume that $r\in (0,1]$ and recall that $Y=\|a\|_2^{-1}\sum_{i=1}^k a_iZ_i$. By Proposition \ref{prop:psi_prelim4},
	\[
		\|Y\|_{\prho{}} \lesssim_{K,L,r} \max\{ \sqrt{\prho{}} , \prho{}^{1/r} \} = \prho{}^{1/r} \qquad \text{for all } \prho{} \ge1.
	\]
	Hence, Lemma \ref{lem:psi-r-equivalence} yields the assertion.
\end{proof}

\begin{lemma} \label{lem:psi_prelim7}
	Assume that $r\in(0,2]$, $\frac{1}{s}\coloneqq\frac{1}{r}-\frac 12$,
	$Y$ is a non-negative random variable such that $\PP(Y\ge t)=e^{-t^s}$ for all  $t\ge 0$,
	and $g\sim \mathcal{N}(0,1)$ is independent of~$Y$.
	Then, for every $t\ge 0$,
	\[
		\PP\bigl(|g|Y\ge t\bigr)\ge c e^{-4t^r},
	\]
	where $c:=\sqrt{2/\pi}e^{-2}$.
\end{lemma}

\begin{proof}
	 In the case $r=2$ we have $s=\infty$ and then $Y=1$ almost surely and the assertion is trivial. Assume now that $r<2$. By our assumptions $r=\frac{2s}{2+s}$.
	Let $x_0\coloneqq(2t^s)^{1/(2+s)}$.
	Note that  $x\ge x_0$ is equivalent to $\frac{t^s}{x^s}\le \frac{x^2}2$.
	Thus,
	\begin{align*} 
		\PP\bigl(|g|Y\ge t\bigr)
		&=\EE e^{-\frac{t^s}{|g|^s}} = \sqrt{\frac 2{\pi}} \int_0^\infty e^{-\frac{t^s}{x^s}-\frac{x^2}2} dx
		\ge  \sqrt{\frac 2{\pi}} \int_{x_0}^{x_0+1} e^{-\frac{t^s}{x^s}-\frac{x^2}2} dx
		\\ &
		\ge \sqrt{\frac 2{\pi}} \int_{x_0}^{x_0+1} e^{-x^2} dx
		\ge \sqrt{\frac 2{\pi}} e^{-(x_0+1)^2}
		\ge \sqrt{\frac 2{\pi}} e^{-2(x_0^2+1)}
		\\ &
		= c e^{-2x_0^2} \ge ce^{-4t^{2s/(2+s)}} = ce^{-4t^r},
	\end{align*}
	where we used $2^{{2/(2+s)}}\leq 2$ and chose $c\coloneqq\sqrt{2/\pi}e^{-2}$.
	\end{proof}

\begin{lemma} \label{lem:coupling}
	Assume that $K, L>0$, $r\in(0,2]$ 
	and that $Z$ is a  random variable satisfying $\PP(|Z|\ge t) \le Ke^{-t^r/L}$ for all $t\ge 0$.
	Let $Y$, $g$, and $c=\sqrt{2/\pi}e^{-2}$ be as in Lemma \ref{lem:psi_prelim7}.
	Then there exist random variables $U\sim |Z|$ and $V\sim |g|Y$ such that
	\begin{equation*}
		U\le  (8L)^{1/r}\Bigl(\Bigl(\frac{\ln(K/c)}4\Bigr)^{1/r} +V \Bigr) \quad \text{a.s.}
	\end{equation*}
\end{lemma}

\begin{proof}
For $t=0$ we have $1=\PP(|Z|\ge 0)\le K$, so $K\ge 1$, and thus $\ln(K/c)=\ln(Ke^2\sqrt{\pi/2})>0$.
We use our assumptions, the inequality $(a+b)^r\ge (a^r+b^r)/2$, and Lemma \ref{lem:psi_prelim7} to obtain for any $t\ge 0$,
	\begin{align*}
		\PP\Bigl((8L)^{-1/r}|Z| \ge t+ \bigl(\ln(K/c)/4\bigr)^{1/r} \Bigr)
		& \le K\exp\Bigl(-8\Bigl[t+ \bigl(\ln(K/c)/4\bigr)^{1/r}\Bigr]^r\Bigr)
		\\ &
		\le K \exp\Bigl(-4\bigl(t^r+ \ln(K/c)/4\bigr)\Bigr) = ce^{-4t^r}
		\\ &
		\le \PP\bigl(|g|Y\ge t\bigr).
	\end{align*}
Consider the version $U$ of $|Z|$ and the version $V$ of $|g|Y$  defined on the (common) probability space $(0, 1)$ equipped with Lebesgue measure,
constructed as the (generalised) inverses of cumulative distribution functions of $|Z|$ and $|g|Y$, respectively.
Then
$(8L)^{-1/r} U - \bigl(\ln(K/c)/4\bigr)^{1/r} \le V $, which implies the assertion.
\end{proof}

\begin{lemma} 	
	\label{lem:max-psi-r}
	Let $K,L>0$, $r\in (0,2]$ and  $k\ge 3$,
	and assume that  $(Z_i)_{i\le k}$, are  random variables satisfying $\PP(|Z_i|\ge t) \le Ke^{-t^r/L}$ for all $t\ge 0$. Then
	\[
		\PP\bigl(\max_{i\le k} |Z_i|\ge (vL \ln k)^{1/r}\bigr)\leq K k^{-v+1} \leq eKe^{-v} \qquad \text{for every }v\ge 1
	\]
	and
	\[
		\EE\max_{i\le k} |Z_i|  \lesssim \bigl( LK^r r^{-1} \ln k \bigr)^{1/r} \lesssim_{r,K,L} (\ln k)^{1/r}.
	\]
\end{lemma}

\begin{proof}
	By a union bound and the assumptions we get, for every $v\ge 1$,
	\begin{align*}
		\PP\bigl(\max_{i\le k} |Z_i|\ge (vL \ln k)^{1/r}\bigr)
		&\leq \sum_{i=1}^k \PP\bigl( |Z_i|\ge (vL \ln k)^{1/r}\bigr)
		\leq k\cdot Ke^{-v\ln k} \\
		&=Ke^{-(v-1)\ln k}
		= K k^{-v+1} \leq eKe^{-v},
	\end{align*}
	where we used $k\geq 3$ in the last step.
	We integrate by parts, change the variables, and use the above bound to obtain the second part of the assertion, i.e.,
	\begin{align*}
		\EE\max_{i\le k} |Z_i|
		&= \int_0^\infty \PP\bigl(\max_{i\le k} |Z_i|\ge u\bigr) du
		\le (L \ln k)^{1/r} + \int_{(L \ln k)^{1/r}}^\infty \hspace{-0.4 cm} \PP\bigl(\max_{i\le k} |Z_i|\ge u\bigr) du	\\
		&=(L \ln k)^{1/r} + \frac{(L \ln k)^{1/r} }r \int_{1}^\infty v^{\frac 1r -1}
		\PP\bigl(\max_{i\le k} |Z_i|\ge (vL \ln k)^{1/r}\bigr) dv		\\
		&\leq (L \ln k)^{1/r} \Bigl(1 + \frac {eK}r \int_{1}^\infty v^{\frac 1r -1}e^{-v} dv \Bigr) 	\\
		&\leq (L \ln k)^{1/r} \Bigl(1 + eK\ \Gamma\Bigl(\frac 1r +1\Bigr) \Bigr). 	\qedhere
	\end{align*}
\end{proof}

%
\section{Proofs of the main results}
\label{sec:Slepian-bound}
%

    After the preparation in the previous section, we shall now present the proofs of our main results.

\subsection{General bound via Slepian's lemma}
\label{subsect:proof-main-thm}

In order to obtain Theorem \ref{thm:main-gauss-sets} we first prove its weaker version, for $p=\infty$ and $q=1$ only.
After that we shall use the polytope $K$ from Lemma \ref{lem:set-K} 
and the Gaussian concentration to see how Proposition \ref{prop:auxilary-main-gauss} implies the general bound. The proof of this proposition relies on the symmetrization together with the contraction principle, which allow us to get rid of $y_i$ and $x_j$, and make use of Slepian's lemma.

\begin{proposition}
\label{prop:auxilary-main-gauss}
 	Assume that $G=(g_{ij})_{i\leq m, j\leq n}$ has i.i.d.\ standard Gaussian entries and $k \leq m$, $l\leq n$.
 	Then
	\begin{align*}	
\EE \sup_{I, J} \sup_{y\in B_\infty^m}\sup_{x\in B_\infty^n}
\sum_{i\in I, j\in J} y_i a_{ij} g_{ij} x_j
&\leq
\bigl(8\sqrt{\ln m} + \sqrt{2/\pi}\bigr)	\sup_{I, J}  \sum_{i\in I}  \sqrt{\sum_{j\in J} a_{ij}^2} \nonumber\\
& \qquad + \bigl(8\sqrt{\ln n} +2 \sqrt{2/\pi}\bigr)  \sup_{I, J} \sum_{j\in J} \sqrt{\sum_{i\in I} a_{ij}^2 },
	\end{align*}
where the suprema are taken over all sets
$I\subset \{1,\ldots,m\}$, $J\subset \{1,\ldots,n\}$ such that $|I| = k$, $|J| = l$.
\end{proposition}

\begin{proof}
	Throughout the proof,  $k \leq m$ and $l\leq n$ are fixed and the suprema are taken over all index sets satisfying
	$I\subset \{1,\ldots,m\}$, $|I| = k$
	and $J\subset \{1,\ldots,n\}$, $|J| = l$.

Let us denote by $(\widetilde{g}_{ij})_{i\leq m, j\leq n}$ an independent copy of $(g_{ij})_{i\leq m, j\leq n}$.
Using the duality  $(\ell_1^m)^*=\ell_\infty^m$, centering the expression, noticing that $\sum_{j\in J} a_{ij} \widetilde{g}_{ij} x_j $ is a Gaussian random variable with variance  $\sqrt{\sum_{j\in J} a_{ij}^2  x_j^2}$, and using Jensen's inequality, we see that
\begin{align}	
	\EE \sup_{I, J} & \sup_{x\in B_\infty^n} \sup_{y\in B_\infty^m}
	\sum_{i\in I, j\in J} y_i a_{ij} g_{ij} x_j
	= \EE \sup_{I, J} \sup_{x\in B_\infty^n} \sum_{i\in I} \Bigl| \sum_{j\in J} a_{ij} g_{ij} x_j \Bigr| \nonumber\\	
	&\leq  \EE \sup_{I, J} \sup_{x\in B_\infty^n} \sum_{i\in I} \Bigl( \Bigl| \sum_{j\in J} a_{ij} g_{ij} x_j \Bigr|  - \EE  \Bigl| \sum_{j\in J} a_{ij} \widetilde{g}_{ij} x_j \Bigr|\Bigr)
	+ \sup_{I, J} \sup_{x\in B_\infty^n} \sum_{i\in I} \EE  \Bigl| \sum_{j\in J} a_{ij} \widetilde{g}_{ij} x_j \Bigr| \nonumber\\
	&=  \EE \sup_{I, J} \sup_{x\in B_\infty^n} \sum_{i\in I} \Bigl( \Bigl| \sum_{j\in J} a_{ij} g_{ij} x_j \Bigr|  - \EE  \Bigl| \sum_{j\in J} a_{ij} \widetilde{g}_{ij} x_j \Bigr|\Bigr)
	+ \sup_{I, J} \sup_{x\in B_\infty^n} \sum_{i\in I} \sqrt{\sum_{j\in J} a_{ij}^2  x_j^2} \EE |g|\nonumber\\
	&\leq  \EE \sup_{I, J} \sup_{x\in B_\infty^n} \sum_{i\in I} \Bigl( \Bigl| \sum_{j\in J} a_{ij} g_{ij} x_j \Bigr|  -   \Bigl| \sum_{j\in J} a_{ij} \widetilde{g}_{ij} x_j \Bigr|\Bigr)
	+ \sqrt{\frac{2}{\pi}}\sup_{I, J} \sum_{i\in I} \sqrt{\sum_{j\in J} a_{ij}^2  } .
	\label{eq:I-J-oo-1-step-i-center}
\end{align}
To estimate the expected value on the right-hand side, we use a symmetrization trick
together with the contraction principle (Lemma \ref{lem:contraction-principle2}).
Let $(\varepsilon_i)_{i\leq m}$ be a sequence of independent Rademacher random variables independent of all others. Since the random vectors 
	\[
	Z_i = \Bigl(\ind_{\{i\in I\}} \Bigl( \bigl|\sum_{j\in J} a_{ij}g_{ij}x_j \bigr| - \bigl|\sum_{j\in J} a_{ij}\widetilde{g}_{ij}x_j \bigr| \Bigr) \Bigr)_{I\subset[m], J\subset [n], x\in B_\infty^n}
	\]
	(where $i\leq m$)
	are independent and symmetric, $(Z_i)_{i\le m}$ has the same distribution as $(\varepsilon_iZ_i)_{i\le m}$. Therefore,
 \begin{align}	
 	\MoveEqLeft[8]
	 	 \EE \sup_{I, J} \sup_{x\in B_\infty^n} \sum_{i\in I} \Bigl( \Bigl| \sum_{j\in J} a_{ij} g_{ij} x_j \Bigr|  -  \Bigl| \sum_{j\in J} a_{ij} \widetilde{g}_{ij} x_j \Bigr|\Bigr) \nonumber\\
 	 &= 	 \EE \sup_{I, J} \sup_{x\in B_\infty^n} \sum_{i\in I} \varepsilon_i \Bigl( \Bigl| \sum_{j\in J} a_{ij} g_{ij} x_j \Bigr|  -  \Bigl| \sum_{j\in J} a_{ij} \widetilde{g}_{ij} x_j \Bigr|\Bigr) \nonumber\\ 	
 	 &\leq 	2 \EE \sup_{I, J} \sup_{x\in B_\infty^n} \sum_{i\in I} \varepsilon_i \Bigl| \sum_{j\in J} a_{ij} g_{ij} x_j \Bigr|  \nonumber\\
 	 &= 	2 \EE \sup_{I, J} \sup_{x\in B_\infty^n} \sum_{i=1}^m \varepsilon_i \Bigl| \sum_{j\in J} a_{ij} g_{ij} x_j \ind_{\{i\in I\}}\Bigr| .
 	 	\label{eq:I-J-oo-1-step-ii-symm}  	 	
 \end{align}
Applying (conditionally, with the values of $g_{ij}$'s fixed) the contraction principle (i.e., Lemma~\ref{lem:contraction-principle2}) with  the set 
\[
	T=\biggl\{\Bigl(\sum_{j\in J} a_{ij} g_{ij} x_j \ind_{\{i\in I\}}\Bigr)_{i\le m} \colon I\subset [m], |I|=k, J\subset [n], |J|=l, x\in B_\infty^n \biggr\}
\]
and the function $u\mapsto |u|$ (which is 1-Lipschitz and takes the value $0$ at the origin), we get
  \begin{align}	\label{eq:I-J-oo-1-step2.5} \nonumber
 	\MoveEqLeft[4]
 	\EE \sup_{I, J} \sup_{x\in B_\infty^n} \sum_{i=1}^m \varepsilon_i \Bigl| \sum_{j\in J} a_{ij} g_{ij} x_j \ind_{\{i\in I\}}\Bigr|  
 	\leq
	\EE \sup_{I, J} \sup_{x\in B_\infty^n} \sum_{i=1}^m \varepsilon_i \sum_{j\in J} a_{ij} g_{ij} x_j \ind_{\{i\in I\}}
	\\
	&= \EE \sup_{I, J} \sup_{x\in B_\infty^n}\sum_{j\in J}  \sum_{i\in I} a_{ij}  \varepsilon_i g_{ij} x_j
	 = \EE \sup_{I, J} \sup_{x\in B_\infty^n}\sum_{j\in J}  \sum_{i\in I} a_{ij}  g_{ij} x_j .
\end{align}
By proceeding similarly as in \eqref{eq:I-J-oo-1-step-i-center}, we obtain
\begin{align}
\MoveEqLeft[4]
 	\EE \sup_{I, J} \sup_{x\in B_\infty^n}\sum_{j\in J}  \sum_{i\in I} a_{ij}  g_{ij} x_j
 	= \EE \sup_{I, J} \sum_{j\in J}  \Bigl| \sum_{i\in I} a_{ij}  g_{ij} \Bigr|\nonumber \\
 	&\leq  \EE \sup_{I, J}  \sum_{j\in J} \Bigl( \Bigl| \sum_{i\in I} a_{ij} g_{ij} \Bigr|  - \EE  \Bigl| \sum_{i\in I} a_{ij} \widetilde{g}_{ij}  \Bigr|\Bigr) + \sqrt{\frac{2}{\pi}} \sup_{I, J} \sum_{j\in J} \sqrt{\sum_{i\in I} a_{ij}^2} .
 	\label{eq:I-J-oo-1-step-iii-contraction}
 \end{align}
 
Observe that using symmetrization and the contraction principle similarly as in~\eqref{eq:I-J-oo-1-step-ii-symm} and \eqref{eq:I-J-oo-1-step2.5}, we can estimate the first summand on right-hand side of~\eqref{eq:I-J-oo-1-step-iii-contraction}
as follows,
 \begin{equation}	
	 \EE \sup_{I, J}  \sum_{j\in J} \Bigl( \Bigl| \sum_{i\in I} a_{ij} g_{ij} \Bigr|  - \EE  \Bigl| \sum_{i\in I} a_{ij} \widetilde{g}_{ij}  \Bigr|\Bigr)
	 \leq 	2 \EE \sup_{I, J} \sum_{i\in I} \sum_{j\in J} a_{ij} g_{ij}.
	\label{eq:I-J-oo-1-step-iv-repeat-symm}
\end{equation}
Altogether, the inequalities in \eqref{eq:I-J-oo-1-step-i-center} -- \eqref{eq:I-J-oo-1-step-iv-repeat-symm} yield that
	\begin{align}	
		\EE \sup_{I, J} \sup_{y\in B_\infty^m}\sup_{x\in B_\infty^n}
		\sum_{i\in I, j\in J} y_i a_{ij} g_{ij} x_j
		&\leq 4 \EE \sup_{I, J} \sum_{i\in I} \sum_{j\in J} a_{ij} g_{ij}
		 + 2 \sqrt{\frac{2}{\pi}} \sup_{I, J} \sum_{j\in J} \sqrt{\sum_{i\in I} a_{ij}^2} \nonumber\\
		 &\qquad  + \sqrt{\frac{2}{\pi}}\sup_{I, J}  \sum_{i\in I} \sqrt{\sum_{j\in J} a_{ij}^2 } .
    \label{eq:I-J-oo-1-step-v-sumup}
	\end{align}

We shall now estimate the first summand on the right-hand side of \eqref{eq:I-J-oo-1-step-v-sumup} using  Slepian's lemma (i.e., Lemma \ref{lem:Slepian}).
Denote
\begin{align*}
	X_{I,J} &\coloneqq \sum_{i\in I} \sum_{j\in J} a_{ij} g_{ij},\\
	Y_{I,J} &\coloneqq \sum_{i\in I} g_i \sqrt{\sum_{j\in J} a_{ij}^2 }
	                   + \sum_{j\in J} \widetilde{g}_j \sqrt{\sum_{i\in I} a_{ij}^2 },
\end{align*}
where $g_i, i=1,\ldots,m$, $ \widetilde{g}_j, j=1,\ldots,n$ are independent standard Gaussian variables.
The random variables $X_{I,J}, Y_{I,J}$ clearly have zero mean.
Thus,
we only need to calculate and compare $\EE (X_{I,J} - X_{\widetilde{I}, \widetilde{J}})^2$ and $\EE (Y_{I,J} - Y_{\widetilde{I}, \widetilde{J}})^2$.
In the calculations below it will be evident
over which sets the index $i$ (resp. $j$) runs,
so in order to shorten the notation and improve readability,
we use the notational convention
\[
\sum_{I} \coloneqq \sum_{i\in I},
\qquad \sum_{\widetilde{J}} \coloneqq \sum_{j\in\widetilde{J}},
\qquad
\sum_{I\cap \widetilde{I}, J\setminus\widetilde{J}}\coloneqq
\sum_{{i\in I\cap \widetilde{I}, j\in J\setminus\widetilde{J}}},
\qquad \text{etc.}
\]
By independence,
\begin{align*}
\EE (X_{I,J} - X_{\widetilde{I}, \widetilde{J}})^2
 &=  \sum_{I, J} a_{ij}^2
 + \sum_{\widetilde{I},  \widetilde{J}} a_{ij}^2
 - 2 \sum_{\mathclap{I\cap \widetilde{I}, J\cap \widetilde{J}} } a_{ij}^2\\
    &= \sum_{I, J} a_{ij}^2
   + \sum_{\widetilde{I},\widetilde{J}} a_{ij}^2
   -  \sum_{{I\cap \widetilde{I}, J}} a_{ij}^2
   -   \sum_{{I\cap \widetilde{I},\widetilde{J}} } a_{ij}^2  + \!\! \sum_{{I\cap \widetilde{I}, J\setminus\widetilde{J}}} a_{ij}^2
   +\!\! \sum_{{I\cap \widetilde{I},\widetilde{J}\setminus J}} a_{ij}^2.
\end{align*}
By independence and the inequality $2\sqrt{ab} \le a+b$ (valid for $a,b\geq 0$),
\begin{align*}
	\EE (Y_{I,J} - Y_{\widetilde{I}, \widetilde{J}})^2
	&= 2\sum_{I, J} a_{ij}^2
	+ 2\sum_{\widetilde{I}, \widetilde{J}} a_{ij}^2\\
	&\qquad 	
	- 2 \sum_{I\cap \widetilde{I}} \sqrt{\sum_{ J \vphantom{\widetilde{J} -- cludge for moving position of limit down}} a_{ij}^2}\sqrt{\sum_{ \widetilde{J} } a_{ij}^2}
	- 2 \sum_{J\cap \widetilde{J}} \sqrt{\sum_{I \vphantom{\widetilde{I} -- cludge for moving position of limit down}} a_{ij}^2}\sqrt{\sum_{\widetilde{I} } a_{ij}^2}\\
	&\geq 2\sum_{I,J} a_{ij}^2
	+ 2\sum_{\widetilde{I}, \widetilde{J}} a_{ij}^2
	 	-  \sum_{I\cap \widetilde{I}, J \vphantom{\widetilde{J} -- cludge for moving position of limit down}} a_{ij}^2 - 	 \sum_{I\cap \widetilde{I},\widetilde{J} } a_{ij}^2
	- \sum_{I \vphantom{\widetilde{I} -- cludge for moving position of limit down}, J\cap \widetilde{J}} a_{ij}^2 - \sum_{\widetilde{I}, J\cap \widetilde{J}} a_{ij}^2\\
	&= \sum_{I, J} a_{ij}^2
	+ \sum_{\widetilde{I}, \widetilde{J}} a_{ij}^2
		-  \sum_{I\cap \widetilde{I}, J \vphantom{\widetilde{J} -- cludge for moving position of limit down}} a_{ij}^2 - 	 \sum_{I\cap \widetilde{I}, \widetilde{J} } a_{ij}^2
	+ \sum_{I \vphantom{\widetilde{I} -- cludge for moving position of limit down}, J\setminus \widetilde{J}}  a_{ij}^2 + \sum_{\widetilde{I}, \widetilde{J}\setminus J }   a_{ij}^2.
\end{align*}
Thus, we clearly have
\[
\EE (X_{I,J} - X_{\widetilde{I}, \widetilde{J}})^2 \leq 	\EE (Y_{I,J} - Y_{\widetilde{I}, \widetilde{J}})^2
\]
(cf.\ Remark~\ref{rem:sidenote-Slepian-wrong} below).
Hence, by Slepian's lemma (Lemma~\ref{lem:Slepian}) and Lemma~\ref{lem:max-Gaussians} on the expected maxima of standard Gaussian random variables,
\begin{align*}
\EE \sup_{I, J} \sum_{i\in I} \sum_{j\in J} a_{ij} g_{ij}
&\leq
\EE \sup_{I, J}  \Biggl[\sum_{i\in I} g_i \sqrt{\sum_{j\in J} a_{ij}^2 }
+ \sum_{j\in J} \widetilde{g}_j \sqrt{\sum_{i\in I} a_{ij}^2 } \Biggr]\\
&\leq
\EE \sup_{I, J}  \sum_{i\in I} g_i \sqrt{\sum_{j\in J} a_{ij}^2}
+ \EE \sup_{I, J} \sum_{j\in J} \widetilde{g}_j \sqrt{\sum_{i\in I} a_{ij}^2 } \\
&\leq
 \EE \sup_{i\leq m} |g_i| \sup_{I, J}  \sum_{i\in I}  \sqrt{\sum_{j\in J} a_{ij}^2}
	+\EE \sup_{j\leq n} |\widetilde{g}_j | \sup_{I, J}\sum_{j\in J} \sqrt{\sum_{i\in I} a_{ij}^2 } \\
	&\leq
2\sqrt{\ln m}	\sup_{I, J}  \sum_{i\in I}  \sqrt{\sum_{j\in J} a_{ij}^2}
	+2\sqrt{\ln n} \sup_{I, J} \sum_{j\in J} \sqrt{\sum_{i\in I} a_{ij}^2 } .
\end{align*}
Recalling the estimate~\eqref{eq:I-J-oo-1-step-v-sumup}, we arrive at
	\begin{align*}	
\EE \sup_{I, J} \sup_{y\in B_\infty^m}\sup_{x\in B_\infty^n}
\sum_{i\in I, j\in J} y_i a_{ij} g_{ij} x_j
&\leq
\bigl(8\sqrt{\ln m} + \sqrt{2/\pi}\bigr)	\sup_{I, J}  \sum_{i\in I}  \sqrt{\sum_{j\in J} a_{ij}^2} \nonumber\\
& \qquad + \bigl(8\sqrt{\ln n} +2 \sqrt{2/\pi}\bigr)  \sup_{I, J} \sum_{j\in J} \sqrt{\sum_{i\in I} a_{ij}^2 },
	\end{align*}
	which completes the proof of Proposition \ref{prop:auxilary-main-gauss}.
	\end{proof}

\begin{remark}
\label{rem:sidenote-Slepian-wrong}
In the above proof, we also have
		\begin{align*}
		\MoveEqLeft \EE (X_{I,J} - X_{\widetilde{I}, \widetilde{J}})^2
		= \sum_{ I, J} a_{ij}^2
		+ \sum_{\widetilde{I},\widetilde{J}} a_{ij}^2
		-  \sum_{I\cap \widetilde{I}, J\cap \widetilde{J} } a_{ij}^2
		- \sum_{J\cap \widetilde{J}, I\cap \widetilde{I}}  a_{ij}^2\\
		&\geq \sum_{ I, J} a_{ij}^2
		+ \sum_{\widetilde{I},\widetilde{J}} a_{ij}^2
		-  \sum_{I\cap \widetilde{I}} \sqrt{\sum_{J } a_{ij}^2} \sqrt{\sum_{\widetilde{J} } a_{ij}^2}
		- \sum_{J\cap \widetilde{J} } \sqrt{\sum_{I}  a_{ij}^2} \sqrt{\sum_{ \widetilde{I}}  a_{ij}^2}\\
		&= \frac{1}{2} \EE (Y_{I,J} - Y_{\widetilde{I}, \widetilde{J}})^2.
		\end{align*}	
		Therefore, by Slepian's lemma (Lemma~\ref{lem:Slepian}) we may reverse the estimate from the proof as follows:
		\begin{align*}
		\EE \sup_{I, J} \sup_{y\in B_\infty^m}\sup_{x\in B_\infty^n}
\sum_{i\in I, j\in J} y_i a_{ij} g_{ij} x_j \ge \frac{1}{\sqrt{2}} \EE \sup_{I, J}  \biggl[\sum_{i\in I} g_i \sqrt{\sum_{j\in J} a_{ij}^2 }
+ \sum_{j\in J} \widetilde{g}_j \sqrt{\sum_{i\in I} a_{ij}^2 } \biggr].
		\end{align*}
\end{remark}

\begin{proof}[Proof of Theorem \ref{thm:main-gauss-sets}]
Recall that $\sup_{I_0,J_0}$ stands for the supremum taken over all sets $I_0\subset [M] \coloneqq \{1,\ldots , M \} $, $J_0\subset [N] \coloneqq \{1,\ldots , N \}$ with $|I_0|=m$, $|J_0|=n$.
Given such sets $I_0$, $ {J_0}$,
we introduce the sets
	\begin{align*}
		K& = K(I_0)  \coloneqq \conv \Bigl\{ \frac{1}{|I|^{1/q^*}}\bigl( \varepsilon_i \ind_{\{i\in I\}} \bigr)_{i\in I_0} \colon I\subset I_0, I\neq \emptyset, (\varepsilon_i )_{i\in  I_0}\in \{-1,1\}^{I_0} \Bigr\},\\
	L&= L(J_0)\coloneqq \conv \Bigl\{ \frac{1}{|J|^{1/p}}\bigl( \eta_j \ind_{\{j\in J\}} \bigr)_{j\in {J_0}} \colon J\subset {J_0}, J\neq \emptyset, (\eta_j )_{j\in  {J_0}}\in \{-1,1\}^ {J_0}\Bigr\}.
	\end{align*}
Then, by Lemma \ref{lem:set-K}, $B_{q^*}^{I_0} \subset \ln(em)^{1/q} K$ and $B_{p}^{J_0}\subset \ln(en)^{1/p^*} L$.
 Therefore,
	\begin{align}
	 \label{eq:proof-aux-prop-1}
	 \nonumber
	 \MoveEqLeft
\EE \sup_{I_0, J_0} \sup_{x\in B_p^ {J_0}} \sup_{y\in B_{q^*}^ {I_0}} \sum_{i\in I_0} \sum_{j\in J_0} y_i a_{ij} g_{ij} x_j
	  \\
	   \nonumber
&\leq 	\ln(em)^{1/q} \ln(en)^{1/p^*} \\
\nonumber &\quad\cdot
	 \EE \sup_{I_0, J_0} \sup_{I\subset I_0,  J\subset J_0} \sup  \Bigl\{ \frac 1{|I|^{1/q^*}|J|^{1/p}}  \sum_{i\in I}\sum_{j\in J}  \varepsilon_ia_{ij} g_{ij}\eta_j :
	 \
	 \varepsilon_i, \eta_j \in \{-1,1 \} \Bigr\}
	\\
	 \nonumber &   = \ln(em)^{1/q} \ln(en)^{1/p^*}
	 \\
	 \nonumber &\quad\cdot
	   \EE \max_{k\leq m, l\leq n}  \frac 1{k^{1/q^*} l^{1/p}} \sup_{I\subset [M], |I|=k} \sup_{J\subset [N], |J|=l} \sup_{x\in B_\infty^N} \sup_{y\in B_\infty^M} \sum_{i\in I}\sum_{j\in J}y_ia_{ij}g_{ij}x_j
	 \\
	 & 	 =   \ln(em)^{1/q} \ln(en)^{1/p^*} \EE \max_{k\leq m, l\leq n} Z_{k,l},
	\end{align}
	where we denoted
	\[
		Z_{k,l}\coloneqq \frac 1{k^{1/q^*}l^{1/p}} \sup_{I,J} \sup_{x\in B_\infty^N} \sup_{y\in B_\infty^M} \sum_{i\in I}\sum_{j\in J} y_ia_{ij} g_{ij}x_j ,
	\]
	with the suprema here (and later on in this proof) being always taken over all sets $I\subset [M], |I|=k$ and $J\subset [N], |J|=l$.
	
	By Proposition \ref{prop:auxilary-main-gauss}, we only know that for all $k\leq m$ and $l\leq n$,
	\begin{align}  \label{eq:proof-aux-prop-2}
	\EE Z_{k,l} & \leq
\bigl(8\sqrt{\ln M}   + \sqrt{2/\pi}\bigr)  \frac 1{k^{1/q^*}l^{1/p}}	\sup_{I, J}  \sum_{i\in I}  \sqrt{\sum_{j\in J} a_{ij}^2} \nonumber\\
 &   \qquad + \bigl(8\sqrt{\ln N} +2 \sqrt{2/\pi}\bigr)  \frac 1{k^{1/q^*}l^{1/p}}  \sup_{I, J} \sum_{j\in J} \sqrt{\sum_{i\in I} a_{ij}^2 },
	\end{align}
	but we shall use the Gaussian concentration and the union bound to obtain an estimate for $\EE \max_{k\leq m, l\leq n} Z_{k,l}.$
	
	Note first that $(k^{-1/q^\ast} \ind_{\{i\in I\}})_{i\in I_0}\in K(I_0)\subset B_{q^\ast}^{I_0}$
		and $(l^{-1/p} \ind_{\{j\in J\}})_{j\in J_0}\in L(J_0)\subset B_{p}^{J_0}$, provided that $|I|=k$, $|J|=l$, $I\subset I_0$, $J\subset J_0$.
	Therefore, 
	\begin{align*}
		  \frac 1{k^{1/q^*}l^{1/p}}	\sup_{I, J}  \sum_{i\in I}  \sqrt{\sum_{j\in J} a_{ij}^2}
		  & \leq \sup_{I_0, J_0} \sup_{x\in B_{p}^{J_0}} \sup_{y\in B_{q^\ast}^{I_0}}   \sum_{i\in I_0}  y_i\sqrt{\sum_{j\in J_0} a_{ij}^2x_j^2}	\\
		  & =  \sup_{I_0, J_0} \sup_{z\in B_{p/2}^{J_0}} \Bigl(\sum_{i\in I_0}  \bigl( \sum_{j\in J_0} a_{ij}^2z_j\bigr)^{q/2}\Bigr)^{1/q}
		\\ & =  \sup_{I_0, J_0} \|A\hadprod A \colon \ell^{J_0}_{p/2} \to \ell^{I_0}_{q/2}\|^{1/2}
	\end{align*}
	and, similarly,
	\begin{align*}
		  \frac 1{k^{1/q^*}l^{1/p}}	\sup_{I, J}   \sum_{j\in J} \sqrt{\sum_{i\in I} a_{ij}^2 }
		   \le \sup_{I_0, J_0}  \|(A\hadprod A)^T \colon \ell^{I_0}_{q^*/2} \to \ell^{J_0}_{p^*/2}\|^{1/2}.
	\end{align*}
	
	This together with the estimate in \eqref{eq:proof-aux-prop-2} gives
	\begin{align}  \label{eq:proof-aux-prop-3}
		\EE Z_{k,l}  &\leq
\bigl(8\sqrt{\ln M}   + \sqrt{2/\pi}\bigr)\sup_{I_0, J_0} \|A\hadprod A \colon \ell^{J_0}_{p/2} \to \ell^{I_0}_{q/2}\|^{1/2}
\nonumber\\
 &   \qquad + \bigl(8\sqrt{\ln N} +2 \sqrt{2/\pi}\bigr)  \sup_{I_0, J_0}  \|(A\hadprod A)^T \colon \ell^{I_0}_{q^*/2} \to \ell^{J_0}_{p^*/2}\|^{1/2} .
	\end{align}
		
	Note that by the Cauchy--Schwarz inequality, the function
	\[
		z\mapsto  \frac 1{k^{1/q^*}l^{1/p}} \sup_{I,J} \sup_{x\in B_\infty^N} \sup_{y\in B_\infty^M} \sum_{i\in I}\sum_{j\in J} y_ia_{ij}z_{ij}x_j
	\]
	is $D$-Lipschitz with
		  \begin{align*}
	 D\le \frac 1{k^{1/q^*}l^{1/p}} \sup_{I, J} \sqrt{  \sum_{j\in J} \sum_{i\in I} a_{ij}^2 }  &\le \sup_{I, J} \sqrt{ \sup_{x\in B_{p/2}^N}\sup_{y\in B_{q^*/2}^M} \sum_{i\in I}\sum_{j\in J} y_i a_{ij}^2 x_j}
	 \\ 
		 & \le \sup_{I_0, J_0} \sqrt{ \sup_{x\in B_{p/2}^N}\sup_{y\in B_{q^*/2}^M} \sum_{i\in I_0}\sum_{j\in J_0} y_i a_{ij}^2 x_j},
	 \end{align*}
	 where in the last inequality we used the fact that $k\le m$ and $l\le n$.
	  In order to estimate the right-hand side of the latter inequality, we consider the following two cases:
	
\textit{Case 1.} If $q^*\geq 2$, then $(q^*/2)^* = q/(2-q)\geq q/2$ and $\|\cdot\|_{q/(2-q)} \leq \|\cdot\|_{q/2}$.
Consequently,
\begin{align}	\label{eq:manipulations-to-fix1} 
\sup_{x\in B_{p/2}^N, y\in B_{q^*/2}^M} \sum_{i\in I_0} \sum_{j\in J_0}  y_i a_{ij}^2 x_j
&= 			\|A \hadprod A\colon \ell_{p/2}^{J_0}\to\ell_{q/(2-q)}^{I_0}\|   \leq 	\|A \hadprod A\colon \ell_{p/2}^{J_0}\to\ell_{q/2}^{I_0}\|.
\end{align}

\textit{Case 2.} If $q^*\leq 2$, then $B_{q^*/2}^M \subset B_1^M$ and $\|\cdot\|_{\infty} \leq\|\cdot\|_{q/2}$. Thus,
\begin{align}	\label{eq:manipulations-to-fix2} \nonumber
\sup_{x\in B_{p/2}^N, y\in B_{q^*/2}^M} \sum_{i\in I_0}\sum_{j\in J_0}  y_i a_{ij}^2 x_j
&\leq  		
\sup_{u\in B_{p/2}^N, v\in B_{1}^M} \sum_{i\in I_0} \sum_{j\in J_0}  v_i a_{ij}^2 u_j \\
& = 		\|A \hadprod A\colon \ell_{p/2}^{J_0}\to\ell_{\infty}^{I_0}\|
  \leq 	\|A \hadprod A\colon \ell_{p/2}^{J_0} \to\ell_{q/2}^{I_0}\|.
\end{align}
In both cases we have
\[
D\leq \sup_{I_0, J_0} \|A\hadprod A \colon \ell^{J_0}_{p/2} \to \ell^{I_0}_{q/2}\|^{1/2},
\] so the Gaussian concentration inequality (see, e.g., \cite[Chapter~5.1]{Ledoux}) implies that for all $u\geq 0$, $k\leq m$, and $l \leq n$,
	\begin{align*}
		\PP (Z_{k,l} \ge \EE Z_{k,l} +u)\le \exp\Bigl(-\frac{u^2}{2 \sup_{I_0, J_0} \|A\hadprod A \colon \ell^{J_0}_{p/2} \to \ell^{I_0}_{q/2}\|}\Bigr),
	\end{align*}
	 so
	 	\begin{multline*}
		\PP \bigl(Z_{k,l} \ge\max_{k\leq m, l\leq n} \EE Z_{k,l} + \sqrt{2\ln(mn)} u \sup_{I_0, J_0} \|A\hadprod A \colon \ell^{J_0}_{p/2} \to \ell^{I_0}_{q/2}\|^{1/2}\bigr)\\
		\le \exp(-u^2\ln (mn)).
	\end{multline*}
	This, together with  the union bound, implies that for $u\ge \sqrt 2$, we have
	\begin{multline*}
\PP\bigl(\max_{k\leq m, l\leq n} Z_{k,l} \geq \max_{k\leq m, l\leq n}\EE Z_{k,l} + \sqrt{2\ln(mn)} u \sup_{I_0, J_0} \|A\hadprod A \colon \ell^{J_0}_{p/2} \to \ell^{I_0}_{q/2}\|^{1/2}\bigr)
\\
\le mne^{-u^2\ln(mn)}
 = \exp\Bigl(- (u^2-1)\ln(mn) \Bigr)
 \leq e^{-u^2/2}.
	\end{multline*}
Hence, by Lemma~\ref{lem:mean-median} and the estimate in \eqref{eq:proof-aux-prop-3},
\begin{align*}
\EE\! \max_{k\leq m, l\leq n} Z_{k,l} & \leq \max_{k\leq m, l\leq n}\EE Z_{k,l}  +  \sqrt{2\ln(mn)}\Bigl(\sqrt{2} +  \frac1{e\sqrt 2}\Bigr) \sup_{I_0, J_0} \|A\hadprod A \colon \ell^{J_0}_{p/2} \to \ell^{I_0}_{q/2}\|^{1/2}
\\ &\leq \bigl(2.4\sqrt{\ln(mn)} +8\sqrt{\ln M}   + \sqrt{2/\pi}\bigr)\sup_{I_0, J_0} \|A\hadprod A \colon \ell^{J_0}_{p/2} \to \ell^{I_0}_{q/2}\|^{1/2}
\nonumber\\
 &   \qquad \qquad+ \bigl(8\sqrt{\ln N} +2 \sqrt{2/\pi}\bigr) \sup_{I_0, J_0}  \|(A\hadprod A)^T \colon \ell^{I_0}_{q^*/2} \to \ell^{J_0}_{p^*/2}\|^{1/2}.
\end{align*}
	Recalling \eqref{eq:proof-aux-prop-1} yields the assertion.
\end{proof}

\subsection{Coupling}
\label{subsect:coupling}

In this subsection we use contraction principles and the coupling described in Lemma~\ref{lem:coupling} to prove Corollaries~\ref{thm:main-bounded-introduction} and~\ref{thm:main-psi-r-introduction}, and Proposition~\ref{prop:tail-bounds}.
Below we state  more general versions of the corollaries akin to Theorem~\ref{thm:main-gauss-sets}
(the versions from the introduction follow by setting $M=m$, $N=n$).

\begin{theorem}[General version of Corollary~\ref{thm:main-bounded-introduction}]
\label{thm:main-bounded}
	Assume that $m\le M$, $n\le N$, $1\leq p,q \leq \infty$, and $X=(X_{ij})_{i\leq M, j\leq N}$ has independent mean-zero entries taking values in $[-1,1]$.  Then
\begin{multline*}
\EE \sup_{I, J} \|X_A\colon \ell_p^{J} \to \ell_q^{I} \|
=\EE \sup_{I, J} \sup_{x\in B_p^ {J}} \sup_{y\in B_{q^*}^ {I}} \sum_{i\in I} \sum_{j\in J} y_i a_{ij}X_{ij}x_j \\
 \le  \ln(en)^{1/p^*}  \ln(em)^{1/q} \Bigl[ \bigl(2.4\sqrt{2\pi} \sqrt{\ln(mn)}+8\sqrt{2\pi}\sqrt{\ln M}+2\bigr) \sup_{I,J} \|A\hadprod A \colon \ell^ {J}_{p/2} \to \ell^{ I}_{q/2}\|^{1/2}\\
  +\bigl(8\sqrt{2\pi}\sqrt{\ln N} +4\bigr) \sup_{I,J} \|(A\hadprod A)^T \colon \ell^{ I}_{q^*/2} \to \ell^ {J}_{p^*/2}\|^{1/2} \Bigr],
 \end{multline*}
 where the suprema are taken over all sets $I\subset \{1,\ldots,M\}$, $J\subset \{1,\ldots,N\}$ such that $|I|=m$, $|J|=n$.
 \end{theorem}

	\begin{remark}[\textit{Symmetrization of  entries of a random matrix}]
	\label{rmk:symmetrization}
		Let $\widetilde{Z}$ be an independent copy of a random matrix $Z$ with mean $0$ entries. Then for any norm $\|\cdot\|$,
		 including the operator norm from $\ell_p^n$ to $\ell_q^m$, 
		 we have by Jensen's inequality	
		\[
			\EE\|Z\| = \EE\| Z-\EE \widetilde{Z}\| \le  \EE\|Z-\widetilde{Z}\| \le \EE\|Z\|+\EE\|\widetilde{Z}\| =2\EE\|Z\|.
		\]
		Therefore, in many cases we may simply assume that we deal with matrices with symmetric (not only mean $0$) entries. For example, in the setting of Theorem~\ref{thm:main-bounded}, the entries of $X-\widetilde{X}$ are symmetric and take values in $[-2,2]$, so it suffices to prove the assertion of this theorem 
		(with a two times smaller constant on the right-hand side) under the additional assumption that the entries of the given random matrix are symmetric.
	\end{remark}

\begin{proof}[Proof of Theorem~\ref{thm:main-bounded}]
By Remark~\ref{rmk:symmetrization} we may and do assume that the entries of $X$ are symmetric --- in this case we need to prove the assertion with a  two times smaller constant.

Since the entries of $X$ are independent and symmetric, $X$ has the same distribution as $(\varepsilon_{ij}|X_{ij}|)_{i,j}$, where $(\varepsilon_{ij})_{i\leq M, j\leq N}$ is a random matrix with i.i.d.\  Rademacher entries, independent of all other random variables.
Thus, the contraction principle
(see Lemma \ref{lem:contraction-principle}) applied conditionally
 yields (below the suprema are taken over all sets $I\subset \{1,\ldots,M\}$, $J\subset \{1,\ldots,N\}$ such that $|I|=m$, $|J|=n$, \emph{and} over all $x\in B_p^ {J}, y\in B_{q^*}^ {I}$, and the sums run over all $i\in I$ and $j\in J$)
	\begin{align*}
		\EE \sup \sum_{I,J} y_ia_{ij} X_{ij} x_j
		& =  \EE \sup \sum_{I,J} y_ia_{ij} \varepsilon_{ij}\bigl|X_{ij} \bigr| x_j \leq \EE \sup \sum_{I,J} y_ia_{ij} \varepsilon_{ij}x_j
		\\
		& = \sqrt{\frac {\pi} 2} \EE \sup \sum_{I,J} y_ia_{ij} \varepsilon_{ij} \EE |g_{ij}|x_j \leq \sqrt{\frac {\pi} 2} \EE \sup \sum_{I,J} y_ia_{ij} \varepsilon_{ij} |g_{ij}|x_j
		\\
		& =  \sqrt{\frac {\pi} 2} \EE \sup \sum_{I,J} y_ia_{ij} g_{ij}x_j,
	\end{align*}
	and the assertion follows from Theorem~\ref{thm:main-gauss-sets}.
\end{proof}

\begin{theorem}[General version of Corollary~\ref{thm:main-psi-r-introduction}]
\label{thm:main-psi-r}
	Assume that $K,L >0$, $r\in(0,2]$, $m\le M$, $n\le N$, $1\leq p,q \leq \infty$, and $X=(X_{ij})_{i\leq M, j\leq N}$ has independent mean-zero entries satisfying
	\[
	\PP(|X_{ij}|\ge t) \le Ke^{-t^r/L} \qquad \text{for all } t\ge 0,\, i\leq M,\,j\leq N.
	\]
	Then
\begin{multline*}
\EE \sup_{I, J} \|X_A\colon \ell_p^{J} \to \ell_q^{I} \|
=\EE \sup_{I, J} \sup_{x\in B_p^ {J}} \sup_{y\in B_{q^*}^ {I}} \sum_{i\in I} \sum_{j\in J} y_i a_{ij}X_{ij}x_j \\
 \lesssim_{r,K,L}  (\ln n)^{1/p^*}  (\ln m)^{1/q} \ln(MN)^{\frac 1r - \frac12} \Bigl[ \bigl( \sqrt{\ln(mn)}+\sqrt{\ln M}\bigr) \sup_{I,J} \|A\hadprod A \colon \ell^ {J}_{p/2} \to \ell^{ I}_{q/2}\|^{1/2}\\
  +\sqrt{\ln N} \sup_{I,J} \|(A\hadprod A)^T \colon \ell^{ I}_{q^*/2} \to \ell^ {J}_{p^*/2}\|^{1/2} \Bigr],
 \end{multline*}
 where  the suprema are taken over all sets $I\subset \{1,\ldots,M\}$, $J\subset \{1,\ldots,N\}$ such that $|I|=m$, $|J|=n$.
  	 \end{theorem}
	
\begin{proof}
	Let $\widetilde{X}$ be an independent copy of $X$. Then 
	\begin{align*}
	\PP(|X_{ij}-\widetilde{X}_{ij}|\ge t) &\le \PP(|X_{ij}|\ge t/2 \text{ or } |\widetilde{X}_{ij}|\ge t/2) 
	\\
	&\le 2 \PP(|X_{ij}|\ge t/2) \le 2Ke^{-t^r/(2^rL)}.
	\end{align*}
	This means that the symmetric matrix $X-\widetilde{X}$ satisfies the assumptions of Theorem~\ref{thm:main-psi-r}. Hence, due to Remark~\ref{rmk:symmetrization}, we may and do assume that the entries of $X$ are symmetric.
	
	Take the unique positive parameter $s$ satisfying $\frac 1r = \frac 12+\frac 1s$.
	For $i\leq M$, $j\leq N$, let $g_{ij}$ be i.i.d.\ standard Gaussian variables,  independent of other variables,
	and let $Y_{ij}$ be i.i.d.\ non-negative Weibull random variables with shape parameter $s$ scale parameter $1$ (i.e., $\PP(Y_{ij}\ge t)=e^{-t^s}$ for $t\geq 0$), independent of other variables.
	(In the case $r=2$, we have $s=\infty$ and then $Y_{ij}=1$ almost surely.)
	Take
	\[ (U_{ij})_{i\leq M, j\leq N} \eqdistr (|X_{ij}|)_{i\leq M, j\leq N}, \qquad (V_{ij})_{i\leq M, j\leq N}\eqdistr ( |g_{ij}|Y_{ij})_{i\leq M, j\leq N}\]
	 as in  Lemma \ref{lem:coupling} (we pick a pair $(U_{ij}, V_{ij})$ separately for every $(i,j)$, and then take such a version of each pair that the system of $MN$ random pairs $(U_{ij}, V_{ij})$ is independent).
	
	Let $(\varepsilon_{ij})_{i\leq M, j\leq N}$ be a random matrix with i.i.d.\ Rademacher entries, independent of all other random variables. Since the entries of $X$ are symmetric and independent, $X$ has the same distribution as $(\varepsilon_{ij}|X_{ij}|)_{ij}$.
	 By Lemma~\ref{lem:coupling} we know that
	\begin{equation*}
		U_{ij}\le  (8L)^{1/r}\Bigl(\Bigl(\frac{\ln(K/c)}4\Bigr)^{1/r} +V_{ij} \Bigr) \lesssim_{r,K,L} 1+V_{ij} \qquad \text{a.s.}
	\end{equation*}
	We use the contraction principle conditionally for $\EE_{\varepsilon}$, i.e., for $U_{ij}$'s and $V_{ij}$'s fixed. 
	More precisely, we apply Lemma~\ref{lem:contraction-principle} to the space $\bf{X}$ of all $M\times N$ matrices with real coefficients, equipped with the norm 
	\[
	\|(M_{ij})_{i\leq M,j\leq N}\|\coloneqq \sup_{I,J} \bigl\|(M_{ij})_{i\in I,j\in J}\colon \ell_p^I \to \ell_q^J \bigr\| = \sup \sum_{I,J} y_iM_{ij}x_j
	\]
	 (where the first supremum is taken over all sets $I\subset \{1,\ldots,M\}$, $J\subset \{1,\ldots,N\}$ such that $|I|=m$, $|J|=n$; recall that the second supremum is taken over all sets $I,J$ as in the first supremum, and  over all $x\in B_p^ {J}, y\in B_{q^*}^ {I}$, and the  sum  runs over all $i\in I$ and $j\in J$);
	 note that we identify $\bf{X}$ with $\RR^{MN}$ (and $MN$ plays the role of $n$ from Lemma~\ref{lem:contraction-principle}).
	  We apply the contraction principle of Lemma~\ref{lem:contraction-principle} 
	  (conditionally, with the values of $U_{ij}$'s and $V_{ij}$'s fixed)
	  with coefficients $\alpha_{ij}\coloneqq \frac{U_{ij}}{C(r,K,L)(1+V_{ij})}$ and  points ${\bf x}_{ij} \coloneqq \bigl(a_{kl}C(r,K,L)(1+V_{kl})\ind_{\{(k,l)=(i,j)\}}\bigr)_{kl} \in \bf{X}$ to get
	\begin{align} \label{proof_cor_psi_r1}
	\nonumber
		\EE \sup \sum_{I,J} y_ia_{ij} X_{ij} x_j
		& =  \EE \sup \sum_{I,J} y_ia_{ij} \varepsilon_{ij}|X_{ij}|x_j =  \EE \sup \sum_{I,J} y_ia_{ij} \varepsilon_{ij}U_{ij}x_j 
		\\
		& \overset{\mathclap{\text{Lemma } \ref{lem:contraction-principle}}} {\lesssim_{r,K,L}} \  \EE\sup \sum_{I,J} y_ia_{ij} \varepsilon_{ij}x_j  +  \EE\sup \sum_{I,J} y_i a_{ij} \varepsilon_{ij}V_{ij}x_j  .
	\end{align}
	We may estimate the first term using Theorem~\ref{thm:main-bounded} applied to the matrix $(\varepsilon_{ij})_{i\leq M, j\leq N}$ as follows,
	\begin{align}
		\EE\sup \sum_{I,J} y_ia_{ij} \varepsilon_{ij}x_j
 		&\lesssim  \ln(en)^{1/p^*}  \ln(em)^{1/q}\nonumber\\
 		&\qquad\cdot\Bigl[ \bigl( \sqrt{\ln(mn)}+\sqrt{\ln M}\bigr) \sup_{I,J} \|A\hadprod A \colon \ell^ {J}_{p/2} \to \ell^{ I}_{q/2}\|^{1/2}	\nonumber\\
  		&\qquad\qquad+\sqrt{\ln N}  \sup_{I,J} \|(A\hadprod A)^T \colon \ell^{ I}_{q^*/2} \to \ell^ {J}_{p^*/2}\|^{1/2} \Bigr].\label{proof_cor_psi_r2}
 	\end{align}
 	Recall that  $(\varepsilon_{ij} V_{ij})_{i\leq M, j\leq N}\eqdistr (\varepsilon_{ij}g_{ij}Y_{ij})_{i\leq M, j\leq N}$ and that $Y_{ij}\ge 0$ almost surely.  
 	Next we again use the contraction principle (applied conditionally for $\EE_\varepsilon$, i.e. for fixed $Y_{ij}$'s and $g_{ij}$'s) and get
	 \begin{align} \label{proof_cor_psi_r2.5}
	\nonumber
		 \EE\sup \sum_{I,J} y_i a_{ij} \varepsilon_{ij}V_{ij}x_j
		 & =  \EE\sup \sum_{I,J} y_i a_{ij} \varepsilon_{ij}g_{ij} Y_{ij}x_j
		 \\  &
		 \le \EE_Y\max_{i\le M, j\le N} |Y_{ij}| \ \EE_{\varepsilon,g} \sup \sum_{I,J} y_i a_{ij} \varepsilon_{ij}g_{ij} x_j.
	\end{align}
	Moreover, 
	 Theorem~\ref{thm:main-gauss-sets}  and Lemma~\ref{lem:max-psi-r} (applied with $r=s$, $k=MN$, $Z_{ij}=Y_{ij}$, and $K=1=L$), imply
	\begin{align} \label{proof_cor_psi_r3}
	\nonumber
		\MoveEqLeft[8]
		\EE_Y\max_{i\le M, j\le N} |Y_{ij}| \ \EE_{\varepsilon,g}  \sup \sum_{I,J} y_i a_{ij} \varepsilon_{ij}g_{ij} x_j	   \\
		\nonumber &
		 \lesssim_{r} \ln(MN)^{1/s} \  \EE \sup \sum_{I,J} y_i a_{ij} g_{ij} x_j  		\\
		  \nonumber &
		 \lesssim \ln(MN)^{\frac 1r -\frac 12}   (\ln n)^{1/p^*}  (\ln m)^{1/q} \\
		  \nonumber
		 & \qquad \cdot\Bigl[ \bigl( \sqrt{\ln(mn)}+\sqrt{\ln M}\bigr) \sup_{I,J} \|A\hadprod A \colon \ell^ {J}_{p/2} \to \ell^{ I}_{q/2}\|^{1/2}	\\
		 & \qquad \qquad \qquad 
  		+\sqrt{\ln N}  \sup_{I,J} \|(A\hadprod A)^T \colon \ell^{ I}_{q^*/2} \to \ell^ {J}_{p^*/2}\|^{1/2} \Bigr].
	\end{align}
	Combining the estimates in \eqref{proof_cor_psi_r1}--\eqref{proof_cor_psi_r3} yields the assertion.
\end{proof}

Finally, we prove that these estimates of the operator norms translate into tail bounds.

\begin{proof}[Proof of Proposition~\ref{prop:tail-bounds}]
	Since \eqref{eq:tail-bound-moments} implies \eqref{eq:tail-bound} (by Lemma~\ref{lem:psi-r-equivalence}), it suffices to prove
	inequality~\eqref{eq:tail-bound-moments}.
	By the symmetrization argument similar to the one from the first paragraph of the proof of Theorem~\ref{thm:main-psi-r}, we may nad will assume that $X$ has independent and \textit{symmetric} entries satisfying \eqref{eq:tail-bound-assumption}.
	By assumption \eqref{eq:tail-bound-assumption}, and the inequality $2(a+b)^r\ge a^r+b^r$
we have for every $t\ge 0$,
	\begin{multline*}
		\PP\bigl((2L)^{-1/r}|X_{ij}| \ge t+ (\ln K)^{1/r}\bigr)
		\\
		\le K\exp\Bigl( -2\bigl( t+ (\ln K)^{1/r} \bigr)^r \Bigr)\le K \exp\bigl( -t^r -\ln K \bigr)
		= e^{-t^r},
	\end{multline*}
	so (as in the proof of Lemma~\ref{lem:coupling}) there exists  a random matrix  $(Y_{ij})_{i\leq m, j\leq n}$  with i.i.d.\  entries with the symmetric Weibull distribution with shape parameter $r$ and scale parameter $1$ (i.e.,  $\PP( |Y_{ij}| \ge t) = e^{-t^r}$ for $t\geq 0$)
satisfying
	\begin{equation} \label{eq:tail-bound-proof0}
		|X_{ij}| \le (2L)^{1/r} \bigl( (\ln K)^{1/r} +|Y_{ij}| \bigr) \lesssim_{r,K,L} 1+Y_{ij}	\qquad \text{a.s.}
	\end{equation}

	Let $(\varepsilon_{ij})_{i\leq m, j\leq n}$ be a matrix of independent Rademacher random variables independent of all others, and
	let $\| \cdot\|$ denote the operator norm from $\ell_{p}^n$ to $\ell_q^m$. Let $E_{ij}$ be a matrix with $1$ at the intersection of $i$th row and $j$th column and with other entries~$0$.
	The contraction principle (i.e., Lemma \ref{lem:contraction-principle}) applied  conditionally, \eqref{eq:tail-bound-proof0}, and the triangle inequality yield for any $\rho\ge 1$,
\begin{align*} 
 \biggl( \EE \Bigl\|\sum_{i=1}^m \sum_{j=1}^n & X_{ij} a_{ij}E_{ij} \Bigr\|^\rho \biggr)^{1/\rho}
\leq  \biggl( \EE \Bigl\|\sum_{i,j} \varepsilon_{ij}|X_{ij}| a_{ij}E_{ij} \Bigr\|^\rho \biggr)^{1/\rho}
\\
 &\lesssim_{r,K,L}    \biggl( \EE \Bigl\|\sum_{i,j} \varepsilon_{ij} a_{ij}E_{ij} \Bigr\|^\rho \biggr)^{1/\rho} +   \biggl( \EE \Bigl\|\sum_{i,j} \varepsilon_{ij}|Y_{ij}| a_{ij}E_{ij} \Bigr\|^\rho \biggr)^{1/\rho}
\\
& =     \biggl( \EE \Bigl\|\sum_{i,j} \varepsilon_{ij} a_{ij}E_{ij} \Bigr\|^\rho \biggr)^{1/\rho} +   \biggl( \EE \Bigl\|\sum_{i,j} Y_{ij} a_{ij}E_{ij} \Bigr\|^\rho \biggr)^{1/\rho}.
\end{align*}
Therefore, it suffices to prove \eqref{eq:tail-bound-moments} for random matrices $(Y_{ij})_{ij}$ and $(\varepsilon_{ij})_{ij}$ instead of $X$.

	Since by assumption $K,L\ge 1$, both random matrices $(Y_{ij})_{ij}$ and $(\varepsilon_{ij})_{ij}$ satisfy \eqref{eq:tail-bound-assumption}, so for them inequality \eqref{eq:tail-bound-assumption-bound} holds.
	By the comparison of weak and strong moments \cite[Theorem~1.1]{LatalaStrzeleckaMat} (note that the random variables $Y_{ij}$ satisfy the assumption $\|Y_{ij}\|_{2s}\le\alpha \|Y_{ij}\|_s$ for all $s\ge 2$ with $\alpha =2^{1/r}$ by \cite[Remark~1.5]{LatalaStrzeleckaMat}), we have
	\begin{multline} \label{eq:tail-bound-proof1}
		 \biggl( \EE \Bigl\|\sum_{i,j} Y_{ij} a_{ij}E_{ij} \Bigr\|^\rho \biggr)^{1/\rho}
		 = \biggl( \EE\sup_{x\in B_p^n,\ y\in B_{q^*}^m} \Bigl|\sum_{i,j} y_iY_{ij} a_{ij}x_{j} \Bigr|^\rho \biggr)^{1/\rho}
		\\
		\lesssim _r   \EE\sup_{x\in B_p^n,\ y\in B_{q^*}^m} \sum_{i,j} y_iY_{ij} a_{ij}x_{j} +  \sup_{x\in B_p^n,\ y\in B_{q^*}^m} \biggl(\EE \Bigl| \sum_{i,j} y_iY_{ij} a_{ij}x_{j} \Bigr|^{\rho} \biggr)^{1/\rho}.
	\end{multline}
	 Because of inequality~\eqref{eq:tail-bound-assumption-bound}, the first summand on the right-hand side may be estimated by $\gamma D$. Lemma~\ref{lem:psi_prelim5} and the implication \ref{item:psi-r-i}$\implies$\ref{item:psi-r-ii} from Lemma~\ref{lem:psi-r-equivalence} yield
	 \begin{equation*}	
	 	 	\biggl(\EE \Bigl| \sum_{i,j} y_iY_{ij} a_{ij}x_{j} \Bigr|^{\rho} \biggr)^{1/\rho} \lesssim_{r,K,L} \ \rho^{1/r} \sqrt{ \sum_{i,j} y_i^2a_{ij}^2x_{j}^2}.
	 \end{equation*}
	 Moreover, by \eqref{eq:manipulations-to-fix1} and \eqref{eq:manipulations-to-fix2} (used with $m=M$ and $n=N$)
	 and our assumption that $\|A\hadprod A \colon \ell^n_{p/2} \to \ell^m_{q/2}\|^{1/2} \leq D$,
	 	 \[
	 	 \sup_{x\in B_p^n,\ y\in B_{q^*}^m} \sqrt{ \sum_{i,j} y_i^2a_{ij}^2x_{j}^2} \leq D,
	 \]
	 so the second summand on the right-hand side of \eqref{eq:tail-bound-proof1} is bounded above (up to a multiplicative constant depending only on $r$, $K$, and $L$) by $\rho^{1/r}D$.
	 Thus, \eqref{eq:tail-bound-moments} indeed holds for the random matrix $(Y_{ij})_{ij}$ instead of $X$.
	A similar reasoning shows that the same inequality holds also for the random matrix $( \varepsilon_{ij})_{ij}$ (one may also simply use the Khintchine--Kahane inequality and assumption~\eqref{eq:tail-bound-assumption-bound}).
	\end{proof}

%
\section{Proofs of further results}
\label{sec:particular-ranges}
%

%
\subsection{Gaussian random variables}
\label{sec:gaussian-rvs}
%

\begin{proof}[Proof of Proposition~\ref{prop:gauss-p-smaller-2}]
Fix $1\leq p\leq 2$ and $1\leq q\leq \infty$.
Let $K$ be the set defined in Lemma~\ref{lem:set-K} for which $B_p^n\subset \ln(en)^{1/p^*} K$.
Then
\begin{equation}
\label{eq:reduction-to-K}
\|G_A \colon \ell^n_p\to \ell^m_q\|
=  \sup_{x\in B_p^n} \|G_A x\|_q
\leq  \ln(en)^{1/p^*}\! \sup_{x\in \Ext (K)}\! \|G_A x\|_q,
\end{equation}
where  $\Ext(K)$ is the set of extreme points of $K$.
We shall now estimate the expected value of the right-hand side of \eqref{eq:reduction-to-K}.

To this end, we first consider a fixed $x=(x_j)_{j=1}^n \in \Ext(K)$.
Then there exists a non-empty index set $J\subset \{1,\dots,n\}$ of cardinality $k\leq n$ such that $x_j = \frac{\pm 1}{k^{1/p}}$ for $j\in J$ and $x_j=0$ for $j\notin J$. We have
\begin{equation}
\label{eq:G_Ax-q-norm-distr}
\|G_A x\|_q =  \Bigl\|  \Bigl( \sum_{j=1}^n a_{ij}g_{ij} x_j \Bigr)_{i=1}^m  \Bigr\|_q
= \Bigl( \sum_{i=1}^m \Bigl| \sum_{j=1}^n a_{ij}g_{ij} x_j\Bigr|^q\Bigr)^{1/q} .
\end{equation}
Let us estimate the Lipschitz constant of the  function
	\begin{align}\label{eq:lipschitz constant q-norm mapping}
		z=(z_{ij})_{ij} \mapsto \Bigl\|  \Bigl( \sum_{j=1}^n a_{ij}z_{ij} x_j \Bigr)_{i=1}^m  \Bigr\|_q = \sup_{y\in B_{q^*}^m}\sum_{i=1}^m	 \sum_{j=1}^n y_ia_{ij}z_{ij}x_j.
\end{align}
It follows from the Cauchy--Schwarz inequality (used in $\RR^{m\times n}$) that
	\begin{align}	\label{eq:lipschitz} \nonumber
		\sup_{y\in B_{q^*}^m} \sum_{i=1}^m \sum_{j=1}^n y_ia_{ij}z_{ij}x_j
		& \leq \|z\|_2  \sqrt{\sup_{y\in B_{q^*}^m} \sum_{i=1}^m \sum_{j=1}^n y_{i}^2a_{ij}^2x_j^2}
		\\
	    & =\|z\|_2 \frac{1}{k^{1/p}}\sqrt{ \sup_{y\in B^m_{q^*/2}}  \sum_{i=1}^m  \sum_{j\in J} y_i a_{ij}^2  }
    		=  \|z\|_2\frac{ b_J}{k^{1/p}},
	\end{align}
where we put
\[
b_J\coloneqq\sqrt{ \sup_{y\in B^m_{q^*/2}}  \sum_{i=1}^m  \sum_{j\in J} y_i a_{ij}^2   }\,.
\]	
This shows that the function defined by \eqref{eq:lipschitz constant q-norm mapping} is $\frac{ b_J}{k^{1/p}}$-Lipschitz continuous.
Therefore, by the Gaussian concentration inequality (see, e.g., \cite[Chapter~5.1]{Ledoux}), for any $u\geq 0$,
\begin{equation}
\label{eq:G_Ax-q-norm-concentration}
\PP\bigl(\|G_A x\|_q \geq \EE\|G_A x\|_q  + u\bigr) \leq \exp\bigl( - \frac{k^{2/p} u^2}{2b_J^2}\bigr).
\end{equation}
We shall transform this inequality into a form which is more convenient to work with. We want to estimate $\EE \|G_A x\|_q $ independently of $x$
 and get rid of the dependence on $J$ and $p$ on the right-hand side.
By \eqref{eq:G_Ax-q-norm-distr}
and the fact that $x\in\Ext(K)\subset B_p^n$, we obtain
\begin{align*}
\EE \|G_A x\|_q
& \leq (\EE\|G_A x\|_q^q)^{1/q}
= \gamma_q \Bigl( \sum_{i=1}^m \Bigl|\sum_{j=1}^n a_{ij}^2 x_j^2\Bigr|^{q/2}\Bigr)^{1/q}\\
& \leq \gamma_q \sup_{z\in B_p^n} \Bigl( \sum_{i=1}^m \Bigl| \sum_{j=1}^n a_{ij}^2 z_j^2\Bigr|^{q/2}\Bigr)^{1/q}
 =\gamma_q \|A\hadprod A \colon \ell^n_{p/2} \to \ell^m_{q/2}\|^{1/2} \eqqcolon a.
\end{align*}

We use the definition of $b_J$,
then interchange the sums,
use the triangle inequality,
and then the inequality between the arithmetic mean and the power mean of order $p^*/2\geq 1$ (recall that $|J|=k$ and $p\leq 2$) to obtain
\begin{align} \label{eq:kbj}	\nonumber
	k^{2/p^*-1} b_J^2
	&= k^{2/p^*-1} \sup_{y\in B^m_{q^*/2}}  \sum_{i=1}^m \sum_{j\in J} a_{ij}^2 y_i
	= k^{2/p^*-1} \sup_{y\in B^m_{q^*/2}} \sum_{j\in J} \Bigl| \sum_{i=1}^m  a_{ij}^2 y_i \Bigr| \\ \nonumber
	&\leq \sup_{y\in B^m_{q^*/2}} \Bigl(  \sum_{j\in J} \Bigl| \sum_{i=1}^m a_{ij}^2 y_i \Bigr| ^{p^*/2}\Bigr)^{2/p^*}
	\leq \sup_{y\in B^m_{q^*/2}} \Bigl( \sum_{j=1}^n\Bigl| \sum_{i=1}^m  a_{ij}^2 y_i \Bigr| ^{p^*/2}\Bigr)^{2/p^*}	\\
	& =  \|(A\hadprod A)^T \colon \ell^m_{q^*/2} \to \ell^n_{p^*/2}\| \eqqcolon b^2.
\end{align}
The two inequalities above, together with inequality~\eqref{eq:G_Ax-q-norm-concentration}
(applied with $u=k^{1/p^* -1/2} b_J \sqrt{2\ln(en)} s$), imply that
\begin{align}
\label{eq:concentration_a_b}
  \PP\bigl(\|G_A x\|_q \geq a + b \sqrt{2\ln(en)}\, s \bigr)
 & \leq \exp\bigl( - k^{2/p+2/p^*-1} \ln(en)s^2\bigr) \\
&= \exp\bigl( - k \ln(en)s^2\bigr)\nonumber
\end{align}
holds for any $s\geq 0$ and all $x\in \Ext(K)$ with support of cardinality $k$.

For any $k\leq n$,
 there are $2^k \binom{n}{k} \leq 2^k n^k \leq \exp( k \ln(en))$ vectors in $\Ext(K)$
 with support of cardinality $k$.
Therefore, using a union bound together with \eqref{eq:concentration_a_b}, we see that, for all $s\geq \sqrt{2}$,
\begin{multline*}
\PP\bigl(\sup_{x\in \Ext K}\! \|G_A x\|_q \geq a + b \sqrt{2\ln(en)} s \bigr)
 \leq \sum_{k=1}^n \exp( -k \ln(en)(s^2-1))\\
  \leq n \exp( - \ln(en)(s^2-1))
 = n(en)^{-s^2+1}\leq e^{-s^2+1}.
\end{multline*}
Hence, by Lemma~\ref{lem:mean-median} (applied with $s_0\coloneqq\sqrt{2}$, $\alpha\coloneqq e$, $\beta\coloneqq 1$, and $r\coloneqq 2$),
\begin{align*}
\EE\! \sup_{x\in \Ext K}\! \|G_A x\|_q & \leq a + b \sqrt{2\ln(en)} \Bigl(\sqrt{2} + e \frac{e^{-2}}{2\sqrt{2}}\Bigr) \leq a + 2.2 b \sqrt{\ln(en)}.
\end{align*}
Recalling \eqref{eq:reduction-to-K} and the definitions of $a$ and $b$ yields the assertion.
\end{proof}

We now turn to the special case $q=1$.

\begin{proof}[Proof of Proposition~\ref{prop:gauss-q=1}]
	Since the first part of this proof works for general $q\geq 1$,
	we do not restrict our attention to $q=1$ for now.
First of all,
\[
\EE \| G_A \colon \ell_{p}^n \to\ell_q^m \| \leq \bigl(\EE \| G_A \colon \ell_{p}^n \to\ell_q^m \|^q\bigr)^{1/q}
= \Bigl( \EE \sup_{x\in B_{p}^n} \sum_{i=1}^m |\langle X_i,x \rangle|^q\Bigr)^{1/q},
\]
where $X_i=(a_{ij}g_{ij})_{j=1}^n$ is the $i$-th row of the matrix $G_A$. Centering this expression gives
\begin{align}
\label{eq:q=1-centered}
	\EE \sup_{x\in B_{p}^n} \sum_{i=1}^m |\langle X_i,x \rangle|^q
	&\leq \EE \sup_{x\in B_{p}^n} \Big[\sum_{i=1}^m |\langle X_i,x \rangle|^q - \EE|\langle X_i,x \rangle|^q \Big]\\
	& \qquad + \sup_{x\in B_{p}^n} \sum_{i=1}^m \EE |\langle X_i, x \rangle|^q. \nonumber
\end{align}
We first take care of the second term on the right-hand side of~\eqref{eq:q=1-centered}.
We have
\begin{align}
  \label{eq:q=1-first-term}	\nonumber
	\sup_{x\in B_{p}^n} \sum_{i=1}^m \EE |\langle X_i, x \rangle|^q & =  \gamma_q^q \sup_{x\in B_{p}^n} \sum_{i=1}^m \Bigl(\sum_{j=1}^na_{ij}^2x_j^2\Bigr)^{q/2} \\
	& = \gamma_q^q  \sup_{z\in B_{p/2}^n} \Big\|\Bigl(\sum_{j=1}^na_{ij}^2z_j\Bigr)_{i\le m} \Big\|_{q/2}^{q/2}
	 = \gamma_q^q  \|A\hadprod A \colon \ell^n_{p/2} \to \ell^m_{q/2}\|^{q/2}.
\end{align}

In order to deal with the first term on the right-hand side of~\eqref{eq:q=1-centered},
we use a~symmetrization trick together with the contraction principle.
The latter is the reason that we need to work with $q=1$ here. We start with the symmetrization. Denoting by $\widetilde{X}_1,\dots,\widetilde{X}_n$ independent copies of $X_1, \dots, X_n$ and by $(\varepsilon_i)_{i=1}^m$ a sequence of Rademacher random variables independent of all others, we obtain by Jensen's and the triangle inequalities that

\begin{align}
	\EE \sup_{x\in B_{p}^n}   \Bigl[\sum_{i=1}^m & |\langle X_i,x \rangle|^q - \EE|\langle X_i,x \rangle|^q \Bigr]
	 = \EE \sup_{x\in B_{p}^n} \Big[\sum_{i=1}^m |\langle X_i,x \rangle|^q - \EE|\langle \widetilde{X}_i,x \rangle|^q \Bigr]\nonumber\\
	& \leq \EE \sup_{x\in B_{p}^n} \Bigl[\sum_{i=1}^m |\langle X_i,x \rangle|^q - |\langle \widetilde{X}_i,x \rangle|^q \Bigr] \nonumber\\
	&	 = \EE \sup_{x\in B_{p}^n} \Bigl[\sum_{i=1}^m \varepsilon_i( |\langle X_i,x \rangle|^q - |\langle \widetilde{X}_i,x \rangle|^q) \Bigr]
	 \leq 2 \cdot  \EE \sup_{x\in B_{p}^n} \sum_{i=1}^m \varepsilon_i |\langle X_i,x \rangle|^q.\label{eq:q=1-second-term-a-sym}
\end{align}

If $q=1$, we may use the contraction principle (i.e., Lemma~\ref{lem:contraction-principle2} applied with functions $\varphi_i(t)=|t|$)
conditionally to obtain
\begin{align}
	\EE \sup_{x\in B_{p}^n} \sum_{i=1}^m \varepsilon_i |\langle X_i,x \rangle|
	& \leq
	\EE \sup_{x\in B_{p}^n} \sum_{i=1}^m  \varepsilon_i\langle X_i,x \rangle\nonumber\\
	&= \EE \sup_{x\in B_{p}^n} \sum_{j=1}^n x_j \sum_{i=1}^m  a_{ij}\cdot  \varepsilon_i g_{ij}
	 = \EE \sup_{x\in B_{p}^n} \sum_{j=1}^n x_j \sum_{i=1}^m  a_{ij} g_{ij}.	\label{eq:general-contraction}
\end{align}
For $p > 1$, we have
\begin{align}
 \EE \sup_{x\in B_{p}^n} \sum_{j=1}^n x_j \sum_{i=1}^m  a_{ij} g_{ij} &=\EE \Bigl( \sum_{j=1}^n\Big|\sum_{i=1}^m a_{ij}g_{ij}\Big|^{p^*}\Bigr)^{1/p^*} \nonumber\\
	& \leq \Bigl( \sum_{j=1}^n \EE \Big|\sum_{i=1}^m a_{ij}g_{ij}\Big|^{p^*}\Bigr)^{1/p^*}
	= \gamma_{p^*} \Bigl(\sum_{j=1}^n \Bigl(\sum_{i=1}^m a_{ij}^2\Bigr)^{p^*/2}\Bigr)^{1/p^*}.\label{eq:q=1-second-term-b-contr}
\end{align}
Moreover, we have
\begin{align}
	\Bigl(\sum_{j=1}^n \Bigl(\sum_{i=1}^m a_{ij}^2\Bigr)^{p^*/2}\Bigr)^{1/p^*}
	& = \sup_{\delta\in\{-1,1 \}^m} \Bigl\|\Bigl(\sum_{i=1}^ma_{ij}^2\delta_i\Bigr)_{j\le n} \Bigr\|_{p^*/2}^{1/2} \nonumber\\
	& = \bigl\| (A\hadprod A)^T\colon \ell_\infty^m \to \ell_{p^*/2}^n \Bigr\|^{1/2}.\label{eq:q=1-second-term-c-fin}
\end{align}
Inequalities~\eqref{eq:q=1-second-term-a-sym}--\eqref{eq:q=1-second-term-c-fin} give the estimate of the first term on the right-hand side of~\eqref{eq:q=1-centered}.
This ends the proof of the upper bound for $p > 1$.

If $p = 1$, then letting $g_1,\ldots,g_n$ be i.i.d.\ standard Gaussian random variables, we have
\begin{align}	\label{eq:proof_q=1_last}	\nonumber
\EE \sup_{x\in B_{p}^n} \sum_{j=1}^n x_j \sum_{i=1}^m  a_{ij} g_{ij} &= \EE \max_{j\le n} \Big|\sum_{i=1}^m a_{ij} g_{ij}\Big| \\
&= \EE \max_{j\le n} g_j b_j \asymp \max_{j\le n} (\sqrt{\ln(j+1)}b_j^{\rearr{}}),
\end{align}
where the last step follows from Lemmas~\ref{lem:vH2017-3} and \ref{lem:vH2017-4} with $b_j \coloneqq \|(a_{ij})_{i\le m}\|_2$, $j\leq n$.
Putting together \eqref{eq:q=1-centered}--\eqref{eq:general-contraction} and \eqref{eq:proof_q=1_last} completes the proof of the upper bound in the case $p=1$.

The lower bound in the case $p>1$ follows from Proposition~\ref{prop:cm-lower-bound} and Corollary~\ref{cor:cm-lower-bound} below. In the case $p=1$, we use Proposition~\ref{prop:cm-lower-bound}, note that
\[
\EE \| G_A \colon \ell_{p}^n \to\ell_1^m \| \geq\EE \sup_{x\in B_{p}^n} \sum_{j=1}^n x_j \sum_{i=1}^m  a_{ij} g_{ij},
\]
and use \eqref{eq:proof_q=1_last} to obtain a lower bound.
\end{proof}

Now we deal with another special case, the one where $p=1$.

\begin{proof}[Proof of Proposition \ref{prop:gauss-p=1-q<2}] Recall that we deal with the range $p=1\leq q\le 2$.
 Using the structure of extreme points of $B_1^n$
 we get 
\begin{displaymath}
  \EE \|G_A \colon \ell_1^n \to \ell_q^m\| = \EE \max_{j\le n} \|(a_{ij}g_{ij})_{i\le m}\|_q.
\end{displaymath}
Denote $Z_j =  \|(a_{ij}g_{ij})_{i\le m}\|_q$. By well-known tail estimates of norms of Gaussian variables with values in Banach spaces (see, e.g., \cite[Corollary 1]{MR1388035} for a more general formulation) we get for all $t > 0$,

\begin{align}
  \PP\bigl( Z_j \ge C(\EE Z_j + \sqrt{t} b_j)\bigr) &\le e^{-t},\label{eq:lq-norms-first-ineq} \\
  \PP\bigl( Z_j \ge c(\EE Z_j + \sqrt{t} b_j)\bigr) &\ge \min(c,e^{-t}), \label{eq:lq-norms}
\end{align}
where $c, C$ are universal positive constants, and
\begin{displaymath}
  b_j^2 = \|(a_{ij}^2)_{i\le m}\|_{q/(2-q)} =\| (a_{ij}^2)_{i\le m}\|_{(q^\ast/2)^\ast} =  \sup_{x\in B_{q^\ast}^m} \sum_{i=1}^m a_{ij}^2x_i^2.
\end{displaymath}

Inequality \eqref{eq:lq-norms-first-ineq} shows in particular that the random variables $(Z_j - C \EE Z_j)_+$ satisfy 
\begin{displaymath}
  \PP((Z_j - C \EE Z_j)_+ \ge t) \le \exp\Bigl(-\frac{t^2}{C^2 b_j^2}\Bigr)
\end{displaymath}
for all $t > 0$, thus by Lemma \ref{lem:vH2017-3} we get
\begin{align*}
  \EE \|G_A \colon \ell_1^n \to \ell_q^m\| = \EE \max_{j\le n} Z_j 
  &\le C \max_{j\le n} \EE Z_j + \max_{j\le n} (Z_j-C\EE Z_j)_+ \\
 &\lesssim \Big(\max_{j\le n} \EE Z_j + \max_{j\le n} (\sqrt{\ln(j+1)}b_j^{\rearr{}})\Big),
\end{align*}
which together with the observation (following from Lemma~\ref{lem:special-norms} and the fact that $1=p\leq q\leq 2$) that
\begin{displaymath}
  \EE Z_j \le \Big(\EE \sum_{i=1}^m |a_{ij}|^q |g_{ij}|^q\Big)^{1/q} 
  = \gamma_q \|(a_{ij})_{i\le m}\|_q 
  = \gamma_q \|A\hadprod A\colon \ell_{1/2}^n \to \ell_{q/2}^m\|^{1/2},
\end{displaymath}
proves the upper estimate of the proposition.

Using comparison of moments of norms of Gaussian random vectors, we also get
\begin{align}\label{eq:first-lower-bound-q<2}
\EE \|G_A \colon \ell_1^n \to \ell_q^m\| & \ge \max_{j\le n} \EE Z_j \gtrsim \max_{j\le n} (\EE Z_j^q)^{1/q} \nonumber \\
& = \gamma_q \|(a_{ij})_{i\le m}\|_q
  = \gamma_q \|A\hadprod A\colon \ell_{1/2}^n \to \ell_{q/2}^m\|^{1/2},
\end{align}
so to end the proof it is enough to show that
\begin{align}\label{eq:second-lower-bound-q<2}
\EE \|G_A \colon \ell_1^n \to \ell_q^m\| \ge \max_{j\le n}(\sqrt{\ln (j+1)}b_j^{\rearr{}}).
\end{align}
This will follow by a straightforward adaptation of the argument from the proof of Lemma \ref{lem:vH2017-4}. We may and do assume that the sequence $(b_j)_{j\leq n}$ is non-increasing in $j$. 
By~\eqref{eq:lq-norms} we have for any $j\le n$ and $k \ge 1$,
\begin{displaymath}
\PP(Z_j \ge c\sqrt{\ln(k+1)}b_j) \ge \frac{c'}{k}.
\end{displaymath}
Thus, since $b_j\geq b_k$ for all $j\le k$, we have for any $k\le n$,
\begin{align*}
\PP(\max_{j\le n} Z_j \ge \sqrt{\ln(k+1)}b_k ) & \geq \PP(\exists_{j\le k} \ Z_j \ge \sqrt{\ln(k+1)}b_j)\\
& \ge 1 - (1-c'/k)^k \ge 1- e^{-c'} > 0.
\end{align*}
Thus, 
\begin{displaymath}
  \EE \|G_A \colon \ell_1^n \to \ell_q^m\| = \EE \max_{j\le n} Z_j \gtrsim \sqrt{\ln(k+1)} b_k.
\end{displaymath}
Taking maximum over $k \le n$ gives \eqref{eq:second-lower-bound-q<2} and ends the proof.
\end{proof}

%
\subsection{Bounded random variables}
\label{sec:bdd-OLD}
%

Here we show how one can adapt the methods of~\cite{BGN}
to prove Proposition~\ref{prop:bdd-p-2-q-OLD}, i.e., a version of Corollary~\ref{prop:bounded-p-smaller-2-nopsi} in the special case of bounded random variables with better logarithmic terms and with  explicit numerical constants.
Following~\cite{BGN}, we start with a lemma.

\begin{lemma}
	\label{lem:BGN-2}
	Assume that $X$ is as in Proposition~\ref{prop:bdd-p-2-q-OLD}.
	Let $(b_j)_{j\leq n}\in \RR^n$ and suppose that $t_0$ is such that  $\bigl|\sum_{j=1}^n b_j X_{ij}\bigr|\leq t_0$ almost surely.
	Then, for all $q\geq 2$ and $0\leq t\leq {t_0^{2-q}}(4 \sum_{j=1}^n b_j^2)^{-1} $,
	\begin{equation}
	\label{eq:lem-BGN-2}
	\EE \exp\bigl(t\bigl|\sum_{j=1}^n b_j X_{ij}\bigr|^q \bigr)
	\leq 1 + C^q(q) \, t \, \bigl(\sum_{j=1}^n b_j^2  \bigr)^{q/2},
	\end{equation}
	where $C(q) \coloneqq
	2(q \Gamma(q/2))^{1/q} \asymp \sqrt{q} $.
\end{lemma}

\begin{proof}
	Without loss of generality we may and do assume that $\sum_{j=1}^n b_j^2=1$.
	
	Since $q\geq 2$, for $s\in[0,t_0]$ and $t\in [0,\frac14 t_0^{2-q}]$ we have  $ts^q - s^2/2 \leq - s^2/4$.
	Thus,
	integration by parts, our assumption  $0\leq \bigl|\sum_{j=1}^n b_j X_{ij}\bigr|\leq t_0$ a.s.,
	and Hoeffding's inequality (i.e., Lemma~\ref{lem:BGN-1-Hoeffding}) yield
	\begin{align*}
	\EE \exp\bigl(t\bigl|\sum_{j=1}^n b_j X_{ij}\bigr|^q \bigr)
	&=  1 + qt \int_0^{t_0} s^{q-1} \exp(ts^q) \PP\bigl(\bigl|\sum_{j=1}^n b_j X_{ij}\bigr|\geq s \bigr) ds\\
	&\leq 1 + 2qt \int_0^{t_0} s^{q-1} \exp(ts^q - s^2/2) ds\\
	&\leq 1 + 2qt \int_0^{\infty} s^{q-1} \exp(- s^2/4) ds\\	
	& = 1 + t 2^q q \Gamma(q/2).
	\qedhere
	\end{align*}
\end{proof}

\begin{proof}[Proof of Proposition~\ref{prop:bdd-p-2-q-OLD}]
	We start with a bunch of reductions.
	Set
	\begin{align*}
	a&\coloneqq \|A\hadprod A \colon \ell^n_{p/2} \to \ell^m_{q/2}\|^{1/2} = \max_{j\leq n} \| (a_{ij})_{i=1}^m\|_q, \\
	b&\coloneqq \|(A\hadprod A)^T \colon \ell^m_{q^*/2} \to \ell^n_{p^*/2}\|^{1/2} = \max_{i\leq m} \| (a_{ij})_{j=1}^n\|_{p^*}.
	\end{align*}	
	(The equalities follow from Lemma~\ref{lem:special-norms}, since $p/2\leq 1\leq q/2$ and $ q^*/2 \leq 1\leq p^*/2$).
		Let $K$ be the set defined in Lemma~\ref{lem:set-K},
	so that $B_p^n\subset \ln(en)^{1/p^*} K$.
	Then
	\begin{equation}
	\label{eq:bdd-reduction-to-K-OLD}
	\|X_A \colon \ell^n_p\to \ell^m_q\|
	=  \sup_{x\in B_p^n} \|X_A x\|_q
	\leq  \ln(en)^{1/p^*}\! \sup_{x\in \Ext (K)}\! \|X_A x\|_q,
	\end{equation}
	where  $\Ext(K)$ is the set of
	extreme points of $K$.

	Consider first a fixed $x=(x_j)_{j=1}^n \in \Ext(K) \subset B_p^n$.
	We have
	\begin{equation}
	\label{eq:bdd-G_Ax-q-norm-distr}
	\|X_A x\|_q^q
	=\sum_{i=1}^m \Bigl| \sum_{j=1}^n a_{ij}X_{ij} x_j\Bigr|^q.
	\end{equation}
	Denote
	\begin{align*}
	t_0 &\coloneqq b = \max_{i\leq m} \| (a_{ij})_{j=1}^n\|_{p^*},\\
	t &\coloneqq \frac{t_0^{2-q}}{4 \max_{i\leq m} \| (a_{ij}x_j)_{j=1}^n\|_{2}^2}.
	\end{align*}
	Then, by the boundedness of $X_{ij}$ and by H{\"o}lder's inequality,  for every $i\leq m$,
	\[
	\bigl|\sum_{j=1}^n a_{ij}x_j X_{ij}\bigr| \leq \sum_{j=1}^n |a_{ij}| |x_j| \leq  \| (a_{ij})_{j=1}^n\|_{p^*}  \| (x_{j})_{j=1}^n\|_{p} \leq t_0.
	\]
	We can now apply, for every $i\leq m$,  Lemma~\ref{lem:BGN-2} (with $t$ and $t_0$ as above and with coefficients $b_{ j} = a_{ij}x_j$).
	Since the random variables $\bigl|\sum_{j=1}^n a_{ij}x_j X_{ij}\bigr| $, $i\leq m$, are independent,
	using Lemma~\ref{lem:BGN-2} yields
	\begin{align*}
	\EE  \exp\bigl(t \sum_{i=1}^m \bigl|\sum_{j=1}^n a_{ij}x_j  X_{ij}\bigr|^q \bigr)
	&= \prod_{i=1}^m \Bigl[ \EE  \exp\bigl(t \bigl|\sum_{j=1}^n a_{ij}x_j X_{ij}\bigr|^q \bigr)\Bigr] \\
	&\leq \prod_{i=1}^m \Bigl(1 + C^q(q) \, t \,   \bigl(\sum_{j=1}^n a_{ij}^2x_j^2  \bigr)^{q/2}\Bigr) \\
	&\leq \exp\Bigl( C^q(q)\, t\sum_{i=1}^m  \bigl(\sum_{j=1}^n a_{ij}^2x_j^2   \bigr)^{q/2}\Bigr)
	\leq \exp\bigl( C^q(q)\, t   a^q\bigr),
	\end{align*}
	where in the last step we used the definition of  $a =\|A\hadprod A \colon \ell^n_{p/2} \to \ell^m_{q/2}\|^{1/2}$ (and the fact that $x\in B^n_{p}$).
	By Chebyshev's inequality and \eqref{eq:bdd-G_Ax-q-norm-distr}, we have, for every $s\geq 0$,
	\[
	\PP\bigl(t \|X_{A} x\|_q^q \geq \ln\bigl[ \EE  \exp\bigl(t \sum_{i=1}^m \bigl|\sum_{j=1}^n a_{ij}x_j  X_{ij}\bigr|^q \bigr)\bigr] +sk \bigr)
	\leq e^{-sk}.
	\]
	Combining this with the previous estimate yields, for every $s\geq 0$,
	\[
	\PP\bigl(\|X_{A} x\|_q^q \geq C^q(q)\,  a^q + \frac{ sk}{t} \bigr)
	\leq e^{-sk}.
	\]
	
	Recall that $x\in \Ext(K)$.
	Thus, there exists  an index set $J\subset \{1,\dots,n\}$ of cardinality $k\leq n$, such that $x_j = \frac{\pm 1}{k^{1/p}}$ for $j\in J$ and $x_j=0$ for $j\notin J$.
	We use the definition of $t$ and the inequality between the arithmetic mean and the power mean of order $p^*/2\geq 1$ (recall that $|J|=k$ and $p\leq 2$) to get
	\begin{align*}
	\frac{1}{4 t} = b^{q-2} \max_{i\leq m} \| (a_{ij}x_j)_{j=1}^n\|_{2}^2
	& = b^{q-2} k^{-2/p} \max_{i\leq m} \sum_{j\in J} a_{ij}^2\\
	&\leq  b^{q-2} k^{-2/p + 1 - 2/p^*} \max_{i\leq m} \bigl(\sum_{j\in J} |a_{ij}|^{p^*}\bigr)^{2/p^*} = b^q k^{-1}.
	\end{align*}
	Putting everything together, we obtain
	\begin{equation}
	\label{eq:bdd-concentration_a_b}
	\PP\bigl(\|X_{A} x\|_q^q \geq C^q(q)\,  a^q + 4 b^q s \bigr)
	\leq e^{-sk}
	\end{equation}
	for all $s\geq 0$ and all $x\in \Ext(K)$ with support of cardinality $k$.
	
	For any $k\leq n$,
	there are $2^k \binom{n}{k} \leq 2^k n^k \leq \exp( k \ln(en))$ vectors in $\Ext(K)$
	with support of cardinality $k$.
	Thus, using the union bound and \eqref{eq:bdd-concentration_a_b}, we see that, for all $s\geq 2$,
	\begin{multline*}
	\PP\bigl(\sup_{x\in \Ext K}\!\|X_A x\|_q^q \geq C^q(q)\,  a^q + 4 b^q \ln(en) s \bigr)
	\leq \sum_{k=1}^n \exp( -k \ln(en)(s-1))\\
	\leq n \exp( - \ln(en)(s-1))
	= n(en)^{-s+1}\leq e^{-s+1}.
	\end{multline*}
	Hence, by Lemma~\ref{lem:mean-median},
	\begin{align*}
	\EE \!\! \sup_{x\in \Ext K}\!\!\ \|X_{A} x\|_q
	\leq \bigl(\EE\!\! \sup_{x\in \Ext K}\!\! \|X_{A} x\|_q^q\bigr)^{1/q}
	&\leq  \bigl( C^q(q)\,  a^q + 4 b^q \ln(en) (2 + e\cdot e^{-2}) \bigr)^{1/q}\\
	&\leq C(q)  a + 10^{1/q}\ln(en)^{1/q}b.
	\end{align*}
	Recalling \eqref{eq:bdd-reduction-to-K-OLD} and the definitions of $a$, $b$, and $C(q)$ yields the assertion.
\end{proof}

\begin{remark}
	\label{rem:BGN-extended-for-all-p,q}
	In the unstructured case, for $X_{ij}$ which are independent, mean-zero,  and take values in $[-1,1]$, it is easy to extend \eqref{eq:BGN} to the whole range of $p, q\in [1,\infty]$ (see~\cite{Bennett,CarlMaureyPuhl}).
	Indeed, for $p\geq 2$ and $q\geq 2$,
	\begin{align*}
	\EE \| X\colon \ell^n_p \to \ell^m_q\|
	&\leq
	\| \ell^n_p \hookrightarrow \ell^n_{2}\| \cdot	\EE \| X\colon \ell^n_2 \to \ell^m_q\| \\
	&\lesssim_q n^{1/2-1/p} \cdot \max\{n^{1/2}, m^{1/q}\}
	=      \max\{n^{1-1/p}, n^{1/2-1/p}m^{1/q}\} .
	\end{align*}
	Thus, for $p\geq 2$ and $1\leq q\leq 2$,
	\begin{align*}
	\EE \| X\colon \ell^n_p \to \ell^m_q\|
	&\leq
	\EE \| X\colon \ell^n_p \to \ell^m_2\| \cdot
	\| \ell^m_{2} \hookrightarrow \ell^m_q\| \\
	&\lesssim_q  \max\{n^{1-1/p}, n^{1/2-1/p}m^{1/2}\} \cdot  m^{1/q- 1/2}\\
	&=	\max\{n^{1-1/p}m^{1/q-1/2}, n^{1/2-1/p}m^{1/q}\}.
	\end{align*}
	Suppose now that $1\leq p \leq 2\leq q \leq \infty$ and $1/p+1/q\leq 1$ (i.e., $q\geq p^*$).
	Choose $\theta\in[0,1]$ and $r \geq 2$ so that $\frac{1}{p} = \frac{\theta}{2} + \frac{1-\theta}{1}$ and $\frac{1}{q} = \frac{\theta}{r} +  \frac{1-\theta}{\infty}$,
	i.e., $\theta = 2/p^*$ and $r = 2q/p^*$.
	Using the Riesz--Thorin interpolation theorem,
	 the fact that $\| X\colon \ell^n_1 \to \ell^m_\infty\|\leq 1$  (since the entries take values in $[-1,1]$), 
	 and Jensen's inequality, we arrive at
	\begin{align*}
	\EE \| X\colon \ell^n_p \to \ell^m_q\|
	&\leq
	\EE \| X\colon \ell^n_2 \to \ell^m_r\|^\theta \| X\colon \ell^n_1 \to \ell^m_\infty\|^{1-\theta}\\
	&\leq  \EE \| X\colon \ell^n_2 \to \ell^m_r\|^\theta \leq  \bigl(\EE \| X\colon \ell^n_2 \to \ell^m_r\|\bigr)^\theta \\
	&\leq \max\{n^{1/2}, m^{1/r}\}^{\theta} =  \max\{n^{1/p^*}, m^{1/q}\}.
	\end{align*}
	The estimates in the remaining ranges of $p, q$ follow by duality \eqref{eq:duality}.
	Moreover, up to constants, all these estimates are optimal, as they can be reversed for matrices with $\pm 1$ entries (see \cite[Proposition~3.2]{Bennett} or \cite[Satz~2]{CarlMaureyPuhl}).
\end{remark}

%
\subsection{\texorpdfstring{$\psi_r$}{Psi-r} random variables}
\label{sec:psi-r}
%

In this section, we prove Theorem~\ref{thm:psi-p<2}.
To this end we shall split the matrix $X$ into two parts $X^{(1)}$ and $X^{(2)}$ such that all entries of $X^{(1)}$ are bounded by $C\ln(mn)^{1/r}$. Then, we shall deal with $X^{(2)}$ using the following crude bound and the fact that the probability that $X^{(2)} \neq 0$ is very small. In order to bound the expectation of the norm of $X^{(1)}$ we need a cut-off version of Theorem~\ref{thm:psi-p<2} --  see Lemma~\ref{prop:bounded-p-smaller-2} below.

\begin{lemma}
	\label{lem:cm-prelim-bound-2}
	Let $r\in(0,2]$. Assume that $X = (X_{ij})_{i\leq m, j\leq n}$ satisfies the assumptions of Theorem~\ref{thm:psi-p<2}.
	Then
	\[
	\bigl(\EE \|X_A\colon \ell_{p}^n\to\ell_q^m\|^2\bigr)^{1/2}
	\lesssim_{r,K,L} (m+n)^{1/r}	\|A\hadprod A \colon \ell^n_{p/2} \to \ell^m_{q/2} \|^{1/2}.
	\]
\end{lemma}

\begin{proof}
	By a standard volumetric estimate (see, e.g., \cite[Corollary~4.2.13]{V2018}),
	we know that there exists (in the metric $\|\cdot\|_p$) a $1/2$-net $S$   in $B_p^n$ of size at most $5^n$.
	In other words, for any $x\in B_p^n$ there exists $y\in S$ such that $x-y \in \frac{1}{2} B_p^n$.
		Thus, for any $z\in\RR^n$,
		\begin{align*}
			\sup_{x\in B_p^n} \sum_{j=1}^n x_jz_j
			&\leq \sup_{x\in B_p^n} \min_{y\in S} \sum_{j=1}^n (x_j-y_j) z_j  + \sup_{y\in S} \sum_{j=1}^n y_j z_j \\
      	&\leq \sup_{u\in \frac{1}{2} B_p^n} \sum_{j=1}^n u_j z_j  + \sup_{y\in S} \sum_{j=1}^n y_j z_j      	= 	 \frac{1}{2}\sup_{x\in  B_p^n} \sum_{j=1}^n x_j z_j  + \sup_{y\in S} \sum_{j=1}^n y_j z_j .
		\end{align*}
	Hence,
	\begin{equation}	\label{eq:eps-nets1}
	\sup_{x\in B_p^n} \sum_{j=1}^n x_jz_j \leq 2 \sup_{y\in S} \sum_{j=1}^n y_j z_j .
	\end{equation}
	Likewise, if we denote by $T$ the $1/2$-net in $B_{q^*}^m$ (in the metric $\|\cdot\|_{q^*}$)  of size at most $5^m$,
	then
	\begin{equation}	\label{eq:eps-nets2}
	\sup_{x\in B_{q^*}^m} \sum_{i=1}^m x_iz_i \leq 2 \sup_{y\in T} \sum_{i=1}^m y_i z_i .
	\end{equation}
	Combining these two estimates, we see that
	\begin{align}
		\label{eq:lc-prelim-bound-1}
		\bigl(\EE \|X_A\colon \ell_{p}^n\to\ell_q^m\|^2\bigr)^{1/2}
		&= \bigl(
		\EE \sup_{x\in B_p^n, y\in B_{q^*}^m} \bigl( \sum_{i=1}^m \sum_{j=1}^n  y_i a_{ij} X_{ij} x_j\bigr)^2
		\bigr)^{1/2} \\
			& \leq 4			\bigl(\EE \sup_{x\in S, y\in T} \bigl( \sum_{i=1}^m \sum_{j=1}^n  y_i a_{ij} X_{ij} x_j\bigr)^2\bigr)^{1/2}.\nonumber
				\end{align}

    Lemma \ref{lem:psi_prelim5} implies that for any $x\in\RR^n$, $y\in\RR^m$, the random variable
   	\[
   		Z(x,y)\coloneqq \bigl(\sum_{i,j} y_i^2a_{ij}^2x_j^2 \bigr)^{-1/2}  \sum_{i=1}^m \sum_{j=1}^n  y_i a_{ij} X_{ij} x_j
	\]
	satisfies condition (i) in Lemma~\ref{lem:psi-r-equivalence}.
   Thus, Lemma~ \ref{lem:psi-r-equivalence} implies that
   	\begin{equation}	\label{eq:lc-prelim-bound-2}
		\EE  \exp\Bigl( c(r,K,L)  \bigl(\sum_{i,j} y_i^2a_{ij}^2x_j^2 \bigr)^{-r/2} \Bigl( \sum_{i=1}^m \sum_{j=1}^n  y_i a_{ij} X_{ij} x_j\Bigr)^{r}\Bigr) \le C(r,K,L),
	\end{equation}
	where $c(r,K,L) \in(0,\infty)$ and $C(r,K,L) \in(0,\infty)$ depend only on $r$, $K$, and $L$.

	The function $z\mapsto e^{z^{r/2}}$ is convex on $[(2r^{-1}-1)^{2/r},\infty)$. Therefore,
	by Jensen's inequality, for any $u> 0$ and any nonnegative random variable $Z$,
	\begin{align*}
	\exp\bigl(u (\EE Z^2)^{r/2}\bigr)
	&\leq \exp\bigl( (u^{2/r}\EE Z^2 + (2r^{-1}-1)^{2/r})^{r/2}\bigr)\\
	&\leq \EE 	 \exp\bigl( (u^{2/r} Z^2 + (2r^{-1}-1)^{2/r})^{r/2}\bigr)\\
	&\leq \EE 	 \exp\bigl( u Z^r + (2r^{-1}-1) \bigr) \leq  e^{2/r} \EE \exp(u Z^r).
	\end{align*}
	Hence,
	\[
    (\EE Z^2)^{1/2} \le u^{-1/r} \Bigl(\ln \bigl(e^{2/r} \EE \exp(u Z^{r})\bigr)\Bigr)^{1/r}.
	\]
	Thus, when
	\[	
		u\coloneqq c(r,K,L) \bigl(\max_{x\in S, y\in T} \sum_{i,j} y_i^2a_{ij}^2x_j^2 \bigr)^{-r/2},
	\]
	 we get by \eqref{eq:lc-prelim-bound-2}, \eqref{eq:manipulations-to-fix1}, and \eqref{eq:manipulations-to-fix2},
	\begin{align*}
		\MoveEqLeft[4]
		\bigl(\EE \sup_{x\in S, y\in T}  \bigl(\sum_{i=1}^m \sum_{j=1}^n  y_i a_{ij} X_{ij} x_j\bigr)^2\bigr)^{1/2} \\
		&\leq u^{-1/r} \ln^{1/r} \Bigl(e^{2/r}\EE \exp\Bigl( c(r,K,L) \sup_{x\in S, y\in T}Z(x,y)^{r}\Bigr)\Bigr) \\
		&\leq u^{-1/r} \ln^{1/r} \Bigl(e^{2/r}\EE \!\!  \sum_{x\in S, y\in T}\!\! \exp\bigl( c(r,K,L) Z(x,y)^{r}\bigr)\Bigr) \\
		&\leq u^{-1/r} \ln^{1/r} \Bigl(e^{2/r}|S||T|  C(r,K,L)\Bigr)  \\
		&\leq \frac 1{c(r,K,L)} \max_{x\in S, y\in T}\bigl( \sum_{i,j} y_i^2a_{ij}^2x_j^2 \bigr)^{1/2} \ln^{1/r} \Bigl(e^{2/r} 5^{m} 5^{n} C(r,K,L)\Bigr) \\
		 & \overset{\mathclap{\eqref{eq:manipulations-to-fix1},\ \eqref{eq:manipulations-to-fix2}}}{\lesssim_{r,K,L}} \ \|A\hadprod A \colon \ell^n_{p/2} \to \ell^m_{q/2} \|^{1/2} \bigl(m+n+ \widetilde{C}(r,K,L)\bigr)^{1/r},
	\end{align*}
	where in the last two inequalities we also used inequalities $|S| \leq 5^n$ and $|T|\leq 5^m$, and the inclusions $S\subset B_p^n$, $T\subset B_{q^*}^m$.
    	Recalling \eqref{eq:lc-prelim-bound-1} completes the proof.
\end{proof}

The following cut-off version of Theorem~\ref{thm:psi-p<2}
can be proved similarly as  Proposition~\ref{prop:gauss-p-smaller-2}.

\begin{lemma}
	\label{prop:bounded-p-smaller-2}
	Let $K,L, M>0$ and $r\in(0,2]$.
	Assume $X=(X_{ij})_{i\leq m, j\leq n}$ is a~random matrix
	with independent symmetric entries taking values in $[-M,M]$
	and satisfying the condition
	\begin{equation}	\label{eq:prop-assumption-lin-psi}
	\PP\bigl(|X_{ij}| \ge t\bigr) \le Ke^{-t^r/L} \quad \text{for all } t\ge 0.
	\end{equation}
		Then, for $1\leq p\leq 2$ and $1\leq q< \infty$, we have
\begin{align*}
\EE \|X_A \colon \ell^n_p\to \ell^m_q\|
& \lesssim q^{1/r} C(r,K,L) \ln(en)^{1/p^*} \|A\hadprod A \colon \ell^n_{p/2} \to \ell^m_{q/2}\|^{1/2}\\
 &\qquad + M \ln(en)^{1/2+1/p^*}  \|(A\hadprod A)^T \colon \ell^m_{q^*/2} \to \ell^n_{p^*/2}\|^{1/2}. \nonumber
 \end{align*}
\end{lemma}

\begin{proof}
		Fix $1\leq p\leq 2$ and $1\leq q\leq \infty$.
	Let $K$ be the set defined in Lemma~\ref{lem:set-K}
	so that $B_p^n\subset \ln(en)^{1/p^*} K$.
	Then
	\begin{equation}
	\label{eq:lc-cutoff-reduction-to-K}
	\|X_A \colon \ell^n_p\to \ell^m_q\|
	=  \sup_{x\in B_p^n} \|X_A x\|_q
	\leq  \ln(en)^{1/p^*}\! \sup_{x\in \Ext (K)}\! \|X_A x\|_q,
	\end{equation}
	where  $\Ext(K)$ is the set of  extreme points of $K$.
	We shall now estimate the expected value of the right-hand side of \eqref{eq:lc-cutoff-reduction-to-K}.

	To this end, we consider a fixed $x=(x_j)_{j=1}^n \in \Ext(K)$.
	This means that there exists a non-empty index set $J\subset \{1,\dots,n\}$ of cardinality $k\leq n$ such that $x_j = \frac{\pm 1}{k^{1/p}}$ for $j\in J$ and $x_j=0$ for $j\notin J$. We know from \eqref{eq:lipschitz} that the Lipschitz constant of the convex function
	\[
		z=(z_{ij})_{ij} \mapsto \Bigl\| \Bigl( \sum_{j=1}^n a_{ij}z_{ij} x_j \Bigr)_i \Bigr\|_q = \sup_{y\in B_{q^*}^m}\sum_{i=1}^m \sum_{j=1}^n y_ia_{ij}z_{ij}x_j
	\]	
	is less than or equal to
	\[
		 \frac{1}{k^{1/p}}\sqrt{ \sup_{y\in B^m_{q^*/2}}  \sum_{i=1}^m  \sum_{j\in J} y_i a_{ij}^2  }\,
    		\eqqcolon  \frac{ b_J}{k^{1/p}}.
	\]
	Thus, Talagrand's concentration for convex functions and random vectors with  independent bounded coordinates (see \cite[Theorem 6.6 and Equation (6.18)]{Talagrand96}), together with the inequality $\operatorname{Med}(|Z|)\leq 2\EE |Z|$, implies
	\begin{equation}	\label{eq:lc-cutoff-norm-concentration-claim}
		\PP (\|X_Ax\|_q \ge 2\EE \|X_Ax\|_q + t) \le 4\exp\Bigl(-\frac{k^{2/p}t^2}{16M^2b_J^2}\Bigr) \qquad \text{for all } t\ge 0.
	\end{equation}
Similar to the proof in the Gaussian case (i.e., proof of Proposition \ref{prop:gauss-p-smaller-2}), we shall transform this into a more convenient form
by getting rid of $b_J$
and
estimating  $\EE \|X_A x\|_q $. Let us denote, for each $i\in\{1,\dots,m\}$,
	\begin{align*}
	Z_i
	       &\coloneqq  \sum_{j=1}^n a_{ij}X_{ij} x_j.
	\end{align*}
From our assumption \eqref{eq:prop-assumption-lin-psi} as well as
 Lemmas \ref{lem:psi_prelim5} and \ref{lem:psi-r-equivalence}, we obtain that   $(\EE|Z_i|^q)^{1/q}\lesssim_{r,K,L}q^{1/r}\sqrt{\sum_{j=1}^n a_{ij}^2 x_j^2}$. Hence,

\begin{align*}	
	\EE \|X_A x\|_q
	& \leq \bigl(\EE \|X_Ax\|_q^q\bigr)^{1/q}
	= \bigl(\sum_{i=1}^m \EE |(X_Ax)_i|^q\bigr)^{1/q} 	
	\lesssim_{r,K,L} q^{1/r} \Bigl( \sum_{i=1}^m \bigl(\sum_{j=1}^n a_{ij}^2 x_j^2\bigr)^{q/2}  \Bigr)^{1/q} \\
	& \leq   q^{1/r}  \sup_{z\in B_p^n} \Bigl( \sum_{i=1}^m \Bigl| \sum_{j=1}^n a_{ij}^2 z_j^2\Bigr|^{q/2}\Bigr)^{1/q}
 	 =  q^{1/r} \|A\hadprod A \colon \ell^n_{p/2} \to \ell^m_{q/2}\|^{1/2}  \eqqcolon q^{1/r} a.
\end{align*}

From \eqref{eq:kbj}, we see that
\[
	k^{2/p^*-1} b_J^2
	\leq  \|(A\hadprod A)^T \colon \ell^m_{q^*/2} \to \ell^n_{p^*/2}\| \eqqcolon b^2.
\]

	The above two inequalities together with estimate \eqref{eq:lc-cutoff-norm-concentration-claim}
	(applied with  $t=4k^{\frac 1{p^*}-\frac 12}b_J M\sqrt{\ln(en)} s$), imply that
	\begin{equation}
		 \label{eq:lc-cutoff-concentration_a_b}
	\PP\bigl(\|X_A x\|_q \geq C(r,K,L) q^{1/r}a +   4bM \sqrt{\ln(en)} s \bigr)
	 \leq  4 \exp\bigl( - k \ln(en)s^2\bigr)
	\end{equation}
	for every $s\geq 0$ and any $x\in \Ext(K)$ with support of cardinality $k$.
	
	For any $k\leq n$,
	there are $2^k \binom{n}{k} \leq 2^k n^k \leq \exp( k \ln(en))$ vectors in $\Ext(K)$
	with support of cardinality $k$.
	Thus, using the union bound and \eqref{eq:lc-cutoff-concentration_a_b}, we see that for $s\geq \sqrt{2}$,
	\begin{multline*}
	\PP\bigl(\sup_{x\in \Ext K}\! \|X_A x\|_q \geq
	C(r,K,L) q^{1/r}a +   4bM \sqrt{\ln(en)} s\bigr)
	\leq
	4\sum_{k=1}^n \exp( -k \ln(en)(s^2-1))	\\
	\leq 4n\exp( - \ln(en)(s^2-1))
	= 4n(en)^{-s^2+1}\leq 4 e^{-s^2+1}.
	\end{multline*}
	Hence, by Lemma~\ref{lem:mean-median},
	\begin{align*}
	\EE\! \sup_{x\in \Ext K}\! \|X_A x\|_q
	& \leq C(r,K,L)q^{1/r}a +  4b M \sqrt{\ln(en)}  \Bigl(\sqrt{2} +4  e \frac{e^{-2}}{2\sqrt{2}}\Bigr).
	\end{align*}
	Recalling \eqref{eq:lc-cutoff-reduction-to-K} and the definitions of $a$ and $b$ yields the assertion.
\end{proof}

\begin{proof}[Proof of Theorem~\ref{thm:psi-p<2}] 	
	By a symmetrization argument (as in the first  paragraph of the proof of Theorem~\ref{thm:main-psi-r}),
	we may and do assume  that all the entries $X_{ij}$ are symmetric.
		Set $M = (4 L \ln(mn)/r)^{1/r}$.
		Denote $\widehat{X}_{ij} = X_{ij} \ind_{\{|X_{ij}|\leq M\}}$
		and let $\widehat{X}$ be the $m\times n$ matrix with entries $\widehat{X}_{ij}$.
		We have
		\begin{align*}
		\EE \|X_A \colon \ell^n_p\to \ell^m_q\|
		&= \EE \|X_A \colon \ell^n_p\to \ell^m_q\| \ind_{\{ \max_{k,l} |X_{kl}|\leq M\}} \\
		&\quad+ \EE \|X_A \colon \ell^n_p\to \ell^m_q\| \ind_{\{ \max_{k,l} |X_{kl}|> M\}}.
		\end{align*}
		
				The random matrix $\widehat{X}$ satisfies the assumptions of Lemma~\ref{prop:bounded-p-smaller-2}. Thus, the first summand above can be estimated as follows:
		\begin{align*}
		\MoveEqLeft
		\EE \|X_A \colon \ell^n_p\to \ell^m_q\| \ind_{\{ \max_{k,l} |X_{kl}|\leq M\}}\\
		& = \EE \sup_{y\in B_{q^*}^m, x\in B_p^n} \bigl\{ \sum_{i=1}^m \sum_{j=1}^n y_i a_{ij} X_{ij} x_j \bigr\}
		\cdot \ind_{\{ \max_{k,l} |X_{kl}|\leq M\}} \\
		& = \EE \sup_{y\in B_{q^*}^m, x\in B_p^n}\bigl\{\sum_{i=1}^m \sum_{j=1}^n y_i a_{ij} X_{ij} \ind_{\{|X_{ij}|\leq M\}} x_j \bigr\}
		\cdot \ind_{\{ \max_{k,l} |X_{kl}|\leq M\}} \\
		&= \EE \|\widehat{X}_A \colon \ell^n_p\to \ell^m_q\| \ind_{\{ \max_{k,l} |X_{kl}|\leq M\}}
		\leq  \EE \|\widehat{X}_A \colon \ell^n_p\to \ell^m_q\| \\
		&\lesssim_{r, K, L} \
	       q^{1/r}\ln(en)^{1/p^*} \|A\hadprod A \colon \ell^n_{p/2} \to \ell^m_{q/2}\|^{1/2}\\
		&\qquad\qquad\ +M \ln(en)^{1/2+1/p^*}  \|(A\hadprod A)^T \colon \ell^m_{q^*/2} \to \ell^n_{p^*/2}\|^{1/2}
		\\
		&\lesssim_{r, K, L} \
	        q^{1/r} \ln(en)^{1/p^*} \|A\hadprod A \colon \ell^n_{p/2} \to \ell^m_{q/2}\|^{1/2}\\
		&\qquad\qquad\ +  \ln(mn)^{1/r} \ln(en)^{1/2+1/p^*}  \|(A\hadprod A)^T \colon \ell^m_{q^*/2} \to \ell^n_{p^*/2}\|^{1/2}.
		\end{align*}
		For the second summand we write,
		using the Cauchy--Schwarz inequality
		and then Lemma~\ref{lem:cm-prelim-bound-2}
		 and Lemma~\ref{lem:max-psi-r} (with $k=mn$ and $v=4/r$; recall that $M = (4 L \ln(mn)/r)^{1/r}$),
		\begin{align*}
		\MoveEqLeft
		\EE \|X_A \colon \ell^n_p\to \ell^m_q\| \ind_{\{ \max_{k,l} |X_{kl}|> M\}}\\
		&\leq \bigl(	\EE \|X_A \colon \ell^n_p\to \ell^m_q\|^2\bigr)^{1/2}
		\PP( \max_{k\leq m, l\leq n} |X_{kl}| > M)^{1/2}\\
		& \lesssim_{r, K, L}  (m+n)^{1/r}	\|A\hadprod A \colon \ell^n_{p/2} \to \ell^m_{q/2} \|^{1/2}  \cdot (mn)^{-2/r +1/2}\\
        & \lesssim_{r} 	\|A\hadprod A \colon \ell^n_{p/2} \to \ell^m_{q/2} \|^{1/2}.
		\end{align*}
		Combinging the above three inequalities ends the proof.
\end{proof}

%
\section{Lower bounds and further discussion of conjectures}
\label{sec:lower-bounds}
%

\subsection{Lower bounds}
\label{subsec:lower-bounds}
Let us first provide lower bounds showing that the upper bounds obtained above
are indeed  sharp (up to logarithms).

\begin{proposition}
	\label{prop:cm-lower-bound}
	Let $X = (X_{ij})_{i\leq m, j\leq n}$ be a random matrix with independent  mean-zero entries satisfying $\EE |X_{ij}| \ge c$ for some $c\in(0,\infty)$.
	Then, for all $1\leq p, q\leq \infty$,
	\[
	\EE \|X_A\colon \ell_{p}^n\to\ell_q^m\|
	\geq \frac{c}{2\sqrt 2}\|A\hadprod A \colon \ell^n_{p/2} \to \ell^m_{q/2} \|^{1/2}.
	\]
\end{proposition}
Using duality \eqref{eq:duality}	  we immediately obtain the following corollary.

\begin{corollary}\label{cor:cm-lower-bound}
	Let $X = (X_{ij})_{i\leq m, j\leq n}$ be as in Proposition~\ref{prop:cm-lower-bound}. Then, for all $1\leq p,q\leq \infty$,
	\[
	\EE \|X_A\colon \ell_{p}^n\to\ell_q^m\|
	\geq \frac{c}{2\sqrt 2}\|(A\hadprod A)^T \colon \ell^m_{q^*/2} \to \ell^n_{p^*/2}\|^{1/2}.
	\]
\end{corollary}

\begin{proof}[Proof of Proposition \ref{prop:cm-lower-bound}]
Let $\| \cdot\|$ denote the operator norm from $\ell_{p}^n$ to $\ell_q^m$. For $i\in\{1,\dots,m\}$ and $j\in\{1,\dots,n \}$, let us denote by $E_{ij}$  the $m\times n$ matrix with entry $1$ at the intersection of $i$th row and $j$th column and with all other entries $0$.
	By the symmetrization trick described in Remark~\ref{rmk:symmetrization},
	 it suffices to consider matrices $X$ with symmetric entries
	 and  prove the assertion with a twice better constant $c/\sqrt 2$
	 (note that, also by Remark~\ref{rmk:symmetrization}, the lower bound for the absolute first moment of the symmetrized entries does not change and is still equal to $c$). 

If $X$ has symmetric independent entries, it has the same distribution as $(\varepsilon_{ij}|X_{ij}|)_{ij}$, where $\varepsilon_{ij}$, $i\leq m$, $j\leq n$, are i.i.d.~Rademacher random variables, independent of all other random variables. 
Hence,
 by Jensen's inequality and the contraction principle (Lemma \ref{lem:contraction-principle} applied  with $\alpha_{ij} = 1/\EE|X_{ij}| \le 1/c$ and $x_{ij}=a_{ij}\EE|X_{ij}| E_{ij}$), we get
\begin{align} \label{eq:cm-lower-bd-prestart}\nonumber
\EE \Bigl\|\sum_{i=1}^m \sum_{j=1}^n X_{ij} a_{ij}E_{ij} \Bigr\|
& \nonumber
= \EE \Bigl\|\sum_{i,j}\varepsilon_{ij}|X_{ij}|a_{ij}E_{ij} \Bigr\|
\ge \EE \Bigl\|\sum_{i,j} \varepsilon_{ij}\EE|X_{ij}|a_{ij} E_{ij} \Bigr\|  \\ &
\ge c\ \EE \Bigl\|\sum_{i,j} \varepsilon_{ij}a_{ij} E_{ij} \Bigr\|.
\end{align}
Thus, it suffices to estimate from below $\ \EE\|\sum_{i,j} \varepsilon_{ij}a_{ij} E_{ij} \|$.

Since the $\ell_q$ norm is unconditional, we obtain from the inequalities of Jensen and Khintchine (see \cite{Haagerup}) that
	\begin{align*}
		\EE\Bigl\|\sum_{i=1}^m \sum_{j=1}^n \varepsilon_{ij}a_{ij} E_{ij} \Bigr\|
		& = \EE \sup_{x\in B_p^n} \Bigl\| \bigl( \sum_{j=1}^n a_{ij}\varepsilon_{ij} x_j\bigr)_{i=1}^m \Bigr\|_q
		= \EE \sup_{x\in B_p^n} \Bigl\| \Bigl( \bigl| \sum_{j=1}^n a_{ij}\varepsilon_{ij} x_j \bigr|\Bigr)_{i=1}^m \Bigr\|_q \\
		& \geq \sup_{x\in B_p^n} \Bigl\| \Bigl( \EE \bigl| \sum_{j=1}^n a_{ij}\varepsilon_{ij} x_j \bigr|\Bigr)_{i=1}^m \Bigr\|_q\\
		&
		\hspace{-0,55 cm}\mathop{\geq}^{\text{Khintchine's}}_{\text{inequality}}   \frac{1}{\sqrt 2} \sup_{x\in B_p^n} \Bigl\| \Bigl( \bigl( \sum_{j=1}^n a_{ij}^2 x_j^2 \bigr)^{1/2} \Bigr)_{i=1}^m \Bigr\|_q \\
		& =  \frac{1}{\sqrt 2} \sup_{z\in B_{p/2}^n} \Bigl\| \Bigl( \sum_{j=1}^n a_{ij}^2 z_j  \Bigr)_{i=1}^m \Bigr\|_{q/2}^{1/2}
		 = \frac{1}{\sqrt 2} \|A\hadprod A \colon \ell^n_{p/2} \to \ell^m_{q/2} \|^{1/2}.
	\end{align*}
	This together with the estimate in \eqref{eq:cm-lower-bd-prestart} yields the assertion.
\end{proof}

Since $\|G_A\colon \ell_{p}^n\to\ell_q^m\| \ge \max_{i,j} |a_{ij}g_{ij}|$, it suffices to prove the following proposition in order to provide the lower bound in Conjecture~\ref{conj:conjecture_1}.

\begin{proposition}
	 For the $m\times n$ Gaussian matrix $G_A$, we have
	 \begin{equation} \label{eq:hippo-lower}
\EE \|G_A\colon \ell_{p}^n\to\ell_q^m\|
\gtrsim_{p,q} 
\begin{cases}
  \max_{j\le n}\sqrt{\ln (j+1)} b_j^{\rearr{}}  &    \text{if }\ p\leq q\leq 2,\\
  \max_{i\le m}\sqrt{\ln (i+1)} d_i^{\rearr{}}  &   \text{if } \ 2\leq p \leq q, \\
  0 & \text{otherwise,}
\end{cases} 
\end{equation}
where $b_j= \|(a_{ij})_{i\le m} \|_{2q/(2-q)}$ and $d_i=\| (a_{ij})_{j\le n}\|_{2p/(p-2)}$.
\end{proposition}

\begin{proof}
	Since $B_1^n\subset B_p^n$ for $p\ge 1$ and the $b_j$'s do not depend on $p$, it suffices to prove the first part of the assertion (in the range $p\leq q\leq 2$) only in the case $p=1\leq q\leq 2$. In this case \eqref{eq:hippo-lower} follows by Propostion~\ref{prop:gauss-p=1-q<2}.
	
	The assertion in the range $2\leq p \leq q$ follows by duality \eqref{eq:duality}.
\end{proof}

\subsection{The proof of Inequalities \texorpdfstring{\eqref{eq:bound_q<p_third_term_unnec} and \eqref{eq:bound_Emax_term}}{with D1 + D2}}
\label{sec:The_proof_of_Inequalities_with_D1+D2}

Let us now show that in the case $q<p$, the third term on the right-hand side in Conjecture~\ref{conj:conjecture_1} is not needed.
To this end it suffices to prove \eqref{eq:bound_q<p_third_term_unnec} only in the case $q<2$, since the case $p>2$ follows by duality \eqref{eq:duality}. 

\begin{proposition}	\label{prop:third-term-is-not-needed}
	Whenever $1\leq q<p \leq \infty$ and $q<2$, we have
	\begin{equation}	\label{eq:third-term-is-not-needed}
	D_2=\|(A\hadprod A)^T \colon \ell^m_{q^*/2} \to \ell^n_{p^*/2}\|^{1/2} \gtrsim_{p,q} \max_{j\le n}\sqrt{\ln (j+1)} b_j^{\rearr{}} ,
	\end{equation}
	where $b_j= \|(a_{ij})_{i\le m} \|_{2q/(2-q)}$.
\end{proposition}

\begin{proof}
	Since the right-hand side of \eqref{eq:third-term-is-not-needed} does not depend on $p$, and the left-hand side is non-decreasing with $p$, we may consider only the case $1\le q<p\le 2$.
	By permuting the columns of $A$  we may and do assume without loss of generality that the sequence $(b_j)_j$ is non-increasing.

	 Fix $j_0\le n$. Let $r$ be the midpoint of the non-empty interval $(\frac{2-p}p, \frac{2-q}q)$.  Take $x=(x_j)_{j\le n}$ with $x_j=\frac{1}{j^r}$. Since $rp/(2-p)>1$, we have
	\[
		\sum_{j=1}^n x_j^{p/(2-p)} \le \sum_{j=1}^\infty \frac{1}{j^{rp/{(2-p)}}} =C(p,q)<\infty,
	\] 
	so $x\in C'(p,q)B^n_{p/(2-p)}=C'(p,q)B^n_{(p^\ast/2)^\ast}$. Therefore, the inequality $(q^\ast/2)^\ast = q/(2~-~q) \ge~1$ and the facts that $b_j\ge b_{j_0}$ for all $j\le j_0$, and that $r<(2-q)/q$ imply
	\begin{align*}
		D_2^2 & = \sup_{z\in B_{(p^\ast/2)^\ast}^n} \biggl( \sum_{i=1}^m \Bigl( \sum_{j=1}^n a_{ij}^2 z_j\Bigr)^{(q^\ast/2)^\ast} \biggr)^{1/(q^\ast/2)^\ast} 
		\gtrsim_{p,q} \biggl( \sum_{i=1}^m \Bigl( \sum_{j=1}^{j_0} a_{ij}^2 j^{-r}\Bigr)^{q/(2-q)} \biggr)^{(2-q)/q} 
		\\
		& \ge  \Bigl( \sum_{i=1}^m  \sum_{j=1}^{j_0} a_{ij}^{2q/(2-q)} j^{-{rq/(2-q)}} \Bigr)^{(2-q)/q} 
		=   \Bigl(  \sum_{j=1}^{j_0} b_j^{2q/(q-2)} j^{-{rq/(2-q)}} \Bigr)^{(2-q)/q}
		\\
		& \ge b_{j_0}^2j_0^{-r+(2-q)/q } \gtrsim_{p,q} b_{j_0}^2 \ln(j_0+1).
	\end{align*}
	Taking the maximum over all $j_0\le n$ completes the proof.
\end{proof}

Now we turn to the proof of \eqref{eq:bound_Emax_term}. Note that it suffices to prove only the first two-sided inequality in \eqref{eq:bound_Emax_term}, since the second one follows from it by duality \eqref{eq:duality}.

\begin{proposition}
	For all $1\le p, q\le \infty$, we have
\begin{multline}		\label{eq:bound_Emax_term-last-section}
	\|A\hadprod A \colon \ell^n_{p/2} \to \ell^m_{q/2} \|^{1/2}+ \EE \max_{i \le m,j\le n} |a_{ij}g_{ij}| 
	\\
	\asymp_{q} \|A\hadprod A \colon \ell^n_{p/2} \to \ell^m_{q/2} \|^{1/2}+\max_{i\le m, j\le n}\sqrt{\ln (j+1)} a_{ij}',
\end{multline}
where the matrix $(a_{ij}')_{i,j}$  is obtained by permuting the columns of the matrix $(|a_{ij}|)_{i,j}$ in such a way that $\max_i a_{i1}'\ge \dots \ge \max_i a_{in}'$.
\end{proposition}

\begin{proof}
	 By permuting the columns of the matrix $A$, we can  assume that the sequence $(\max_{i\le m} |a_{ij}|)_{j=1}^n$ is non-increasing. 
	 We have
\begin{multline}	\label{eq:E-max-0}
    \EE\max_{i\le m,j\le n} |a_{ij} g_{ij}| \le \EE \max_{j\le n} \Big(\max_{i\le m} |a_{ij} g_{ij}| - \EE \max_{i\le m} |a_{ij} g_{ij}|\Big) \\ + \max_{j\le n} \EE \max_{i\le m}|a_{ij}g_{ij}|.
\end{multline}

The function $y \mapsto \max_{i\le m} |a_{ij} y_i|$ is $\max_{i\le m} |a_{ij}|$-Lipschitz with respect to the Euclidean norm on $\RR^m$, so by Gaussian concentration (see, e.g., \cite[Chapter~5.1]{Ledoux}),
\[
	\PP\bigl(\max_{i\le m} |a_{ij} g_{ij}| - \EE \max_{i\le m} |a_{ij} g_{ij}| \ge t \bigr) \le \exp \Bigl(-\frac{t^2} {2\max_{i\le m} |a_{ij}|}\Bigr)
\]
for all  $t\ge 0$, $j\leq n$.
Thus, Lemma~\ref{lem:vH2017-3} and inequality \eqref{eq:E-max-0} imply
\begin{align}\label{eq:E-max-1}
  \EE\max_{i\le m,j\le n} |a_{ij} g_{ij}|  \lesssim \max_{j\le n} \Big(\sqrt{\ln(j+1)}\max_{i\le m} |a_{ij}|\Big)+
\max_{j\le n} \EE \max_{i\le m}|a_{ij}g_{ij}|.
\end{align}
We have
\begin{align*}
  \max_{j\le n} \EE \max_{i\le m}|a_{ij}g_{ij}|
  & \le \max_{j\le n} \EE  \Big(\sum_{i=1}^m |a_{ij} g_{ij}|^q\Big)^{1/q} \le \gamma_q  \max_{j\le n} \|(a_{ij})_{i}\|_q \\
 & = \gamma_q  \max_{j\le n} \|(a_{ij}^2)_{i}\|_{q/2}^{1/2} \le
\gamma_q \|A\hadprod A \colon \ell^n_{p/2} \to \ell^m_{q/2} \|^{1/2},
\end{align*}
which, together with \eqref{eq:E-max-1}, provides the asserted upper bound.

On the other hand, if $(a_{l}^{\rearr{}})_{l\le mn}$ denotes the non-increasing rearrangement of the sequence of all absolute values of entries of $A$, then   Lemma~\ref{lem:vH2017-4} implies 
\begin{align*}		\nonumber
	\EE\max_{j\le n} \max_{i\le m}|a_{ij} g_{ij}| 
	\gtrsim \max_{l\le mn} \sqrt{\ln (l+1)}a_l^{\rearr{}} 
	&\ge \max_{j\le n} \sqrt{\ln(j+1)} a_{j}^{\rearr{}} 
	\\
	& \ge \max_{j\le n} \Big(\sqrt{\ln(j+1)}\max_{i\le m} a_{ij}'\Big),
\end{align*}
which provides the asserted lower bound.
\end{proof}

Note that the above proof  shows in fact that
\begin{equation*}	
	 \max_{j\le n} \|(a_{ij})_{i}\|_q + \EE \max_{i \le m,j\le n} |a_{ij}g_{ij}| 
	\\
	\asymp_{q} \max_{j\le n} \|(a_{ij})_{i}\|_q +\max_{i\le m, j\le n}\sqrt{\ln (j+1)} a_{ij}',
\end{equation*}
so
\begin{multline}	\label{eq:Emax-random-vs-max-deterministic}
	 \max_{j\le n} \|(a_{ij})_{i}\|_q + \max_{i\le m} \|(a_{ij})_{j}\|_{p^\ast} + \max_{j\le n, i \le m} \sqrt{\ln (i+1)} a_{ij}''	\\
	\asymp_{q} \max_{j\le n} \|(a_{ij})_{i}\|_q + \max_{i\le m} \|(a_{ij})_{j}\|_{p^\ast} +\max_{i\le m, j\le n}\sqrt{\ln (j+1)} a_{ij}',
	\end{multline}
	where the matrix $(a_{ij}'')_{i,j}$  is obtained by permuting the rows of the matrix $(|a_{ij}|)_{i,j}$ in such a way that $\max_j a_{1j}''\ge \dots \ge \max_j a_{mj}''$.

\subsection{Counterexample to a seemingly natural conjecture}
\label{subsection:counterexample}
In this subsection we provide an example showing that for any $p\le q <2$ the bound 
\begin{align}
\EE \|G_A\colon \ell_{p}^n\to\ell_q^m\| \lesssim_{p,q}  \|A\hadprod A \colon \ell^n_{p/2} \to \ell^m_{q/2} \|^{1/2} &+ \|(A\hadprod A)^T \colon \ell^m_{q^*/2} \to \ell^n_{p^*/2}\|^{1/2} \nonumber\\
&+ \EE\max_{i\leq m, j\leq n}|a_{ij}g_{ij}|.
\label{eq:upper-bound-fails}
\end{align}
cannot hold. By duality \eqref{eq:duality}, it also cannot hold for any $2<p\le q$. This explains that   Conjecture~\ref{conj:conjecture_1} cannot be simplified into a form like on the right-hand side of~\eqref{eq:lower-bound}.

Let  $p\le q <2$, $k, N\in \NN$, and let $A_1, \ldots, A_N$ be $k\times k$ matrices with all entries equal to one. Consider a block matrix
\[
A = 
\begin{pmatrix}
\begin{matrix}
A_1 &     \\
& A_2
\end{matrix}
&       0        \\
0           &  \begin{matrix}
\ddots &     \\
& A_N
\end{matrix}       
\end{pmatrix}
\]
of size $kN \times kN$, with  blocks $A_1, \ldots A_N$ on the diagonal and with all other entries equal to $0$.

Note that since $p\le q \le 2$,
\begin{align*}
	\|A\hadprod A \colon \ell^{kN}_{p/2} \to \ell^{kN}_{q/2} \| 
	&= \max_{l\le N} \|A_l\hadprod A_l \colon \ell^{k}_{p/2} \to \ell^{k}_{q/2} \| 
	= \|A_1\hadprod A_1 \colon \ell^{k}_{p/2} \to \ell^{k}_{q/2} \|
	\\
	& =\sup_{x\in B_{p/2}^k}  \Bigl( \sum_{i=1}^k \Bigl| \sum_{j=1}^k x_i \Bigr|^{q/2} \Bigr)^{2/q} 
	=\sup_{x\in B_{p/2}^k} k^{2/q} \Bigl| \sum_{i=1}^k x_i \Bigr| = k^{2/q},
\end{align*}
and similarly, since $2\le q^\ast\le p^{\ast}$,
\begin{equation*}
	\|(A\hadprod A)^T \colon \ell^{kN}_{q^\ast/2} \to \ell^{kN}_{p^\ast/2} \| 
	=  \|(A_1\hadprod A_1)^T \colon \ell^{k}_{q^\ast/2} \to \ell^{k}_{p^\ast/2} \| 
	= k^{2/p^\ast+1-2/q^\ast}.
\end{equation*}
The two bounds above and Lemma~\ref{lem:max-Gaussians} imply that the right-hand side of \eqref{eq:upper-bound-fails} is bounded from above by 
\begin{equation}	\label{eq:upper-bound-fails2}
	C \Bigl( k^{1/q}+k^{1/p^\ast+1/2-1/q^\ast} + \sqrt{\ln (kN)} \Bigr).
\end{equation}

On the other hand, since for all $j\le kN$, $\|(a_{ij})_{i} \|_{2q/(2-q)} = k^{(2-q)/(2q)}$, we obtain from the lower bound \eqref{eq:hippo-lower} that 
\begin{equation}	\label{eq:upper-bound-fails3}
	\EE \|G_A\colon \ell_p^{kN}\to \ell_q^{kN}\| \gtrsim \sqrt{\ln (kN)}  k^{(2-q)/(2q)}.
\end{equation}
If we take $N\asymp e^{e^k}$, then \eqref{eq:upper-bound-fails3} is of larger order than \eqref{eq:upper-bound-fails2} as $k\to \infty$, so \eqref{eq:upper-bound-fails} cannot hold.

\subsection{Discussion of another natural conjecture} \label{subsect:conj-wrong}

In this subsection we prove all the assertions of Remark~\ref{rmk:conj-wrong}. 
We begin by showing that for every $1\le p \le 2\le q \le \infty$, 
\begin{equation}	\label{eq:equiv-two-conj}
	D_1+D_2+ \EE \max_{i,j}|a_{ij}g_{ij}| \asymp_{p,q} \EE\max_{i\le m} \| (a_{ij}g_{ij})_j\|_{p^\ast} +\EE\max_{j\le n}  \| (a_{ij}g_{ij})_i\|_q,
\end{equation}
and,  in the case  $p,q\ge 2$,
\begin{multline} 	\label{eq:upper-exp-max-random}
	\EE\max_{i\le m} \| (a_{ij}g_{ij})_j\|_{p^\ast} +\EE\max_{j\le n}  \| (a_{ij}g_{ij})_i\|_q 
		\\
		\lesssim_{p,q}
		\max _{i\le m} \| (a_{ij})_j\|_{p^\ast} +\max_{j\le n}  \| (a_{ij})_i\|_q +\max_{i\le m} \sqrt{\ln(i+1)}d_i^{\rearr{}},
\end{multline}
where $ D_1  = \|A\hadprod A \colon \ell^n_{p/2} \to \ell^m_{q/2}\|^{1/2}$, $D_2  = \|(A\hadprod A)^T \colon \ell^m_{q^*/2} \to \ell^n_{p^*/2}\|^{1/2}$, and $d_i = \|(a_{ij})_{j\le n} \|_{2p/(p-2)}$. 
In other words, \eqref{eq:equiv-two-conj} shows that Conjecture~\ref{conj:conjecture_1} is equivalent to \eqref{eq:conj-wrong} as long as $1\le p \le 2\le q \le \infty$.

\begin{proof}[Proof of \eqref{eq:equiv-two-conj} and \eqref{eq:upper-exp-max-random}]
	Fix $i\le m$ and let $f(x)=\|(a_{ij}x_j)_j\|_{p^\ast}$ for $x\in \RR^n$. For $p\ge 2$ we have $p^\ast (2/p^\ast)^\ast = 2p/(p-2)$. Thus $f$ is Lipschitz continuous with constant $L_i$ equal to
	\[
			\sup_{x\in B_2^n} \Bigl(\sum_{j=1}^n |a_{ij}x_{j}|^{p^\ast}\Bigr)^{1/p^\ast} 
		\hspace{-3pt}  = \hspace{-3pt} \sup_{y\in B_{2/p^\ast}^n} \Bigl(\sum_{j=1}^n |a_{ij}|^{p^\ast} y_{j} \Bigr)^{1/p^\ast}
		\hspace{-3pt}  =\begin{cases}
 \max_{j\le n} |a_{ij}| &  \text{if }\  p\leq 2,\\
   \|(a_{ij})_j\|_{{2p/(p-2)}} &   \text{if } \ p\ge 2.
\end{cases} 
	\]
	Therefore, the Gaussian concentration inequality (see, e.g., \cite[Chapter~5.1]{Ledoux}) implies that for every $t\ge 0$ and every $i\le m$,
	\[
		\PP \Bigl(\|(a_{ij}g_{ij})_j\|_{p^\ast}  - \EE \|(a_{ij}g_{ij})_j\|_{p^\ast}  \ge t \Bigr) \le e^{-t^2/2L_i^2},
	\]
	so by Lemma~\ref{lem:vH2017-3} we get
	\begin{multline}\label{eq:help-5.10-1}
	\EE \max_{i\le m}\Bigl(\|(a_{ij}g_{ij})_j\|_{p^\ast}   - \EE \|(a_{ij}g_{ij})_j\|_{p^\ast} \Bigr) 
	\\
	\lesssim 
	\begin{cases}
\max_{i\le m} \max_{j\le n}   \sqrt{\ln(i+1)} a_{ij}'' &  \text{if }\  p\leq 2,\\
  \max_{i\le m} \sqrt{\ln(i+1)}d_i^{\rearr{}} &   \text{if } \ p\ge 2,
   \end{cases} 
	\end{multline}
	where  the matrix  $(a_{ij}'')_{i,j}$ is obtained by permuting the rows of the matrix $(|a_{ij}|)_{i,j}$ in such a way that $\max_j a_{1j}'' \ge \dots \ge \max_j a_{mj}''$. 
 
 	Moreover, by Jensen's inequality,
	\[
		\EE  \|(a_{ij}g_{ij})_j\|_{p^\ast} \le \bigl(\EE  \|(a_{ij}g_{ij})_j\|_{p^\ast}^{p^\ast} \bigr)^{1/p^\ast} =  \Bigl(\EE \sum_{j=1}^n |a_{ij}g_{ij}|^{p^\ast}\Bigr)^{1/p^\ast} =\gamma_{p^\ast}  \|(a_{ij})_j\|_{p^\ast}.
	\]
	This together with the triangle inequality and  \eqref{eq:help-5.10-1} implies
	\[
		\EE \max_{i\le m}  \|(a_{ij}g_{ij})_j\|_{p^\ast} \lesssim_{p} \max_{i\le m}\|(a_{ij})_j\|_{p^\ast} + \begin{cases}
\max_{i\le m} \max_{j\le n}   \sqrt{\ln(i+1)} a_{ij}'' &  \text{if }\  p\leq 2,\\
  \max_{i\le m} \sqrt{\ln(i+1)}d_i^{\rearr{}} &   \text{if } \ p\ge 2,
   \end{cases} 
	\]
	and, by duality, 
	\[
		\EE \max_{j\le n}  \|(a_{ij}g_{ij})_i\|_{q} \lesssim_{q} \max_{j\le n}\|(a_{ij})_i\|_{q} + \begin{cases}
 \max_{j\le n}\max_{i\le m}   \sqrt{\ln(j+1)} a_{ij}' &  \text{if }\  q\ge 2,\\
  \max_{j\le n} \sqrt{\ln(j+1)}b_j^{\rearr{}} &   \text{if } \ q\le 2,
   \end{cases} 
	\]
	where $b_j=\|(a_{ij})_i)\|_{2q/(2-q)}$, and the matrix $(a_{ij}')_{i,j}$ is obtained by permuting the columns of the matrix $(|a_{ij}|)_{i,j}$ in such a way that $\max_i a_{i1}'\ge \dots \ge \max_i a_{in}'$.
	This, together with Lemma~\ref{lem:special-norms} and \eqref{eq:bound_Emax_term-last-section}  yields in the case $p\le 2\le q$,
	\[
		\EE \max_{i\le m}  \|(a_{ij}g_{ij})_j\|_{p^\ast}  + \EE \max_{j\le n}  \|(a_{ij}g_{ij})_i\|_{q} \lesssim_{p,q} D_1+D_2 + \EE\max_{i,j} |a_{ij}g_{ij}|,
	\]
	what implies the lower bound of \eqref{eq:equiv-two-conj}.
	 In the case $2<p, q$ we additionally use \eqref{eq:Emax-random-vs-max-deterministic} and the simple observation that 
	 \[
	 	\max_{i\le m} \max_{j\le n}   \sqrt{\ln(i+1)} a_{ij}''  \le   \max_{i\le m} \sqrt{\ln(i+1)}d_i^{\rearr{}} 	\]
	to get \eqref{eq:upper-exp-max-random}.

	Now we move to the proof of the upper bound of \eqref{eq:equiv-two-conj} in the case $p\le 2\le q$. Since the $\ell_{p^\ast}^n$ norm is unconditional, we have by Jensen's inequality and Lemma~\ref{lem:special-norms}
	\begin{align*}
		\EE  \max_{i\le m}  \|(a_{ij}g_{ij})_j\|_{p^\ast}  
		=\EE  \max_{i\le m}  \|(|a_{ij}g_{ij}|)_j\|_{p^\ast}
		 &\ge   \max_{i\le m}   \|(|a_{ij}|\EE|g_{ij}|)_j\|_{p^\ast} 
		\\ &=\sqrt{2/\pi}   \max_{i\le m}   \|(|a_{ij}|)_j\|_{p^\ast} 
		=\sqrt{2/\pi}  D_2 ,
	\end{align*}
	and dually
	\begin{equation*}
		\EE  \max_{j\le n}  \|(a_{ij}g_{ij})_i\|_{q}  
		 \ge  
		 \sqrt{2/\pi}   D_1.
	\end{equation*}
	Moreover, since $\|\cdot\|_q\ge \|\cdot\|_{\infty}$,
	\[
		\EE  \max_{j\le n}  \|(a_{ij}g_{ij})_i\|_{q}  \ge \EE\max_{j} \max_{i} |a_{ij}g_{ij}|,
	\]
	which finishes the proof of the upper bound of \eqref{eq:equiv-two-conj}.
	\end{proof}

Next, for every pair $(p,q)\in [1,\infty]^2$ which does not satisfy the condition $1\le p\le 2\le q\le \infty$ we shall give  examples of $m,n\in \NN$, and  $m\times n$ matrices $A$, for which
\begin{equation} \label{eq:example-two-conj}
	 \EE \|G_A\colon \ell_{p}^n\to\ell_q^m\| \gg \EE\max_{i\le m} \| (a_{ij}g_{ij})_j\|_{p^\ast} +\EE\max_{j\le n}  \| (a_{ij}g_{ij})_i\|_q
\end{equation}
when $m,n \to \infty$. This shows that the natural conjecture \eqref{eq:conj-wrong} is wrong outside the range $1\le p\le 2\le q\le \infty$. 
The case $p=2=q$, when \eqref{eq:conj-wrong} is valid (cf. \eqref{eq:LvHY}), is in a sense a boundary case, for which \eqref{eq:conj-wrong} (i.e., a natural generalization of \eqref{eq:LvHY}) may hold.

\begin{example} [for \eqref{eq:example-two-conj} in the case $q<p$.]
	Let $m=n$, and $A=\operatorname{Id}_n$. 
	Then by Lemmas~\ref{lem:max-Gaussians} and \ref{lem:vH2017-4} we have
	\[
		\EE\max_{i\le m} \| (a_{ij}g_{ij})_j\|_{p^\ast} +\EE\max_{j\le n}  \| (a_{ij}g_{ij})_i\|_q = 2 \max_{i\le n} |g_{ii}| \asymp \sqrt{\ln  n},
	\]
	whereas Proposition~\ref{prop:cm-lower-bound} and our assumption $p/q>1$ imply
	\begin{align*}
		 \EE \|G_A\colon \ell_{p}^n\to\ell_q^n\| &\gtrsim \|\operatorname{Id}_n \colon \ell_{p/2}^n \to \ell_{q /2}^n \|^{1/2} = \sup_{x\in B_{p/2}^n}  \Bigl( \sum_{i=1}^n |x_i|^{q/2} \Bigr)^{1/q}
		\\
		&=   \Bigl(  \sup_{y\in B_{p/q}^n} \sum_{i=1}^n |y_i| \Bigr)^{1/q} = \bigl( n^{1/(p/q)^*} \bigr)^{1/q}  \gg \sqrt{\ln n}.
	\end{align*}
\end{example}

Since cases $2<p \le q$ and $p\le q <2$ are dual (see \eqref{eq:duality}), we give an example for which  \eqref{eq:example-two-conj} holds only in the first case.

\begin{example} [for \eqref{eq:example-two-conj} in the case $2<p \le q$.]
	Fix $p$ and $q$ satisfying $2<p \le q$. Let $m,n \to \infty$ be such that $m^{1/q}\gg n^{1/p^\ast}$, and let $A$ be an $m\times n$ matrix with all entries equal to $1$. 
	For $p>2$ we have $2(p/2)^\ast = 2p/(p-2)$.
	This together with \eqref{eq:upper-exp-max-random} implies
	\begin{align*}
	   \MoveEqLeft[4]
		\EE\max_{i\le m} \| (a_{ij}g_{ij})_j\|_{p^\ast} +\EE\max_{j\le n}  \| (a_{ij}g_{ij})_i\|_q 
		\\
		&\lesssim_{p,q}
		\max _{i\le m} \| (a_{ij})_j\|_{p^\ast} +\max_{j\le n}  \| (a_{ij})_i\|_q +\max_{i\le m} \sqrt{\ln(i+1)}d_i^{\rearr{}}
		\\
		& = n^{1/p^\ast} + m^{1/q} + \sqrt{\ln(m+1)}n^{(p-2)/2p} 
		\lesssim m^{1/q} + \sqrt{\ln m}\, n^{\frac1{2(p/2)^\ast}}.
	\end{align*}
	
	On the other hand, Proposition~\ref{prop:cm-lower-bound} and our assumption $p/2>1$ imply
	\begin{align*}
		 \EE \|G_A\colon \ell_{p}^n\to\ell_q^n\| 
		 &\gtrsim \|A \colon \ell_{p/2}^n \to \ell_{q /2}^n \|^{1/2} 
		 = \sup_{x\in B_{p/2}^n}  \Bigl( \sum_{i=1}^m \Bigl|\sum_{j=1}^n x_j\Bigr|^{q/2} \Bigr)^{1/q}
		\\
		&=  m^{1/q} \sup_{x\in B_{p/2}^n}  \Bigl( \Bigl|\sum_{j=1}^n x_j\Bigr| \Bigr)^{1/2}
		= m^{1/q} n^{\frac1{2(p/2)^\ast}} \gg m^{1/q} + \sqrt{\ln m} \, n^{\frac1{2(p/2)^\ast}}.
	\end{align*}
\end{example}

\subsection{Infinite dimensional Gaussian operators}		\label{subsect:infty-dim}

In this subsection we prove Proposition~\ref{prop:finite-implies-infinite} concerning infinite dimensional Gaussian operators. It allows us to see that Conjecture~\ref{conj:conjecture_1} implies Conjecture~\ref{conj:conjecture_infty}.

\begin{proof}[Proof of Proposition~\ref{prop:finite-implies-infinite}]
We adapt the proof of \cite[Corollary~1.2]{LvHY} to prove Proposition~\ref{prop:finite-implies-infinite} in the case $p\leq 2\leq q$ -- remaining cases may be proven similarly. Fix $1\le p \le 2\le q\le \infty$ for which \eqref{eq:hippo} holds and a deterministic infinite matrix $A=(a_{ij})_{i,j\in \NN}$. Using the monotone convergence theorem one can show that a matrix $B = (b_{ij})_{i,j\in \NN}$ defines a bounded operator between $\ell_p(\NN)$ and $\ell_q(\NN)$ if an only if $\sup_{n\in \NN} \|(b_{ij})_{i,j\le n}\colon \ell_p^n\to \ell_q^n\| < \infty$. Interpreting $\|B\colon \ell_p(\NN)\to \ell_q(\NN)\|$ as infinity for matrices which do not define a bounded operator, we have
\begin{align*}
	\MoveEqLeft[18]\EE \|G_A\colon \ell_p(\NN) \to \ell_q(\NN) \|
	= \EE \sup_{x\in B_p^\infty} \biggl( \sum_{i=1}^\infty \Bigl| \sum_{j=1}^\infty a_{ij}g_{ij}x_j \Bigr|^q \biggr)^{1/q}
	\\ =  \EE \lim_{n\to \infty} \sup_{x\in B_p^n} \biggl( \sum_{i=1}^n \Bigl| \sum_{j=1}^n a_{ij}g_{ij}x_j \Bigr|^q \biggr)^{1/q}
	&= \lim_{n\to \infty}  \EE \sup_{x\in B_p^n} \biggl( \sum_{i=1}^n \Bigl| \sum_{j=1}^n a_{ij}g_{ij}x_j \Bigr|^q \biggr)^{1/q}
	\\
	&= \lim_{n\to \infty} \EE \bigl\| (g_{ij}a_{ij})_{i,j\le n} \colon \ell_p^n \to \ell_q^n \bigr\|
\end{align*}
and similarly
\begin{align*}
	\|A\hadprod A \colon \ell_{p/2}(\NN) \to \ell_{q/2}(\NN) \| &= \lim_{n\to \infty} \|(a_{ij}^2)_{i,j\le n} \colon \ell^n_{p/2} \to \ell^n_{q/2} \|,\\
	\|(A\hadprod A)^T \colon \ell_{q^*/2}(\NN) \to \ell_{p^*/2}(\NN) \| &= \lim_{n\to \infty} \|(a_{ji}^2)_{i,j\le n} \colon \ell^n_{q^*/2} \to \ell^n_{p^*/2} \|,\\
\shortintertext{and}
	\EE\sup_{i, j \in \NN}|a_{ij}g_{ij}| &= \lim_{n\to \infty} \EE\sup_{i, j \le n}|a_{ij}g_{ij}|.
	\end{align*}
	Therefore, \eqref{eq:hippo} implies the following: $\EE \|G_A\colon \ell_p(\NN) \to \ell_q(\NN) \| <\infty $ if and only if $\|A\hadprod A \colon \ell_{p/2}(\NN) \to \ell_{q/2}(\NN) \| <\infty$, $\|(A\hadprod A)^T \colon \ell_{q^*/2}(\NN) \to \ell_{p^*/2}(\NN) \| <\infty$, and $\EE\sup_{i, j \in \NN}|a_{ij}g_{ij}| <\infty$.
	It thus suffices to prove the following claim: $\|G_A\colon \ell_p(\NN) \to \ell_q(\NN) \| <\infty $ almost surely  if and only if $\EE \|G_A\colon \ell_p(\NN) \to \ell_q(\NN) \| <\infty $.
	
	If $\PP(\|G_A\colon \ell_p(\NN) \to \ell_q(\NN) \| <\infty ) <1$, then $\PP(\|G_A\colon \ell_p(\NN) \to \ell_q(\NN) \| =\infty ) >0$, so $\EE \|G_A\colon \ell_p(\NN) \to \ell_q(\NN) \| =\infty $.
	
	 Assume now that $\PP(\|G_A\colon \ell_p(\NN) \to \ell_q(\NN) \| <\infty ) =1$. By \eqref{eq:eps-nets1} and \eqref{eq:eps-nets2} we know that for every $n\in\NN$ there exist finite sets $S_n$ and $T_n$ such that
	 \begin{align*}
 	\|G_A\colon \ell_p(\NN) \to \ell_q(\NN) \|
	&=  \sup_{n\in \NN} \sup_{x\in B_p^n, y\in B_{q^\ast}^n}  \sum_{i=1}^n \sum_{j=1}^n  y_i a_{ij} g_{ij} x_j
		\\ &\asymp	\sup_n \sup_{x\in S_n, y\in T_n}  \sum_{i=1}^n \sum_{j=1}^n  y_i a_{ij} g_{ij} x_j 	\qquad \text{a.s.}
	 \end{align*}
	 In particular, there exist Gaussian random variables $(\Gamma_k)_{k\in \NN}$ such that
	 \[
	 \|G_A\colon \ell_p(\NN) \to \ell_q(\NN) \|
		\asymp
		\sup_{k\in \NN} \Gamma_k \qquad \text{a.s.}
	\]
	  Therefore, we may apply \cite[(1.2)]{LandauShepp} to see that there exists $\varepsilon >0$ such that $\EE \exp(\varepsilon  \|G_A\colon \ell_p(\NN) \to \ell_q(\NN) \|^2 ) < \infty$, so $\EE \|G_A\colon \ell_p(\NN) \to \ell_q(\NN) \| <\infty $, which completes the proof of the claim.
	\end{proof}

\medskip

\subsection*{Acknowledgments}
R.~Adamczak is partially supported by the  National Science Center, Poland via the Sonata Bis  grant  no.\ 2015/18/E/ST1/00214. R.~Adamczak was partially supported by by the WTZ Grant PL 06/2018 of the OeAD.  J.~Prochno and M.~Strzelecka are --- and M.~Strzelecki was --- supported by the Austrian Science Fund (FWF) Project P32405 \textit{Asymptotic Geometric Analysis and Applications}.
M.~Strzelecka was partially supported by the National Science Center, Poland,
via the Maestro grant no.\ 2015/18/A/ST1/00553. 

  \bibliographystyle{amsplain}
  \bibliography{matrices}

\providecommand{\bysame}{\leavevmode\hbox to3em{\hrulefill}\thinspace}
\providecommand{\MR}{\relax\ifhmode\unskip\space\fi MR }
\providecommand{\MRhref}[2]{%
  \href{http://www.ams.org/mathscinet-getitem?mr=#1}{#2}
}
\providecommand{\href}[2]{#2}
\begin{thebibliography}{10}

\bibitem{AM2007}
D.~Achlioptas and F.~Mcsherry, \emph{Fast computation of low-rank matrix
  approximations}, J. ACM \textbf{54} (2007), no.~2, 9–es.

\bibitem{MR2949869}
R.~Adamczak, R.~Lata{\l}a, A.~E. Litvak, A.~Pajor, and N.~Tomczak-Jaegermann,
  \emph{Chevet type inequality and norms of submatrices}, Studia Math.
  \textbf{210} (2012), no.~1, 35--56. \MR{2949869}

\bibitem{MR3478525}
R.~Adamczak, R.~Lata{\l}a, Z.~Pucha{\l}a, and K.~\.{Z}yczkowski,
  \emph{Asymptotic entropic uncertainty relations}, J. Math. Phys. \textbf{57}
  (2016), no.~3, 032204, 24. \MR{3478525}

\bibitem{AC2009}
N.~Ailon and B.~Chazelle, \emph{The fast {J}ohnson-{L}indenstrauss transform
  and approximate nearest neighbors}, SIAM J. Comput. \textbf{39} (2009),
  no.~1, 302--322. \MR{2506527}

\bibitem{ABDF2015}
G.~Akemann, J.~Baik, and P.~Di~Francesco (eds.), \emph{The {O}xford handbook of
  random matrix theory}, Oxford University Press, Oxford, 2015.

\bibitem{AGZ2010}
G.~W. Anderson, A.~Guionnet, and O.~Zeitouni, \emph{An introduction to random
  matrices}, Cambridge Studies in Advanced Mathematics, vol. 118, Cambridge
  University Press, Cambridge, 2010. \MR{2760897}

\bibitem{BvH2016}
A.~S. Bandeira and R.~van Handel, \emph{Sharp nonasymptotic bounds on the norm
  of random matrices with independent entries}, Ann. Probab. \textbf{44}
  (2016), no.~4, 2479--2506. \MR{3531673}

\bibitem{Bennett}
G.~Bennett, \emph{Schur multipliers}, Duke Math. J. \textbf{44} (1977), no.~3,
  603--639. \MR{493490}

\bibitem{BGN}
G.~Bennett, V.~Goodman, and C.~M. Newman, \emph{Norms of random matrices},
  Pacific J. Math. \textbf{59} (1975), no.~2, 359--365. \MR{393085}

\bibitem{BG1981}
Y.~Benyamini and Y.~Gordon, \emph{Random factorization of operators between
  {B}anach spaces}, J. Analyse Math. \textbf{39} (1981), 45--74. \MR{632456}

\bibitem{BLM2013}
S.~Boucheron, G.~Lugosi, and P.~Massart, \emph{Concentration inequalities},
  Oxford University Press, Oxford, 2013, A nonasymptotic theory of
  independence, With a foreword by Michel Ledoux. \MR{3185193}

\bibitem{BDN2015}
J.~Bourgain, S.~Dirksen, and J.~Nelson, \emph{Toward a unified theory of sparse
  dimensionality reduction in {E}uclidean space}, Geom. Funct. Anal.
  \textbf{25} (2015), no.~4, 1009--1088. \MR{3385629}

\bibitem{CarlMaureyPuhl}
B.~Carl, B.~Maurey, and J.~Puhl, \emph{Grenzordnungen von
  absolut-{$(r,\,p)$}-summierenden {O}peratoren}, Math. Nachr. \textbf{82}
  (1978), 205--218. \MR{498116}

\bibitem{CGLP2012}
D.~Chafa\"{\i}, O.~Gu\'{e}don, G.~Lecu\'{e}, and A.~Pajor, \emph{Interactions
  between compressed sensing random matrices and high dimensional geometry},
  Panoramas et Synth\`eses [Panoramas and Syntheses], vol.~37, Soci\'{e}t\'{e}
  Math\'{e}matique de France, Paris, 2012. \MR{3113826}

\bibitem{DS2001}
K.~R. Davidson and S.~J. Szarek, \emph{Local operator theory, random matrices
  and {B}anach spaces}, Handbook of the geometry of {B}anach spaces, {V}ol.
  {I}, North-Holland, Amsterdam, 2001, pp.~317--366. \MR{1863696}

\bibitem{FR2013}
S.~Foucart and H.~Rauhut, \emph{A mathematical introduction to compressive
  sensing}, Applied and Numerical Harmonic Analysis, Birkh\"{a}user/Springer,
  New York, 2013. \MR{3100033}

\bibitem{FY2019}
O.~Friedland and P.~Youssef, \emph{Approximating matrices and convex bodies},
  Int. Math. Res. Not. IMRN (2019), no.~8, 2519--2537. \MR{3942169}

\bibitem{Gl1983}
E.~D. Gluskin, \emph{Norms of random matrices and diameters of
  finite-dimensional sets}, Mat. Sb. (N.S.) \textbf{120(162)} (1983), no.~2,
  180--189, 286. \MR{687610}

\bibitem{GluskinKwapien95}
E.~D. Gluskin and S.~Kwapie\'{n}, \emph{Tail and moment estimates for sums of
  independent random variables with logarithmically concave tails}, Studia
  Math. \textbf{114} (1995), no.~3, 303--309. \MR{1338834}

\bibitem{GvN1951}
H.~H. Goldstine and J.~von Neumann, \emph{Numerical inverting of matrices of
  high order. {II}}, Proc. Amer. Math. Soc. \textbf{2} (1951), 188--202.
  \MR{41539}

\bibitem{G1985}
Y.~Gordon, \emph{Some inequalities for {G}aussian processes and applications},
  Israel J. Math. \textbf{50} (1985), no.~4, 265--289. \MR{800188}

\bibitem{GLSW2002}
Y.~Gordon, A.~E. Litvak, C.~Sch\"{u}tt, and E.~M. Werner, \emph{Geometry of
  spaces between polytopes and related zonotopes}, Bull. Sci. Math.
  \textbf{126} (2002), no.~9, 733--762. \MR{1941083}

\bibitem{GHLP}
O.~Gu\'{e}don, A.~Hinrichs, A.~E. Litvak, and J.~Prochno, \emph{On the
  expectation of operator norms of random matrices}, Geometric aspects of
  functional analysis, Lecture Notes in Math., vol. 2169, Springer, Cham, 2017,
  pp.~151--162. \MR{3645120}

\bibitem{GMPT2008}
O.~Gu\'{e}don, S.~Mendelson, A.~Pajor, and N.~Tomczak-Jaegermann,
  \emph{Majorizing measures and proportional subsets of bounded orthonormal
  systems}, Rev. Mat. Iberoam. \textbf{24} (2008), no.~3, 1075--1095.
  \MR{2490210}

\bibitem{GR2007}
O.~Gu\'{e}don and M.~Rudelson, \emph{{$L_p$}-moments of random vectors via
  majorizing measures}, Adv. Math. \textbf{208} (2007), no.~2, 798--823.
  \MR{2304336}

\bibitem{Haagerup}
U.~Haagerup, \emph{The best constants in the {K}hintchine inequality}, Studia
  Math. \textbf{70} (1981), no.~3, 231--283 (1982). \MR{654838}

\bibitem{HKNPUsurvey}
A.~Hinrichs, D.~Krieg, E.~Novak, J.~Prochno, and M.~Ullrich, \emph{{On the
  power of random information}}, Multivariate Algorithms and Information-Based
  Complexity (F.~J. Hickernell and P.~Kritzer, eds.), De Gruyter,
  Berlin/Boston, 1994, pp.~43--64.

\bibitem{HKNPU2021}
\bysame, \emph{Random sections of ellipsoids and the power of random
  information}, Trans. Amer. Math. Soc. \textbf{374} (2021), no.~12,
  8691--8713. \MR{4337926}

\bibitem{HPS2021}
A.~Hinrichs, J.~Prochno, and M.~Sonnleitner, \emph{Random sections of
  $\ell_p$-ellipsoids, optimal recovery and {G}elfand numbers of diagonal
  operators}, 2021.

\bibitem{HPV2021}
A.~Hinrichs, J.~Prochno, and J.~Vyb\'{\i}ral, \emph{Gelfand numbers of
  embeddings of {S}chatten classes}, Math. Ann. \textbf{380} (2021), no.~3-4,
  1563--1593. \MR{4297193}

\bibitem{HM-SO}
P.~Hitczenko, S.~J. Montgomery-Smith, and K.~Oleszkiewicz, \emph{Moment
  inequalities for sums of certain independent symmetric random variables},
  Studia Math. \textbf{123} (1997), no.~1, 15--42. \MR{1438303}

\bibitem{Hoeffding63}
W.~Hoeffding, \emph{Probability inequalities for sums of bounded random
  variables}, J. Amer. Statist. Assoc. \textbf{58} (1963), 13--30. \MR{144363}

\bibitem{KU2021}
D.~Krieg and M.~Ullrich, \emph{Function values are enough for
  {$L_2$}-approximation}, Found. Comput. Math. \textbf{21} (2021), no.~4,
  1141--1151. \MR{4298242}

\bibitem{Kwapien87}
S.~Kwapie\'{n}, \emph{Decoupling inequalities for polynomial chaos}, Ann.
  Probab. \textbf{15} (1987), no.~3, 1062--1071. \MR{893914}

\bibitem{LandauShepp}
H.~J. Landau and L.~A. Shepp, \emph{On the supremum of a {G}aussian process},
  Sankhy\={a} Ser. A \textbf{32} (1970), 369--378. \MR{286167}

\bibitem{MR1388035}
R.~Lata{\l}a, \emph{Tail and moment estimates for sums of independent random
  vectors with logarithmically concave tails}, Studia Math. \textbf{118}
  (1996), no.~3, 301--304. \MR{1388035}

\bibitem{Latala-Some-estimates}
\bysame, \emph{Some estimates of norms of random matrices}, Proc. Amer. Math.
  Soc. \textbf{133} (2005), no.~5, 1273--1282. \MR{2111932}

\bibitem{LatalaStrzeleckaMat}
R.~Lata\l{}a and M.~Strzelecka, \emph{Comparison of weak and strong moments for
  vectors with independent coordinates}, Mathematika \textbf{64} (2018), no.~1,
  211--229. \MR{3778221}

\bibitem{latala-swiatkowski2021norms}
R.~Lata{\l}a and W.~{\'S}wi{\k{a}}tkowski, \emph{Norms of randomized circulant
  matrices}, Electron. J. Probab. \textbf{27} (2022), Paper No. 80, 23.
  \MR{4441144}

\bibitem{LvHY}
R.~Lata{\l}a, R.~van Handel, and P.~Youssef, \emph{The dimension-free structure
  of nonhomogeneous random matrices}, Invent. Math. \textbf{214} (2018), no.~3,
  1031--1080. \MR{3878726}

\bibitem{Ledoux}
M.~Ledoux, \emph{The concentration of measure phenomenon}, Mathematical Surveys
  and Monographs, vol.~89, American Mathematical Society, Providence, RI, 2001.
  \MR{1849347}

\bibitem{L2007}
\bysame, \emph{Deviation inequalities on largest eigenvalues}, Geometric
  aspects of functional analysis, Lecture Notes in Math., vol. 1910, Springer,
  Berlin, 2007, pp.~167--219. \MR{2349607}

\bibitem{Ledoux-Talagrand}
M.~Ledoux and M.~Talagrand, \emph{Probability in {B}anach spaces}, Ergebnisse
  der Mathematik und ihrer Grenzgebiete (3) [Results in Mathematics and Related
  Areas (3)], vol.~23, Springer-Verlag, Berlin, 1991, Isoperimetry and
  processes. \MR{1102015}

\bibitem{LLR1995}
N.~Linial, E.~London, and Y.~Rabinovich, \emph{The geometry of graphs and some
  of its algorithmic applications}, Combinatorica \textbf{15} (1995), no.~2,
  215--245. \MR{1337355}

\bibitem{Matlak}
D.~Matlak, \emph{Oszacowania norm macierzy losowych}, Master's thesis,
  Uniwersytet Warszawski, 2017.

\bibitem{NRV2014}
A.~Naor, O.~Regev, and T.~Vidick, \emph{Efficient rounding for the
  noncommutative {G}rothendieck inequality}, Theory Comput. \textbf{10} (2014),
  257--295. \MR{3267842}

\bibitem{MR3081910}
Z.~Pucha{\l}a, {\L}.~Rudnicki, and K.~\.{Z}yczkowski, \emph{Majorization
  entropic uncertainty relations}, J. Phys. A \textbf{46} (2013), no.~27,
  272002, 12. \MR{3081910}

\bibitem{Rau2010}
H.~Rauhut, \emph{Compressive sensing and structured random matrices},
  Theoretical foundations and numerical methods for sparse recovery, Radon Ser.
  Comput. Appl. Math., vol.~9, Walter de Gruyter, Berlin, 2010, pp.~1--92.
  \MR{2731597}

\bibitem{Riemer-Schuett}
S.~Riemer and C.~Sch\"{u}tt, \emph{On the expectation of the norm of random
  matrices with non-identically distributed entries}, Electron. J. Probab.
  \textbf{18} (2013), no. 29, 13. \MR{3035757}

\bibitem{RV2007}
M.~Rudelson and R.~Vershynin, \emph{Sampling from large matrices: an approach
  through geometric functional analysis}, J. ACM \textbf{54} (2007), no.~4,
  Art. 21, 19. \MR{2351844}

\bibitem{Seginer}
Y.~Seginer, \emph{The expected norm of random matrices}, Combin. Probab.
  Comput. \textbf{9} (2000), no.~2, 149--166. \MR{1762786}

\bibitem{S1962}
D.~Slepian, \emph{The one-sided barrier problem for {G}aussian noise}, Bell
  System Tech. J. \textbf{41} (1962), 463--501. \MR{133183}

\bibitem{So2011}
A.~M.-C. So, \emph{Moment inequalities for sums of random matrices and their
  applications in optimization}, Math. Program. \textbf{130} (2011), no.~1,
  Ser. A, 125--151. \MR{2853163}

\bibitem{ST2004}
D.~A. Spielman and S.-H. Teng, \emph{Nearly-linear time algorithms for graph
  partitioning, graph sparsification, and solving linear systems}, Proceedings
  of the Thirty-Sixth Annual ACM Symposium on Theory of Computing (New York,
  NY, USA), STOC '04, Association for Computing Machinery, 2004, p.~81–90.

\bibitem{Strzelecka}
M.~Strzelecka, \emph{Estimates of norms of log-concave random matrices with
  dependent entries}, Electron. J. Probab. \textbf{24} (2019), Paper No. 107,
  15. \MR{4017125}

\bibitem{Talagrand96}
M.~Talagrand, \emph{A new look at independence}, Ann. Probab. \textbf{24}
  (1996), no.~1, 1--34. \MR{1387624}

\bibitem{Tr2008_b}
J.~A. Tropp, \emph{Norms of random submatrices and sparse approximation}, C. R.
  Math. Acad. Sci. Paris \textbf{346} (2008), no.~23-24, 1271--1274.
  \MR{2473306}

\bibitem{Tr2008}
\bysame, \emph{On the conditioning of random subdictionaries}, Appl. Comput.
  Harmon. Anal. \textbf{25} (2008), no.~1, 1--24. \MR{2419702}

\bibitem{Tropp12}
\bysame, \emph{User-friendly tail bounds for sums of random matrices},
  Foundations of Computational Mathematics \textbf{12} (2012), no.~4, 389--434.

\bibitem{T2015}
\bysame, \emph{An introduction to matrix concentration inequalities},
  Foundations and Trends® in Machine Learning \textbf{8} (2015), no.~1-2,
  1--230.

\bibitem{vH2017}
R.~van Handel, \emph{On the spectral norm of {G}aussian random matrices},
  Trans. Amer. Math. Soc. \textbf{369} (2017), no.~11, 8161--8178. \MR{3695857}

\bibitem{vH2017_survey}
\bysame, \emph{Structured random matrices}, Convexity and concentration, IMA
  Vol. Math. Appl., vol. 161, Springer, New York, 2017, pp.~107--156.
  \MR{3837269}

\bibitem{V2012}
R.~Vershynin, \emph{Introduction to the non-asymptotic analysis of random
  matrices}, Compressed sensing, Cambridge Univ. Press, Cambridge, 2012,
  pp.~210--268. \MR{2963170}

\bibitem{V2018}
\bysame, \emph{High-dimensional probability}, Cambridge Series in Statistical
  and Probabilistic Mathematics, vol.~47, Cambridge University Press,
  Cambridge, 2018, An introduction with applications in data science, With a
  foreword by Sara van de Geer. \MR{3837109}

\bibitem{vN}
J.~von Neumann and H.~H. Goldstine, \emph{Numerical inverting of matrices of
  high order}, Bull. Amer. Math. Soc. \textbf{53} (1947), no.~11, 1021--1099.

\bibitem{W1955}
E.~P. Wigner, \emph{Characteristic vectors of bordered matrices with infinite
  dimensions}, Ann. of Math. (2) \textbf{62} (1955), 548--564. \MR{77805}

\bibitem{W1957}
\bysame, \emph{Characteristic vectors of bordered matrices with infinite
  dimensions. {II}}, Ann. of Math. (2) \textbf{65} (1957), 203--207. \MR{83848}

\bibitem{W1958}
\bysame, \emph{On the distribution of the roots of certain symmetric matrices},
  Ann. of Math. (2) \textbf{67} (1958), 325--327. \MR{95527}

\bibitem{W1928}
J.~Wishart, \emph{The generalised product moment distribution in samples from a
  normal multivariate population}, Biometrika \textbf{20A} (1928), no.~1/2,
  32--52.

\end{thebibliography}

\end{document}